\documentclass[letterpaper, 10pt]{article}

\usepackage[dvipsnames]{xcolor}

\usepackage{
  amsmath, amsthm, amssymb, mathtools, dsfont,          
  graphicx, wrapfig, subcaption, float,                            
  listings, color, inconsolata, pythonhighlight,               
  fancyhdr, sectsty, hyperref, enumerate, enumitem, framed }   

\usepackage{hyperref, lipsum} 
\usepackage{siunitx}

\usepackage{cleveref}
\usepackage{multicol}
\usepackage{multirow}
\usepackage{makecell}
\usepackage{url}
\usepackage{subcaption}

\usepackage{newpxtext, newpxmath, inconsolata}
\usepackage{amsfonts}
\usepackage{pgfplots}
\pgfplotsset{compat=1.12}
\usepackage{tkz-fct}
\usepackage{tikz,forest}
\usetikzlibrary{patterns}
\usetikzlibrary{arrows.meta}
\useforestlibrary{edges}

\tikzset{
decision/.style={diamond, inner sep=1pt},
chance/.style={circle, minimum width=10pt, draw=blue, fill=none, thick, inner sep=0pt},
}

\usepackage{tikz-cd}
\usepackage{enumitem}
\usepackage[title]{appendix}
\usepackage{booktabs}
\usepackage{dynkin-diagrams}

\usepackage{geometry}

\usepackage[bottom]{footmisc}
\usepackage{mathtools}

\allowdisplaybreaks

\usepackage[font={it,footnotesize}]{caption}

\hypersetup{colorlinks=true, linkcolor=RoyalBlue, citecolor=RedOrange, urlcolor=ForestGreen}

\usepackage{titlesec}
\titleformat{\section}{\large\bfseries}{\thesection\;\;\;}{0em}{}
\titleformat{\subsection}{\normalsize\bfseries\selectfont}{\thesubsection\;\;\;}{0em}{}

\setlist[itemize]{wide=0pt, leftmargin=16pt, labelwidth=10pt, align=left}

\usepackage[subfigure]{tocloft}

\cftsetindents{section}{0em}{0em}

\makeatletter
\renewcommand{\cftsecpresnum}{\begin{lrbox}{\@tempboxa}}
\renewcommand{\cftsecaftersnum}{\end{lrbox}}
\makeatother

\graphicspath{{Images/}{../Images/}}

\usepackage[mathscr]{euscript}

\newcommand{\rN}{\mathrm{N}}

\newcommand{\Gal}{\mathrm{Gal}}

\newcommand{\Mod}[1]{\ (\mathrm{mod}\ #1)}

\newcommand{\bF}{\mathbb{F}}

\newcommand{\bQ}{\mathbb{Q}}

\newcommand{\bZ}{\mathbb{Z}}

\newcommand{\cO}{\mathcal{O}}

\newcommand{\fra}{\mathfrak{a}}

\newcommand{\frp}{\mathfrak{p}}

\theoremstyle{definition}
\newtheorem{theorem}{Theorem}[section]
\newtheorem{proposition}[theorem]{Proposition}
\newtheorem{lemma}[theorem]{Lemma}
\newtheorem{corollary}[theorem]{Corollary}

\newtheorem{example}[theorem]{Example}

\newtheorem*{theorem*}{Theorem}
\newtheorem*{proposition*}{Proposition}
\newtheorem*{lemma*}{Lemma}
\newtheorem*{corollary*}{Corollary}
\newtheorem*{definition*}{Definition}
\newtheorem*{example*}{Example}
\newtheorem*{remark*}{Remark}

\theoremstyle{remark}
\newtheorem{remark}[theorem]{Remark}

\pgfkeys{/Dynkin diagram,
edge length=1.5cm,
fold radius=1cm,
indefinite edge/.style={
draw=black,
fill=white,
thin,
densely dashed}}

\usepackage{fancyhdr}
\newcommand{\meanstd}[2]{#1 \text{\tiny $\pm$ #2}}

\usepackage{pgfplots}
\pgfplotsset{compat=1.18}
\usepgfplotslibrary{colorbrewer}
\usepackage{xparse}

\ExplSyntaxOn
\fp_new:N \g_cm_max_fp
\seq_new:N \l_cm_entry_seq

\cs_new_protected:Npn \cm_set_max_from_entries:n #1
  {
    \fp_gzero:N \g_cm_max_fp
    \clist_map_inline:nn {#1}
      {
        \seq_set_split:Nnn \l_cm_entry_seq {/} {##1}
        \fp_compare:nNnT
          { \seq_item:Nn \l_cm_entry_seq {3} } > { \g_cm_max_fp }
          { \fp_gset:Nn \g_cm_max_fp { \seq_item:Nn \l_cm_entry_seq {3} } }
      }
    \xdef\CMmax{\fp_to_decimal:N \g_cm_max_fp}
  }

\cs_new_protected:Npn \CMSetMaxFromEntries #1
  { \cm_set_max_from_entries:n {#1} }

\cs_new:Npn \CMIfDarkTF #1#2#3
  {
    \fp_compare:nTF { #1 > 0.55 * \g_cm_max_fp }
      {#2}
      {#3}
  }
\ExplSyntaxOff

\makeatletter
\NewDocumentCommand{\ConfusionMatrix}{ O{} m m m }{%

  \CMSetMaxFromEntries{#4}%

  \def\CM@coords{}%
  \def\CM@lightcoords{}%
  \def\CM@darkcoords{}%
  \foreach \CMx/\CMy/\CMval in {#4}{%
    \xdef\CM@coords{\CM@coords (\CMx,\CMy)[\CMval]}%
    \CMIfDarkTF{\CMval}{%
      \xdef\CM@darkcoords{\CM@darkcoords (\CMx,\CMy)[\CMval]}%
    }{%
      \xdef\CM@lightcoords{\CM@lightcoords (\CMx,\CMy)[\CMval]}%
    }%
  }%

  \begin{tikzpicture}
    \begin{axis}[
      width=9cm,
      height=7.5cm,
      scale only axis,
      xmin=0.5, xmax=#2+0.5,
      ymin=0.5, ymax=#2+0.5,
      y dir=reverse,
      xtick={1,...,#2},
      ytick={1,...,#2},
      xticklabels={#3},
      yticklabels={#3},
      xlabel={Predicted label},
      ylabel={True label},
      enlargelimits=false,
      axis on top,
      colormap/Blues-9,
      point meta min=0,
      point meta max=\CMmax,
      colorbar,
      colorbar style={
        width=0.38cm,
        scaled y ticks=false,
        yticklabel style={
          /pgf/number format/fixed,
          /pgf/number format/precision=0,
          /pgf/number format/1000 sep={}
        }
      },
      xlabel style={font=\sffamily\Large},
      ylabel style={font=\sffamily\Large},
      tick label style={font=\sffamily\Large},
      #1
    ]
      \addplot[
        matrix plot*,
        mesh/cols=#2,
        point meta=explicit,
      ] coordinates {\CM@coords};

    \addplot[
      only marks,
      mark=none,
      point meta=explicit,
      nodes near coords={
        \pgfmathprintnumber[
          fixed,
          fixed zerofill,
          precision=2,
          1000 sep={}
        ]{\pgfplotspointmeta}
      },
      every node near coord/.append style={
        font=\sffamily\fontsize{18}{18}\selectfont,
        text=blue!80!black,
        yshift=-6pt,
      },
    ] coordinates {\CM@lightcoords};
    
    \addplot[
      only marks,
      mark=none,
      point meta=explicit,
      nodes near coords={
        \pgfmathprintnumber[
          fixed,
          fixed zerofill,
          precision=2,
          1000 sep={}
        ]{\pgfplotspointmeta}
      },
      every node near coord/.append style={
        font=\sffamily\fontsize{18}{18}\selectfont,
        text=white,
        yshift=-6pt,
      },
    ] coordinates {\CM@darkcoords};
    \end{axis}
  \end{tikzpicture}%
}
\makeatother

\NewDocumentCommand{\BinaryConfusionMatrix}{ O{} m m m m m m }{%
  \ConfusionMatrix[#1]{2}{#2,#3}{1/1/#4,2/1/#5,1/2/#6,2/2/#7}%
}

\begin{document}


\title{Machines Learn Number Fields, But How? \\ The Case of Galois Groups}


\author{Kyu-Hwan Lee and Seewoo Lee}
\date{}


\maketitle


\begin{abstract}
By applying interpretable machine learning methods such as decision trees, we study how simple models can classify the Galois groups of \emph{Galois extensions} over $\mathbb{Q}$ of degrees \emph{4, 6, 8, 9, and 10}, using Dedekind zeta coefficients. Our interpretation of the machine learning results allows us to understand how the distribution of zeta coefficients depends on the Galois group, and to prove new criteria for classifying the Galois groups of these extensions. Combined with previous results, this work provides another example of a new paradigm in mathematical research driven by machine learning.
\end{abstract}


\section{Introduction}

In recent years, machine learning (ML) has begun to influence the landscape of pure mathematics, echoing its transformative impact on the natural sciences. Around 2017, experimental efforts emerged to apply ML techniques to mathematical problems, revealing that machines can often detect structural patterns in mathematical data. Early examples arose in algebraic geometry \cite{HE2017564,Carifio2017,Ruehle2017,Krefl2017}, followed by applications in number theory \cite{he2022machine, HLOa, HLOc, HLOP}, representation theory \cite{davies2021advancing, Kr}, and knot theory \cite{davies2021advancing}. 

Several of these studies reported remarkable empirical successes, achieving high accuracy in classifying mathematical objects according to their invariants. Interpreting what the machine learns can sometimes lead to new discoveries or perspectives on the underlying mathematical structures. The recent discovery of {\em murmuration} phenomena in arithmetic is a prototypical example \cite{HLOP, quanta}. Together, these developments suggest a new paradigm for mathematical research \cite{davies2021advancing, douglaslee_mds}: generate datasets, apply ML tools, interpret the results, gain new perspectives, formulate conjectures, and prove them.

The purpose of this paper is to complete a full cycle of this paradigm in the study of Galois groups. In earlier work by the first-named author with He and Oliver \cite{he2022machine}, machine learning algorithms were applied to the classification of number fields by invariants such as Galois groups, class numbers, and unit ranks, using input features derived from the coefficients of defining polynomials or Dedekind zeta functions. While many of these classification tasks achieved remarkably high accuracy, the mathematical rationale behind the machine's decision-making remained opaque. One notable case was the class numbers of real quadratic fields, where subsequent efforts to interpret the ML results revealed that the algorithms were, in effect, rediscovering classical results from Gauss's genus theory \cite{amir2023machine}.

In this paper, we focus on Galois groups as the next case study. We apply interpretable machine learning methods, particularly decision trees, to the problem of predicting Galois groups of number fields. By analyzing the structure of the resulting models, we are able to reverse-engineer the machine's logic and extract precise mathematical statements which we formulate into new conjectures about the determination of Galois groups. Unlike the case of class numbers of real quadratic fields and many earlier examples, these new conjectures are mathematically significant, and we prove them by traditional means.  
As an example, we prove the following  proposition, motivated by an analysis of the classification logic of a decision tree model (see Example \ref{ex:C9_1000} for the proof).
\begin{proposition*}
    Let $K / \bQ$ be a nonic Galois extension, and let $a_n(K)$ be the $n$-th coefficient of the Dedekind zeta function $\zeta_K(s)$.
     If $a_{1000}(K) \le 4$, then $K / \bQ$ is a cyclic extension.
\end{proposition*}
We obtain many other similar and more general results (see Corollaries \ref{cor:ellsqzeta}, \ref{cor:sextic_zc}, \ref{cor:decic_zeta}, \ref{cor:octic_zeta_ab} and \ref{cor:octic_ab_nab}).
Although Galois groups can be determined in various ways, our results offer very simple criteria that could be formulated only after analyzing the ML results. Our work demonstrates how machine learning can serve not merely as a black-box classifier, but as a tool for generating conjectures and aiding in the discovery of rigorous mathematical theorems.
 We anticipate that this approach will continue to be adopted and will lead to many fruitful discoveries in the future.

After this introduction, the next section presents preliminaries on the Dedekind zeta function and the ramification of local fields. Section \ref{sec:galois} contains the main results, along with a detailed explanation of the ML experiments. The corresponding code is available in the GitHub repository \href{https://github.com/seewoo5/ML-NF}{https://github.com/seewoo5/ML-NF}, which can be used to reproduce all the experiments in this paper.
The final section offers concluding remarks.

\paragraph{Acknowledgments}

We thank Jordan Ellenberg, Tony Feng, Yang-Hui He, Hyukpyo Hong and Sug Woo Shin for helpful discussions.
We are also grateful to John Jones for answering questions on local fields in \texttt{LMFDB}, and to David Roe for helping us  use the \texttt{lmfdb-lite} library.
We also thank ChatGPT for assisting us in writing  \LaTeX \ scripts for the decision tree figures and  in searching the literature.


\section{Preliminaries}

\subsection{Dedekind zeta function}
\label{subsec:prelim_dedekindzeta}

Let $K / \bQ$ be a number field and $\cO_K$ be the ring of integer of $K$.
For each nonzero ideal $\fra \subset \cO_K$, let $\rN \fra \coloneqq [\cO_K : \fra]$ be its ideal norm.
The Dedekind zeta function $\zeta_K(s)$ is defined by 
\begin{equation}
    \zeta_K(s) \coloneqq \sum_{0 \ne \fra \subset \cO_K} \frac{1}{\rN\fra^s} = \sum_{n \ge 1} \frac{a_n(K)}{n^s},
\end{equation}
where $a_n(K) = \#\{\fra \subset \cO: \rN \fra = n\}$ counts the number of ideals with norm $n$,  which is always a nonnegative integer.
It admits an Euler product
\[
\zeta_K(s) = \prod_{\frp} (1 - \rN\frp^{-s})^{-1} = \prod_{p} \prod_{\frp|p} (1 - \rN\frp^{-s})^{-1},
\]
where the first product is over the prime ideals of $\cO_K$, and the $p$-th Euler factor for a rational prime $p$ is defined by
\begin{equation}
    \zeta_{K, p}(s) \coloneqq \prod_{\frp|p} (1 - \rN \frp^{-s})^{-1}.
\end{equation}
The zeta coefficients are multiplicative; that is, we have $a_{mn}(K) = a_{m}(K)a_{n}(K)$ for any coprime $m, n$. Moreover, the Euler factor $\zeta_{K, p}(s)$ is completely determined by the decomposition behavior of $p$ in $K$, as stated in the following standard proposition, whose proof we omit.

\begin{proposition}
    \label{prop:dedekind_euler}
    Let $K / \bQ$ be a number field and $p$ be a rational prime. Assume that the prime $p$ factors in $K$ as $(p) = \frp_1^{e_1} \cdots \frp_{g}^{e_g}$, and let $f_i = f(\frp_i | p)$ be the inertia degree of $\frp_i$ over $p$, so that $\rN \frp_i = p^{f_i}$.
    Then the Euler factor can be written as
    \begin{equation}
        \zeta_{K, p}(s) = \prod_{i=1}^{g} (1 - p^{-f_i s})^{-1}.
    \end{equation}
    In particular, when $K / \bQ$ is Galois, all $e_i$'s and $f_i$'s are equal for $i=1,2, \dots , g$, and
    \begin{equation}
    \label{eqn:localzeta}
        \zeta_{K, p}(s) = (1 - p^{-fs})^{-g} = \sum_{k \ge 0} \binom{g + k - 1}{g - 1} p^{-fks}.
    \end{equation}
\end{proposition}
\begin{corollary}
    \label{cor:prime_pow_zeta_coeff}
   Let $K / \bQ$ be a Galois extension and $p$ be a rational prime.
   Assume that $p$ decomposes in $K$ as $(p) = (\frp_1 \cdots \frp_g)^{e}$ with $\rN \frp_i =p^f$ for all $1 \le i \le g$.
   Then the $p^a$-th zeta coefficient for $a \in \mathbb Z_{\ge 0}$ is given by
   \begin{equation}
   \label{eqn:dedekindzeta_primepow}
       a_{p^{a}}(K) = \begin{cases} \binom{g + \frac{a}{f} - 1}{g-1} & \text{ if } f \mid a, \\ 0 & \text{ otherwise}.\end{cases}
   \end{equation}
\end{corollary}

\begin{proof}
    If $f \nmid a$, then $a_{p^a}(K) = 0$ since $p^{-as}$ does not appear in \eqref{eqn:localzeta}.
    If $f \mid a$, take $k = \frac{a}{f}$.
\end{proof}

\subsection{Ramification of local fields}
\label{subsec:prelim_tameram}

We review ramification theory of local fields, which will be useful for studying decomposition behavior of primes.
Let $L / K$ be a finite Galois extension of local fields with $G = \Gal(L / K)$.
Let $\cO_K, \pi_K, v_K, k$ be the ring of integer, uniformizer, valuation, and residue field of $K$, and similarly $\cO_L, \pi_L, v_L, \ell$.
For $i \ge -1$, define the $i$-th ramification group to be
\begin{equation}
    G_i \coloneqq \{\sigma \in G: v_L(\sigma(x) - x) \ge i + 1 \,\,\forall x \in \cO_L\}.
\end{equation}
Then $\{ G_i \}_{i \ge -1}$ form a decreasing filtration of $G$.
The group $G_0$ (resp. $G_1$) is often called the inertia subgroup (resp. wild inertia subgroup), and we denote their fixed fields as $L^u = L^{G_0}$ and $L^t = L^{G_1}$.
We have $G_0 / G_1 \hookrightarrow \cO_L^\times / (1 + \pi_L \cO_L) \simeq \ell^\times$ and $G_{i} / G_{i+1} \hookrightarrow (1 + \pi_L^{i} \cO_L) / (1 + \pi_L^{i+1} \cO_L) \simeq \ell$ for $i \ge 1$, where the first embedding is given by $\sigma \mapsto \sigma(\pi_L) / \pi_L$ which is independent of choice of the uniformizer $\pi_L$.
When $K = K_{\mathfrak{p}}'$ and $L = L_\mathfrak{P}'$ are localizations of a number field $K'$ and its extension $L'$, we may also denote the decomposition group and inertia group as $D_{\mathfrak{P}}$ and $I_{\mathfrak{P}}$.
When $L'/K'$ is abelian, the decomposition and inertia groups only depend on $\frp$, and we denote them as $D_{\frp}$ or $I_\frp$ instead.

Let $e = e_{L/K}$ be the ramification degree ($\pi_K = \varepsilon\pi_L^e$ for some $\varepsilon \in \cO_L^\times$) and $f = f_{L/K} = [\ell : k]$ be the inertia degree.
Then $L^u /  K$ is an unramified extension of degree $f$, and $L^t / L^u$ is a tamely ramified extension of degree equal to the prime-to-$p$ part of $e$, and $L / L^t$ is a wildly ramified extension of degree equal to the $p$-part of $e$.
We say that the extension $L / K$ is tamely ramified if $L^t = L$, i.e. $G_1 = \{1\}$ or equivalently, $(p, e) = 1$.
Especially, $G_{0} \hookrightarrow \ell^\times$ and the inertia group has to be cyclic of order $e$. 
Moreover, we have the following proposition.

\begin{proposition}[{ Greenberg\cite[Lemma 1]{greenberg1974elementary},  Newton \cite[Proposition 2.2.9]{newton2012explicit}}]
\label{prop:newton}
    Let $L/K$ be a tamely ramified abelian extension of local fields, and $e = e_{L/K}$ be its ramification degree.
    Let $k$ be the residue field of $K$.
    Then $e$ divides $\#k^\times$.
\end{proposition}

For a tamely ramified extension of  $K$,
the conjugation action of (a lift of) Frobenius on the tame inertia group is simply given by $\#k$-power map:
\begin{proposition}[{ Iwasawa\cite{iwasawa1955galois}}]
\label{prop:tamefrobaction}
     Let $L / K$ be a tamely ramified extension with ramification degree $e$.
    Let $\tau$ be a generator of the inertia subgroup $G_0 \le G = \Gal(L / K)$ and $\sigma \in G$ be a lift of a Frobenius.
    Then  $\sigma \tau \sigma^{-1} = \tau^q$, where $q = \# k$ is the size of the residue field of $K$.
\end{proposition}

\begin{remark}
 For the later proofs, we will only use the case when $K = \bQ_p$ and $q = p$ (see Lemma \ref{lem:p16_cyclic} and \ref{lem:p110_cyclic}).
\end{remark}

One may ask whether a local field with a given decomposition type can be \emph{globalized}; that is, whether an extension $K / \bQ_p$ arises as the localization of a number field $K' / \bQ$.
One may even ask whether we can choose $K'$ to have the same Galois group, i.e., $\Gal(K' / \bQ) \simeq \Gal(K / \bQ_p)$.
When such a $K'$ exists and $K' \simeq \bQ(\alpha)$, $K \simeq \bQ_p(\alpha)$ for a root $\alpha$ of some $f(x) \in \bZ[x]$, we call the polynomial $f(x)$ a \emph{Galois splitting model}.
Carrillo \cite{carrillo2024finding} proposed several methods to find it, including Panayi's root finding algorithm \cite[Algorithm 8.4]{pauli2001computation}.

However, such a field $K'$ does not always exist. Indeed, Wang provided a famous counterexample to the incorrect proof of ``Grunwald's theorem'' \cite{wang1948counter, wang1950grunwald}.
In particular, Wang proved the following:
\begin{proposition}[{ Wang \cite{wang1948counter}}]
\label{thm:wangcounterex}
    Let $s \ge 3$ and $K' / \bQ$ be a cyclic extension of order $2^s$.
    Then the prime $2$ either totally ramifies in $K'$ or has at least two distinct prime factors.
\end{proposition}
For example, when $K / \mathbb Q_2$ is the unramified cyclic extension of order $8$, there exists no cyclic extension $K'/\mathbb Q$ of degree $8$ that localizes to $K$.

\subsection{Number fields and $p$-adic fields in \texttt{LMFDB}}
\label{subsec:prelim_lmfdb}

\texttt{LMFDB} \cite{lmfdb} is an online database that contains a wide range of mathematical objects in number theory and arithmetic geometry, including $L$-functions, modular forms, elliptic curves, number fields, $p$-adic fields, and representations.
In this work, we used the number field database for machine learning experiments and the local field database for the proofs of some theorems.

The number field database in \texttt{LMFDB} is a compilation of several works including \cite{kluners2001database,jones2014database,buchmann1993enumeration,oliver1992computation,voight2008enumeration,battistoni2020small}, where the detailed references can be found in the \emph{Source and Acknowledgements} page for the database \cite{lmfdb-nf-source}.
Since our focus is on Galois groups, we restrict our experiments to Galois extensions of $\mathbb Q$ of degrees $4,6,8,9$ and $10$, among all number fields in the database. The corresponding statistics are summarized in Table~\ref{tab:nf_count_galois}.

For a given number field, there are infinitely many choices of defining polynomials. 
\texttt{LMFDB} uses the \emph{normalized (reduced)} polynomial where the sum of the squares of the absolute values of all complex roots of a polynomial is minimized.
 If there is more than one such polynomial, choose the polynomials whose discriminant have minimal absolute value.
If there are still multiple polynomials, say $P = x^n + a_1 x^{n-1} + a_2 x^{n-2} + \cdots + a_n$, choose the one where the vector $S(P) := (|a_1|, a_1, \dots, |a_n|, a_n)$ has smallest lexicographic order (see \cite{lmfdb-nf-poly} for more details).
Note that these features (sum of squares of zeros, coefficients, etc.) are not field invariants, and the models trained on such data highly depends on such normalizations.
See Section \ref{subsubsec:49_poly}, \ref{sec:sextic_poly}, and \emph{Polynomial coefficients} paragraphs of \ref{subsubsec:octic_ab}, and \ref{subsubsec:octic_nab} for the experiments.

The \( p \)-adic field database in \texttt{LMFDB} is primarily based on the computations of Jones and Roberts \cite{jones2004nonic, roberts2006database, jones2008octic}, along with several other sources for specific data columns.  
Key columns include the ramification degree, inertia degree, Galois group, (wild) inertia group, and Galois splitting model \cite{carrillo2024finding}.
As of now, the database contains all degree \( n \) extensions of \( \mathbb{Q}_p \) for \( p < 200 \) and \( n \le 23 \); see the  \emph{completeness} page for $p$-adic fields \cite{lmfdb-completeness}.

\begin{table}[h]
    \begin{center}
        \begin{tabular}{c|c|c|c}
            \toprule
            Degree & Galois group & \multicolumn{2}{c}{Number of fields} \\
            \midrule
            \multirow{2}{*}{4} & $C_4$ (\texttt{4T1}) & 21873 (11.96\%) & \multirow{2}{*}{182860} \\
            & $C_2^2$ (\texttt{4T2}) & 160987 (88.04\%) \\
            \midrule
            \multirow{2}{*}{9} & $C_9$ (\texttt{9T1}) & 284 (22.43\%) & \multirow{2}{*}{1266} \\
            & $C_3^2$ (\texttt{9T2}) & 982 (77.57\%) \\
            \midrule
            \midrule
            \multirow{2}{*}{6} & $C_6$ (\texttt{6T1}) & 21146 (13.45\%) & \multirow{2}{*}{157168} \\
            & $S_3$ (\texttt{6T2}) & 136022 (86.55\%) \\
            \midrule
            \multirow{2}{*}{10} & $C_{10}$ (\texttt{10T1}) & 1854 (33.86\%) & \multirow{2}{*}{5476} \\
            & $D_5$ (\texttt{10T2}) & 3622 (66.14\%) \\
            \midrule
            \midrule
            \multirow{5}{*}{8} & $C_8$ (\texttt{8T1}) & 6206 (5.42\%) & \multirow{5}{*}{114576} \\
            & $C_4 \times C_2$ (\texttt{8T2}) & 19550 (17.06\%) \\
            & $C_2^3$ (\texttt{8T3}) & 10767 (9.40\%) \\
            & $D_4$ (\texttt{8T4}) & 27796 (24.26\%) \\
            & $Q_8$ (\texttt{8T5}) & 50257 (43.86\%) \\
            \bottomrule
        \end{tabular}
        \caption{Galois extensions of degree $4, 6, 8, 9, 10$ in \texttt{LMFDB} \cite{lmfdb}}
        \label{tab:nf_count_galois}
    \end{center}
\end{table}


\section{Galois groups of number fields}
\label{sec:galois}

In this main section, we investigate how a Galois group can be determined from a finite number of coefficients of the Dedekind zeta function and also from polynomial coefficients. This question is motivated and guided by machine learning (ML) experiments, in which these coefficients are used as features of the dataset. The details of the ML experiments are presented alongside theorems and proofs. While Galois groups can be determined by various classical methods, it is somewhat surprising that this particular approach of using zeta coefficients has not been explored in the existing literature. It is also worth noting that our criteria could only be formulated after analyzing the outcomes of the ML experiments.

\subsection{Experimental details}

As previously mentioned, we used the number field database from the \texttt{LMFDB} \cite{lmfdb}.
 The zeta coefficients are separately computed using the function \texttt{zeta\_coefficients} of \texttt{SageMath} \cite{sagemath} to make dataset for the experiments.
For each experiment, we report average and standard deviation of accuracy, balanced accuracy, and per-class precision, recall, F1 over 5 different random splits of ratio 80 to 20 for training and test sets.
All machine learning experiments were conducted using the  default setup of the \texttt{scikit-learn} library \cite{pedregosa2011scikit}, with fixed random state and no hyperparameter tuning except for setting the maximum number of iterations in logistic regression,  which is set to be 10000.

For logistic regression models, they are trained with LBFGS solver with $L^2$-regularization.
For decision trees, we trained them with Gini criterion and we did not set the maximum depths, so that the nodes are expanded until all leaves are pure.
We also use \texttt{best} splitter, which uses midpoints of features values as thresholds.
See the official \texttt{scikit-learn} documentations \cite{sklearn-lr,sklearn-dt} for more details about default configurations.

\subsection{From experiments to theorems}
\label{subsec:pipeline}

We briefly describe the overall workflow underlying our machine learning experiments and the subsequent process of formulating and proving theorems.

In the decision tree experiments involving zeta coefficients, we observe that coefficients indexed by perfect powers consistently appear near the top of the trees. In the cases of quartic and nonic extensions, the resulting decision trees (Figures~\ref{fig:quartic_zeta_dt} and \ref{fig:nonic_zeta_dt}) are remarkably simple. This allows us to readily conjecture that the vanishing of certain zeta coefficients characterizes cyclic extensions. We are able to prove both this conjecture \emph{and} its converse (see Corollary~\ref{cor:ellsqzeta}).

For sextic, octic, and decic extensions, the decision trees become more intricate; nevertheless, the prominence of coefficients with perfect power indices remains evident. Motivated by this pattern, we conduct further exploratory data analysis (EDA), which reveals that specific values of zeta coefficients at prime-power indices occur only when the extension has a particular Galois group and when the corresponding prime satisfies certain congruence conditions. We subsequently establish these phenomena rigorously using Propositions~\ref{prop:newton}, \ref{prop:tamefrobaction}, and Theorem~\ref{thm:wangcounterex}, together with additional group-theoretic arguments.

\subsection{Quartic and nonic fields}
\label{subsec:galois_quartic_nonic}

Let $K/\bQ$ be a Galois extension of degree 4.
There are two possible Galois groups: cyclic group $C_{4} = \bZ / 4$ or the Klein group $V_{4} = C_2^2 = \bZ / 2 \times \bZ / 2$.
It is completely known how to determine the Galois group of a splitting field of a quartic polynomial $f(x) \in \bQ[x]$ in terms of discriminants and resolvents \cite[p. 613]{dummit2004abstract}.
Especially, for a quartic Galois extension $K$, one has
\begin{equation}
    \label{eqn:quartic_disc}
    \Gal(K / \bQ) \simeq \begin{cases}
        V_4 & \text{ if } \Delta_K \text{ is a perfect square, } \\
        C_4 & \text{ otherwise. } 
    \end{cases}
\end{equation}
Similarly, if $K / \bQ$ is a degree 9 Galois extension, then there are only two possible Galois groups: cyclic group $C_9 = \bZ / 9$ or $C_3^2 = \bZ / 3 \times \bZ / 3$.
In general, only possible Galois groups for degree $\ell^2$-Galois-extensions for prime $\ell$ are $C_{\ell^2}$ or $C_\ell^2$, since those are the only groups of order $\ell^2$ up to isomorphism. In the rest of this subsection, we will use zeta and polynomial coefficients, respectively, to distinguish between the two cases.

\subsubsection{Learning Galois groups from zeta coefficients}

\paragraph{Decision tree}

\begin{figure}[t]
    \centering
    \footnotesize
    \begin{forest}
      label L/.style={
        edge label={node[midway,left,font=\scriptsize]{#1}}
      },
      label R/.style={
        edge label={node[midway,right,font=\scriptsize]{#1}}
      },
      for tree={
        font=\normalsize,
        forked edge,
        child anchor=north,
        for descendants={
          {edge=->}
        }
      },
      [{$a_{2^2 \cdot 3^2 \cdot 5^2} = 0$}, draw
        [{$C_{4}$}, rectangle, thick, draw, label L=Y, tier=bottom, line width=1.5pt]
        [{$a_{23^2} = 0$}, draw, label R=N, edge=very thick,
            [{$C_{4}$}, rectangle, thick, draw, label L=Y, tier=bottom, line width=1.5pt]
            [{$a_{19^2} = 0$}, draw, label R=N, edge=very thick,
                [{$C_{4}$}, rectangle, thick, draw, label L=Y, tier=bottom, line width=1.5pt]
                [{$a_{31^2} = 0$}, draw, label R=N, edge=very thick,
                    [{$C_{4}$}, rectangle, thick, draw, label L=Y, tier=bottom, line width=1.5pt]
                    [{$a_{2^2 \cdot 7^2} = 0$}, draw, label R=N, edge=very thick,
                        [{$C_{4}$}, rectangle, thick, draw, label L=Y, tier=bottom, line width=1.5pt]
                        [{$a_{11^2} = 0$}, draw, label R=N, edge=very thick,
                            [{$C_{4}$}, rectangle, thick, draw, label L=Y, tier=bottom, line width=1.5pt]
                            [{$a_{2^2 \cdot 3^2 \cdot 5^2} \le 1$}, draw, label R=N, edge=very thick,
                                [{$a_{2^7 \cdot 5} = 0$}, draw, label L=Y, edge=very thick
                                    [{$V_{4}$}, rectangle, thick, draw, label L=Y, tier=bottom, line width=1.5pt]
                                    [{$C_{4}$}, rectangle, thick, draw, label R=N, tier=bottom, line width=1.5pt, edge=very thick, fill=gray!40]
                                ]
                                [{$V_{4}$}, rectangle, thick, draw, label R=N, tier=bottom, line width=1.5pt]
                            ]
                        ]
                    ]
                ]
            ]
        ]
      ]
    \end{forest}
    \caption{A decision tree predicts the Galois groups of quartic fields from zeta coefficients up to $a_{1000}(K)$.}
    \label{fig:quartic_zeta_dt}
\end{figure}

We found that decision tree models can achieve  almost perfect accuracy (100\% on 4 out of 5 runs) for distinguishing Galois groups of quartic fields using zeta coefficients $a_n(K)$ up to $n \le 1000$.
Moreover, the splitting of the first few decision nodes (including the root node) have a form of $a_{m^2}(K)  = 0$\footnote{ In the actual experiments, the conditions are of the form $a_{m^2}(K) \le 0.5$, since we are using the \texttt{best} splitter which chooses midpoints of feature values as thresholds \cite{sklearn-dt2}.}, i.e. whether $a_{m^2}(K)$ is zero or not, and predict as a cyclic extension when it is true (Figure \ref{fig:quartic_zeta_dt}).
This strongly suggests the  criterion: for a quartic Galois extension $K / \bQ$, we have $\Gal(K/\bQ) \simeq C_4$  if  $a_{m^2}(K) = 0$ for some $m \ge 2$.

\begin{figure}[t]
    \centering
    \footnotesize
    \begin{forest}
      label L/.style={
        edge label={node[midway,left,font=\scriptsize]{#1}}
      },
      label R/.style={
        edge label={node[midway,right,font=\scriptsize]{#1}}
      },
      for tree={
        font=\normalsize,
        forked edge,
        child anchor=north,
        for descendants={
          {edge=->}
        }
      },
      [{$a_{2^3 \cdot 5^3} \le 4$}, draw,
        [{$C_{9}$}, rectangle, thick, draw, label L=Y, tier=bottom, line width=1.5pt, edge=very thick, fill=gray!40]
        [{$a_{7^3} = 0$}, draw, label R=N, edge={very thick, densely dashed}
            [{$C_{9}$}, rectangle, thick, draw, label L=Y, tier=bottom, line width=1.5pt]
            [{$a_{3^3} = 0$}, draw, label R=N, edge={very thick, densely dashed}
                [{$C_{9}$}, rectangle, thick, draw, label L=Y, tier=bottom, line width=1.5pt]
                [{$C_{3}^{2}$}, rectangle, thick, draw, label R=N, tier=bottom, line width=1.5pt, edge={very thick, densely dashed}, fill=gray!20]
            ]
        ]
      ]
    \end{forest}
    \caption{A decision tree predicts the Galois groups of nonic fields from zeta coefficients up to $a_{1000}(K)$.}
    \label{fig:nonic_zeta_dt}
\end{figure}

Similarly, for degree 9 Galois extensions of $\bQ$, decision tree models can achieve 100\% accuracy on distinguishing Galois groups ($C_9$ or $C_3^2$) with zeta coefficients $a_n(K)$ up to $n \le 1000$.
More surprisingly, Figure \ref{fig:nonic_zeta_dt} shows that the model uses only three zeta coefficients $a_{1000}, a_{343}$ and $a_{27}$ out of 1000, and it strongly suggests that the prediction is based on the cube-indexed zeta coefficients.
From these observations, it is natural to ask if one can distinguish two possible Galois groups of degree $\ell^2$ Galois extensions of $\bQ$, using $\ell$-power coefficients.
 In particular, it is natural to conjecture the following: let $K / \bQ$ be a Galois extension of degree $\ell^2$. If $a_{p^\ell}(K) = 0$ for some prime $p$, then $K / \bQ$ is a cyclic extension.

\begin{table}[h]
    \begin{center}
        \begin{tabular}{c|c|c|c}
            \toprule
            $(e, f, g)$ & Euler factor & $a_{p}(K)$ & $a_{p^\ell}(K)$ \\
            \midrule
            $(1, 1, \ell^2)$ & $(1 - p^{-s})^{-\ell^2} = \sum_{k \ge 0} \binom{\ell^2 + k - 1}{\ell^2 - 1}p^{-ks}$ & $\ell^2$ & $\binom{\ell^2 + \ell - 1}{\ell^2 - 1}$ \\
            $(1, \ell, \ell)$ & $(1 - p^{-\ell s})^{-\ell} = \sum_{k \ge 0} \binom{\ell + k - 1}{\ell - 1} p^{-\ell k s}$ & $0$ & $\ell$ \\
            $(1, \ell^2, 1)$ & $(1 - p^{-\ell^2 s})^{-1} = \sum_{k \ge 0}p^{-\ell^2 k s}$ & $0$ & $0$ \\
            $(\ell, 1, \ell)$ & $(1 - p^{-s})^{-\ell} = \sum_{k \ge 0} \binom{\ell + k - 1}{\ell - 1} p^{-ks}$ & $\ell$ & $\binom{2\ell - 1}{\ell - 1}$ \\
            $(\ell, \ell, 1)$ & $(1 - p^{-\ell s})^{-1} = \sum_{k \ge 0} p^{-\ell k s}$ & $0$ & $1$ \\
            $(\ell^2, 1, 1)$ & $(1 - p^{-s})^{-1} = \sum_{k \ge 0} p^{-ks}$ & $1$ & $1$ \\
            \bottomrule
        \end{tabular}
        \caption{Decomposition types, Euler factors, and $a_p(K)$ and $a_{p^{\ell}}(K)$ for degree $\ell^2$ extensions.}
        \label{tab:lsqeuler}
    \end{center}
\end{table}

In Theorem \ref{thm:ClClextn}, we show that the conjecture is true. More precisely,

using Proposition \ref{prop:dedekind_euler} and Corollary \ref{cor:prime_pow_zeta_coeff}, one can list all the possible decomposition types of a given prime $p$ in $K$ and the corresponding Euler factors and $p^\ell$-th zeta coefficients, as shown in Table \ref{tab:lsqeuler}.
By proving that certain decomposition types cannot occur for $C_\ell^2$-extensions, we show that $p^\ell$-th zeta coefficients are nonvanishing for such extensions.

\begin{theorem}
    \label{thm:ClClextn}
    Let $\ell$ be a prime and $K / \bQ$ be a Galois extension with $\Gal(K / \bQ) \simeq C_\ell^2$.
    Then $a_{p^\ell}(K) \ne 0$ for all prime $p$.
\end{theorem}
\begin{proof}
     The last column of Table \ref{tab:lsqeuler} shows that $a_{p^\ell}(K) = 0$ if and only if the decomposition type of $p$ is $(e, f, g) = (1, \ell^2, 1)$, i.e. totally inerts in $K$. Hence, it is enough to show that such decomposition type cannot occur for $C_{\ell}^2$-extensions.
    If it happens, we should have $I_p = 1$ and $D_p = C_\ell^2$, but  \[ D_p \simeq D_p / I_p \simeq \Gal(\mathbb{F}_{p^{\ell^2}} / \mathbb{F}_p) \simeq C_{\ell^2},\] which is a contradiction.
\end{proof}

In contrast, there are infinitely many primes $p$ for which $a_{p^\ell}(K) = 0$ when $K$ is a cyclic extension of degree $\ell^2$.
We can prove the following general statement for all cyclic extensions of prime-power degree,  which implies that the converse of Theorem \ref{thm:ClClextn} is also true.
\begin{theorem}
    \label{thm:cyclic_degree_lpower}
    Let $\ell$ be a prime and $K / \bQ$ be a cyclic extension of degree $\ell^m$ for some $m \ge 2$, i.e. $G = \Gal(K / \bQ) \simeq C_{\ell^m}$.
     Then, for any $b \ge 1$ not divisible by $\ell^m$, the natural density of primes $p$ such that $a_{p^b}(K) = 0$ is positive. More precisely, if $\nu = v_\ell(b) < m$ is the $\ell$-adic valuation of $b$, then the density of primes with $a_{p^b}(K) = 0$ is $(\ell^{m - \nu} - 1) / \ell^{m - \nu}$.
    
\end{theorem}
\begin{proof}
    
    Consider unramified prime $p$ of decomposition type $(e, f, g) = (1, \ell^{t}, \ell^{m-t})$.
    The corresponding Euler factor is 
    \[
        \zeta_{K, p}(s) = (1 - p^{-\ell^{t}s})^{-\ell^{m - t}} = \sum_{k \ge 0} \binom{\ell^{m-t} + k - 1}{\ell^{m-t} - 1} p^{-\ell^{t}k s}
    \]
    and this shows $a_{p^b}(K) = 0$ if and only if $t > \nu$, i.e. $\ell^t \nmid b$.
    For each $t > \nu$, $\mathrm{Frob}_p$ has order $\ell^{t}$ and there are $\varphi(\ell^t) = \ell^{t - 1}(\ell - 1)$-many elements in $C_{\ell^m}$ of order $\ell^t$, so the density of prime $p$ with $a_{p^b}(K) = 0$ is
    \[
    \sum_{\nu < j \le m} \frac{\ell^{j-1}(\ell - 1)}{\ell^m} = \frac{\ell - 1}{\ell^m} \frac{\ell^{\nu}(\ell^{m - \nu} - 1)}{\ell - 1} = \frac{\ell^{m - \nu} - 1}{\ell^{m -\nu}}.
    \]
    Note that the natural density and Dirichlet density coincides in Chebotarev density theorem.
    
\end{proof}

 By combining two theorems, we obtain the following result suggested by the decision tree experiment.
\begin{corollary}
\label{cor:ellsqzeta}
    Let $\ell$ be a prime and $K / \bQ$ be a Galois extension of degree $\ell^2$, hence $\Gal(K/\bQ) \simeq C_{\ell^2}$ or $C_{\ell}^2$.
    Then $\Gal(K / \bQ) \simeq C_{\ell^2}$ if and only if there exists a prime $p$ such that $a_{p^\ell}(K) = 0$.
\end{corollary}

In fact, for each Galois group and prime, we can analyze the possible decomposition types, as will be discussed below.

\begin{proposition}
    Let $K / \bQ$ be a quartic $C_4$-extension and $p \equiv 3 \pmod{4}$ be a rational prime with decomposition type $(e, f, g)$ in $K$.
    Then $(e, f, g) \ne (4, 1, 1)$.
\end{proposition}
\begin{proof}
    Assume $(e, f, g) = (4, 1, 1)$.
    Since $(p, 4) = 1$, the corresponding extension $K_\frp / \bQ_p$ (for some $\frp|p$) is a tamely ramified extension, hence we should have $4 \mid p - 1$ by Proposition \ref{prop:newton}.
    This cannot happen since $p \equiv 3 \pmod{4}$.
\end{proof}

\begin{proposition}
    \label{prop:C2C2ram}
    Let $K / \bQ$ be a quartic $C_2^2$-extension and $p$ be a rational prime with decomposition type $(e, f, g)$ in $K$. Then:
    \begin{enumerate}
        \item $(e, f, g) \ne (1, 4, 1)$.
        \item If $p \ne 2$, then $(e, f, g) \ne (4, 1, 1)$.
    \end{enumerate}
\end{proposition}
\begin{proof}
    1) Since $C_4$ cannot be a quotient of $C_2^2$, we have $(e, f, g) \ne (1, 4, 1)$.
    
    2) For odd $p$, the corresponding extension $K_\frp / \bQ_p$ is tamely ramified and the inertia group has to be cyclic, so it cannot be totally ramified.
\end{proof}

\begin{example}[$C_4$, sufficient condition]
    Assume that a quartic Galois extension \( K / \mathbb{Q} \) follows the bold branch in Figure~\ref{fig:quartic_zeta_dt}, so that the decision tree predicts the Galois group to be \( C_4 \).   In this case, we have
\[
a_{30^2} = a_{2^2} \cdot a_{3^2} \cdot a_{5^2} = 1 \quad \Leftrightarrow \quad a_{2^2} = a_{3^2} = a_{5^2} = 1,
\]
and
\[
a_{640} = a_{2^7} \cdot a_5 >  0 \quad \Rightarrow \quad a_5 \ne 0.
\]
 From Table~\ref{tab:lsqeuler}, we can check that the only possible decomposition type of $p = 5$ satisfying $a_{5} \ne 0$ and $a_{5^2} = 1$ is $(e, f, g) = (4, 1, 1)$, i.e. totally ramified.
By Proposition~\ref{prop:C2C2ram}, \( K / \mathbb{Q} \) cannot be a \( C_2^2\)-extension, and thus must be cyclic.  
    This shows that the branch always gives correct prediction.
\end{example}

A similar analysis can be carried out for nonic extensions.

\begin{proposition}
    \label{prop:C9ram}
    Let $K / \bQ$ be a nonic $C_9$-extension and let $p \ne 3$ be a rational prime with decomposition type $(e, f, g)$ in $K$. Then: 
    \begin{enumerate}
        \item If $p \equiv 4, 7 \pmod{9}$, then $(e, f, g) \ne (9, 1, 1)$.
        \item If $p \equiv 2 \pmod{3}$, then $(e, f, g) \not \in \{(3, 1, 3), (3, 3, 1), (9, 1, 1)\}$.
    \end{enumerate}
\end{proposition}
\begin{proof}
1)    Assume that $p \ne 3$ and $(e, f, g) = (9, 1, 1)$.
    Then the corresponding extension $K_\frp / \bQ_p$ (for $\frp|p$) is a totally and tamely ramified $C_9$-extension. By Proposition \ref{prop:newton} this implies $p \equiv 1 \pmod{9}$. Thus $p \not \equiv 4,7 \pmod 9$ and $p \not \equiv 2 \pmod 3$. In particular, this proves the first claim.

  2)  Now assume $p \equiv 2 \pmod{3}$.
    We already know $(e, f, g) \ne (9, 1, 1)$.
    Assume $(e, f, g) = (3, 1, 3)$. Then the corresponding extension $K_\frp / \bQ_p$ is a totally and tamely ramified $C_3$-extension.
    Then Proposition \ref{prop:newton} again shows $p \equiv 1 \pmod{3}$, a contradiction.
    For the case of $(3, 3, 1)$, the corresponding extension $K_\frp / \bQ_p$ has an unramified $C_3$-subextension $\tilde K \subset K_\frp$, where $K_\frp / \tilde K$ is totally and tamely ramified $C_3$-extension.
    Since the residue field of $\tilde K$ is $\bF_{p^3}$, Proposition \ref{prop:newton} implies that $p^3 \equiv 1 \pmod{3}$, which contradicts to $p \equiv 2 \pmod{3}$.
\end{proof}

\begin{proposition}
    \label{prop:C3C3ram}
    Let $K / \bQ$ be a nonic $C_3^2$-extension and let $p \ne 3$ be a rational prime with decomposition type $(e, f, g)$ in $K$. Then: 
    \begin{enumerate}
        \item If $p \equiv 1 \pmod{3}$, then $(e, f, g) \not \in \{(1, 9, 1), (9, 1, 1)\}$
        \item If $p \equiv 2 \pmod{3}$, then $(e, f, g) \not \in \{(1, 9, 1), (3, 1, 3), (3, 3, 1), (9, 1, 1)\}$
    \end{enumerate}
\end{proposition}
\begin{proof}
The decomposition type    $(e, f, g) = (1, 9, 1)$ (resp. $(9, 1, 1)$) cannot occur since $C_9$ is not a quotient (resp. subgroup) of $C_3^2$.
    Now assume $p \equiv 2 \pmod{3}$.
    If $(e, f, g) = (3, 1, 3)$, we have $(p) = \frp_1^3 \frp_2^3 \frp_3^3$, and a corresponding extension of local fields $K_\frp / \bQ_p$ (for any $\frp = \frp_i$) is a totally and tamely ramified $C_3$-extension.
    By Proposition \ref{prop:newton}, we have $3 \mid p - 1$, a contradiction.
    If $(e, f, g) = (3, 3, 1)$, we have $(p) = \frp^3$ with $f(\frp|p) = 3$.
    Let $\tilde K = K^{I_\frp}$ be the fixed field of the inertia group $I_\frp$, which is an unramified $C_3$-extension of $\bQ_p$.
    Then its residue field is $\bF_{p^3}$, and $K / \tilde K$ is a totally and tamely ramified $C_3$-extension.
    By Proposition \ref{prop:newton} again, we get $3 \mid p^3 - 1$, which contradicts to $p \equiv 2 \pmod{3}$.
\end{proof}

By Proposition \ref{prop:C3C3ram} and Table \ref{tab:lsqeuler}, we have:
\begin{corollary}
    \label{cor:C3C3zc}
    Let $K / \bQ$ be a nonic extension with $\Gal(K / \bQ) \simeq C_3^2$ and let $p$ be a rational prime.
    If $p \equiv 2 \pmod{3}$, then we have $a_{p^3}(K) \in \{3, 165\}$.
\end{corollary}

\begin{example}[$C_9$, sufficient condition]
    \label{ex:C9_1000}
    Assume that a nonic Galois extension \( K / \mathbb{Q} \) follows the bold branch in Figure~\ref{fig:nonic_zeta_dt}, so that \( a_{10^3} = a_{2^3} \cdot a_{5^3} \le  4 \).  
Then \( \Gal(K / \mathbb{Q}) \simeq C_9 \).  
Indeed, if \( \Gal(K / \mathbb{Q}) \simeq C_3^2 \), then by Corollary~\ref{cor:C3C3zc}, we should have
\(
a_{10^3} = a_{2^3} \cdot a_{5^3} \ge 3 \cdot 3 = 9.
\)
This explains the correctness of the prediction.  

\end{example}

\begin{example}[ False positive for $C_3^2$]
\label{ex:C3C3}
    Although the decision tree (Figure \ref{fig:nonic_zeta_dt}) attains 100\% accuracy on the test set, this does not imply that it is a perfect classifier. Indeed, there exits a $C_9$-extension that satisfies all the conditions of the right-most branch (the dashed bold branch).
    With brute-force search, we found a prime $q = 960643$ satisfying $q \equiv 1 \pmod{9}$ and $p^{\frac{q-1}{9}} \equiv 1 \pmod{q}$ for $p = 2, 3, 5, 7$.
    Let \( L = \mathbb{Q}(\zeta_q) \), and let \( K \) be the unique nonic subextension of \( L \).    Then the Frobenius at $p = 2, 3, 5, 7$ are all trivial on $K$, so these primes split completely in $K$. From Table~\ref{tab:lsqeuler}, we find that \( a_{2^3} = a_{3^3} = a_{5^3} = a_{7^3} = 165 \), satisfying all the inequalities required by the decision tree, even though \( K \) is not a \( C_3^2 \)-extension.
    The extension is ramified only at $960643$, and it is not currently listed in \texttt{LMFDB}~\cite{lmfdb}.
    
    Note that there are infinitely many such primes \( q \) by the Chebotarev density Theorem.   More precisely, an odd prime \( q \) satisfies \( q \equiv 1 \pmod{9} \) if and only if \( \Phi_9(x) = x^6 + x^3 + 1 \equiv 0 \pmod{q} \) has a solution.   For such a prime \( q \), we have \( p^{\frac{q - 1}{9}} \equiv 1 \pmod{q} \) if and only if the polynomial \( f_p(x) \coloneqq x^9 - p \) has a solution modulo \( q \).   By the Chebotarev density Theorem, the set of primes \( q \) for which all of \( \Phi_9(x), f_2(x), f_3(x), f_5(x), f_7(x) \) have solutions modulo \( q \) has positive (natural) density; for example, the compositum of the splitting fields of these polynomials has a degree $39366 = 2 \cdot 3^9$ (which can be verified, for instance, using \texttt{Magma}~\cite{magma}). Hence, at least $1 / 39366$ of all primes split completely in the compositum, and for such primes \(q\), each polynomial has a solution modulo \(q\).

\end{example}

\paragraph{Logistic regression}

We also trained logistic regression models and see if they can successfully distinguish Galois groups for quartic and nonic number fields.
With 1000 zeta coefficients (resp. with standard normalization), they can obtain accuracy of  99.60\% (resp. 99.41\%) and 97.16\% (resp. 78.51\%) for quartic and nonic fields, respectively (Table \ref{tab:galois_4} and \ref{tab:galois_9}).
We checked the indices with largest and smallest weights (Table \ref{tab:quartic_nonic_lr_weights}).
One can find that most of the indices for large weights are square or cube-indexed coefficients, which means that the model tends to predict output as \emph{1}, which corresponds to $C_2^2$ or $C_3^2$, when those features are large.
These results align well with Corollary \ref{cor:ellsqzeta}.

\begin{table}[h]
    \begin{center}
        \resizebox{\columnwidth}{!}{%
        \begin{tabular}{c|c|c|c|c|c|c|c|c|c|c|c}
            \toprule
            \multirow{2}{*}{quartic} & top & \textbf{25} (8.4) & \textbf{49} (7.2) & \textbf{961} (6.1) & \textbf{841} (5.5) & \textbf{529} (5.3) & \textbf{9} (5.1) & \textbf{361} (4.3) & \textbf{36} (3.2) & \textbf{100} (3.0) & \textbf{289} (3.0) \\
            \cline{2-12}
            & bottom & 343 (-10.6) & \textbf{625} (-7.6) & 31 (-6.8) & \textbf{81} (-6.3) & 29 (-5.7) & 23 (-5.5) & 19 (-4.7) & \textbf{324} (-3.7) & \textbf{16} (-3.6) & 17 (-3.4) \\
            \midrule
            \multirow{2}{*}{nonic} & top & 392 (2.0) & \textbf{27} (1.9) & 837 (1.9) & \textbf{8} (1.8) & 999 (1.8) & 456 (1.6) & \textbf{216} (1.6) & 248 (1.6) & \textbf{125} (1.6) & 639 (1.5) \\
            \cline{2-12}
            & bottom & 213 (-1.3) & 111 (-1.3) & 321 (-1.0) & 57 (-1.0) & 291 (-0.8) & 3 (-0.8) & 717 (-0.7) & 87 (-0.7) & 49 (-0.6) & 807 (-0.6) \\
            \midrule
            \multirow{2}{*}{nonic*} & top & \textbf{216} (0.6) & \textbf{1000} (0.5) & \textbf{64} (0.2) & 248 (0.2) & 513 (0.1) & \textbf{125} (0.1) & \textbf{8} (0.1) & 584 (0.1) & \textbf{27} (0.1) & 424 (0.1) \\
            \cline{2-12}
            & bottom & 625 (-0.5) & 199 (-0.2) & 256 (-0.2) & 267 (-0.2) & 569 (-0.2) & 941 (-0.1) & 677 (-0.1) & 619 (-0.1) & 200 (-0.1) & 241 (-0.1) \\
            \bottomrule
        \end{tabular}
        }
        \caption{ Indices of the 10 largest and smallest weights (with actual weight values in parentheses) among 1000 logistic regression coefficients used to classify Galois groups of quartic and nonic fields.  The last two rows with nonic* is the weights of the model trained with original zeta coefficients without normalization.
        The bold numbers correspond to square or cube-indexed coefficients.}
        \label{tab:quartic_nonic_lr_weights}
    \end{center}
\end{table}

We also compared the distribution of the absolute values of the model weights for square-indexed (resp. cube-indexed) coefficients with the overall distribution.  
We found that the former tends to be higher on average than the latter (Figure~\ref{fig:galois_4_9_lr_weights}).  
For quartic (resp. nonic) fields, the average absolute value of all weights is 0.19 (resp. 0.02), while the average for square-indexed (resp. cube-indexed) coefficients is 1.79 (resp. 0.20).  
This shows that the logistic regression models, like the decision tree models, are also able to learn that such coefficients are useful for prediction.

Table \ref{tab:galois_4} and \ref{tab:galois_9} also include results for different subsets of zeta coefficients: using only square-indexed (or cube-indexed) coefficients, or using only prime(-power)-indexed coefficients.  
Notably, using just 31 square-indexed coefficients still yields an accuracy of 99.99\% (resp. 100\%) for quartic fields (resp. nonic fields), further validating their importance.  
Moreover, the performance of models trained only on prime-power-indexed coefficients matches that of models using all coefficients,
indicating that the logistic regression model is able to learn the importance of such coefficients for prediction, in agreement with the decision tree model.
In contrast, the model achieves only 91.11\% accuracy on quartic dataset with prime-indexed coefficients, which is not particularly high given the dataset's bias: 88.04\% of quartic Galois fields have \( C_2^2 \) as their Galois group (Table~\ref{tab:nf_count_galois}).

\begin{remark}
     Mostly, the performance of logistic regression models with normalized zeta coefficients are similar or better than one trained with original zeta coefficients, e.g. quartic (\ref{tab:galois_4}), sextic (\ref{tab:galois_6}), and decic (\ref{tab:galois_10}). However, this is not the case for nonic Galois extensions, where the accuracy with original zeta coefficient (97.60\%) is much higher than with normalized (78.51\%), where the latter barely matches the majority baseline (Table \ref{tab:nf_count_galois}). It would be interesting to figure out why this happens specifically for nonic extensions. 
\end{remark}

\subsubsection{Learning Galois groups from polynomial coefficients}
\label{subsubsec:49_poly}

\begin{figure}[t]
    \centering
    \footnotesize
\begin{forest}
  label L/.style={edge label={node[midway,left,font=\scriptsize]{#1}}},
  label R/.style={edge label={node[midway,right,font=\scriptsize]{#1}}},
  for tree={forked edge, child anchor=north, for descendants={edge=->}}
  [{$c_{2} \le -16650$}, draw
    [{$c_{3} \le -2$}, draw, label L=Y, edge={very thick, densely dashed}
      [{$C_2^2$}, rectangle, thick, draw, label L=Y, tier=bottom, line width=1.5pt]
      [{$c_{1} \le -31742$}, draw, label R=N, edge={very thick, densely dashed}
        [{$c_{0} \le 5075281152$}, draw, label L=Y, edge={very thick, densely dashed}
          [{$c_{1} \le -60201$}, draw, label L=Y, edge={very thick, densely dashed}
            [{$c_{1} \le -80278$}, draw, label L=Y, edge={very thick, densely dashed}
              [{$C_4$}, rectangle, thick, draw, label L=Y, tier=bottom, line width=1.5pt, edge={very thick, densely dashed}, fill=gray!20]
              [{${\cdots}$}, rectangle, draw, label R=N, tier=bottom]
            ]
            [{${\cdots}$}, rectangle, draw, label R=N, tier=bottom]
          ]
          [{${\cdots}$}, rectangle, draw, label R=N, tier=bottom]
        ]
        [{${\cdots}$}, rectangle, draw, label R=N, tier=bottom]
      ]
    ]
    [{${\cdots}$}, rectangle, draw, label R=N, tier=bottom]
  ]
\end{forest}
    \caption{A decision tree predicts the Galois groups of quartic fields from polynomial coefficients.}
    \label{fig:quartic_poly_dt}
\end{figure}

We also conduct experiments on predicting Galois groups from the coefficients of defining polynomials, using decision tree and logistic regression models.  
Similar experiments were previously carried out in~\cite{he2022machine} using random forest models.

Let \( c_0, \dots, c_{d-1} \) be the coefficients of a monic polynomial of degree \( d \), so that \( f(x) = x^d + \sum_{k=0}^{d-1} c_k x^k, \)
where the polynomial is uniquely determined by the procedure described in Section~\ref{subsec:prelim_lmfdb}.  
Using these polynomial coefficients as features, decision tree models achieve reasonably high accuracies of  95.16\% and  94.15\% for quartic and nonic fields, respectively, although the resulting models are highly complex and not easily interpretable.
Nevertheless, we found that the distributions of the coefficients vary significantly across Galois groups, and the models appear to exploit this variation for prediction (Figure~\ref{fig:quartic_poly_dt}).  
As shown in Table~\ref{tab:quartic_poly}, the constant term \( c_0 \) tends to be much larger for polynomials with Galois group \( C_4 \) compared to those with \( C_2^2 \), while \( c_1 \) and \( c_2 \) (resp. \( c_3 \)) are much (resp. slightly) smaller.

There are 146{,}288 quartic fields in the training set, of which 17{,}483 (resp. 128{,}805) have Galois group \( C_4 \) (resp. \( C_2^2 \)).  
The first node of the decision tree checks whether \( c_2 \le -16{,}650 \).  
Only 576 of the \( C_2^2 \)-extensions (0.44\% of the 128{,}805 \( C_2^2 \)-extensions) satisfy this condition, while 3{,}025 of the \( C_4 \)-extensions (17.3\% of the 17{,}483 \( C_4 \)-extensions) do.
Following the remaining dashed path in Figure~\ref{fig:quartic_poly_dt}, the tree predicts that a quartic field with   \[ c_0 \le 5{,}075{,}281{,}152, \quad c_1 \le -80{,}278, \quad \text{and} \quad c_3 > -2 \]
has Galois group \( C_4 \).  
This prediction is reasonable in light of the discrepancy in the distribution of polynomial coefficients shown in Table~\ref{tab:quartic_poly}.

Similarly, we found that the distributions of polynomial coefficients differ significantly across Galois groups of nonic extensions (Table~\ref{tab:nonic_poly}), and the decision tree model appears to use this information for prediction.   The first node of the tree checks whether \( c_1 \le 233{,}597{,}936 \approx 2.3 \times 10^8 \); if the condition is not satisfied, the model predicts the Galois group to be \( C_9 \).   This is reasonable, considering that the values of \( c_1 \) for \( C_9 \)-extensions are typically large and positive, while those for \( C_3^2 \)-extensions tend to be negative.
At present, we do not have a clear mathematical explanation for this discrepancy, and it would be interesting to explore one.  
We emphasize again that this phenomenon may strongly depend on how \texttt{LMFDB} selects a unique defining polynomial for each number field (see Section~\ref{subsec:prelim_lmfdb}).

In contrast, the logistic regression model performed very poorly on quartic fields when using polynomial coefficients, achieving only  22.45\% accuracy, which is significantly worse than the accuracy by a majority guess. This poor performance is due to the large magnitude of the polynomial coefficients; after standard normalization, the model's accuracy matches or higher than the  majority baseline:  88.06\% for quartic fields and  85.46\% for nonic fields (the majority baselines are 88.04\% and 77.57\%, respectively).

\begin{table}[h]
    \begin{center}
        \resizebox{0.6\columnwidth}{!}{%
        \begin{tabular}{c|c|c|c|c|c|c|c|c}
            \toprule
             & \multicolumn{2}{c|}{$c_0$} & \multicolumn{2}{c|}{$c_1$} & \multicolumn{2}{c|}{$c_2$} & \multicolumn{2}{c}{$c_3$}  \\
            \midrule
            $C_4$ & $1.0 \cdot 10^{11}$ & $2.1 \cdot 10^6$ & $-1.1 \cdot 10^7$ & $0$ & $-4.1 \cdot 10^4$ & $-189$ & $-0.6$ & $-1$  \\
            \midrule
            $C_2^2$ & $6.3 \cdot 10^7$ & $9.7 \cdot 10^4$ & $77.1$ & $0$ & $-4.6 \cdot 10^2$ & $-52$ & $-0.5$ & $0$ \\
            \bottomrule
        \end{tabular}
        }
        \caption{Mean (left) and median (right) of polynomial coefficients for quartic extensions.}
        \label{tab:quartic_poly}
    \end{center}
\end{table}

\begin{table}[h]
\centering
    \resizebox{0.75\columnwidth}{!}{%
    \begin{tabular}{c|c|c|c|c|c|c|c|c}
        \toprule
         & \multicolumn{2}{c|}{$c_0$} & \multicolumn{2}{c|}{$c_1$} & \multicolumn{2}{c|}{$c_2$} & \multicolumn{2}{c}{$c_3$} \\
        \midrule
        $C_9$ & $-4.5 \cdot 10^8$ & $-1.9 \cdot 10^5$ & $8.1 \cdot 10^8$ & $8.1 \cdot 10^6$ & $-4.9 \cdot 10^7$ & $-5.0 \cdot 10^4$ & $-2.4 \cdot 10^7$ & $-4.6 \cdot 10^6$ \\
        \midrule
        $C_3^2$ & $-1.5 \cdot 10^9$ & $-5.4 \cdot 10^4$ & $-8.4 \cdot 10^8$ & $-1.0 \cdot 10^6$ & $-1.2 \cdot 10^8$ & $-1.9 \cdot 10^6$ & $-9.4 \cdot 10^6$ & $-9.1 \cdot 10^5$ \\
        \bottomrule
    \end{tabular}
    }
    
    \vspace{0.5em}  

    \resizebox{0.75\columnwidth}{!}{%
    \begin{tabular}{c|c|c|c|c|c|c|c|c|c|c}
        \toprule
         & \multicolumn{2}{c|}{$c_4$} & \multicolumn{2}{c|}{$c_5$} & \multicolumn{2}{c|}{$c_6$} & \multicolumn{2}{c|}{$c_7$} & \multicolumn{2}{c}{$c_8$} \\
        \midrule
        $C_9$ & $4.4 \cdot 10^5$ & $1.2 \cdot 10^4$ & $2.0 \cdot 10^5$ & $8.9 \cdot 10^4$ & $-1.3 \cdot 10^3$ & $-211$ & $-7.3 \cdot 10^2$ & $-657$ & $-0.5$ & $0$ \\
        \midrule
        $C_3^2$ & $6.2 \cdot 10^5$ & $6.1 \cdot 10^4$ & $1.1 \cdot 10^5$ & $3.5 \cdot 10^4$ & $-8.4 \cdot 10^2$ & $-346$ & $-4.6 \cdot 10^2$ & $-340$ & $-0.9$ & $0$ \\
        \bottomrule
    \end{tabular}
    }
    \caption{Mean (left) and median (right) of polynomial coefficients for nonic extensions.}
    \label{tab:nonic_poly}
\end{table}

Tables \ref{tab:galois_4} and \ref{tab:galois_9} summarize the performance of the models on quartic and nonic fields.

\subsection{Sextic fields}

For a degree \( 6 \) Galois extension \( K / \mathbb{Q} \) (a \emph{sextic} extension), there are two possible Galois groups: \( C_6 \) (cyclic) and \( S_3 \).  
In particular, the nonabelian group \( S_3 \) can arise as a Galois group, which makes the sextic case more interesting than the previously discussed cases of degree \( 4 \) and \( 9 \).

There has been various work on computing the Galois groups of sextic fields (including non-Galois extensions) from their defining polynomials \cite{cohen2007number,cavallo2019elementary,awtrey2015irreducible}.   
As noted at the beginning of this Section~\ref{sec:galois}, our interest lies in investigating how the Galois group can be determined from a finite number of Dedekind zeta function coefficients and polynomial coefficients, guided by the interpretation of machine learning (ML) experiments.

\subsubsection{Learning Galois groups from zeta coefficients}

\paragraph{Decision tree}

\begin{figure}[t]
    \centering
    \footnotesize
\begin{forest}
  label L/.style={edge label={node[midway,left,font=\scriptsize]{#1}}},
  label R/.style={edge label={node[midway,right,font=\scriptsize]{#1}}},
  for tree={
    font=\normalsize,
    forked edge,
    child anchor=north,
    for descendants={edge=->}
  }
  [{$a_{2^8}  = 0$}, draw, edge={very thick},
    [{$a_{2^6} \le 2$}, draw, label L=Y, edge={very thick},
      [{$C_6$}, rectangle, thick, draw, label L=Y, tier=bottom, line width=1.5pt]
      [{$a_{3^4} \le 5$}, draw, label R=N, edge={very thick}
        [{$a_{23^2} =0$}, draw, label L=Y, edge={very thick}
          [{${\cdots}$}, rectangle, draw, label L=Y, tier=bottom]
          [{$a_{19^2} \le 2$}, draw, label R=N, edge={very thick}
            [{$a_{13^2} \le 2$}, draw, label L=Y, edge={very thick}
              [{${\cdots}$}, rectangle, draw, label L=Y, tier=bottom]
              [{$a_{3^6} \le 2$}, draw, label R=N, edge={very thick}
                [{${\cdots}$}, rectangle, draw, label L=Y, tier=bottom]
                [{$a_{19^2} =0$}, draw, label R=N, edge={very thick}
                  [{${\cdots}$}, rectangle, draw, label L=Y, tier=bottom]
                  [{$C_6$}, rectangle, thick, draw, label R=N, tier=bottom, line width=1.5pt, edge={very thick}, fill=gray!40]
                ]
              ]
            ]
            [{${\cdots}$}, rectangle, draw, label R=N, tier=bottom]
          ]
        ]
        [{${\cdots}$}, rectangle, draw, label R=N, tier=bottom]
      ]
    ]
    [{${\cdots}$}, rectangle, draw, label R=N, tier=bottom]
  ]
\end{forest}
    \caption{A decision tree predicts Galois groups of sextic fields from zeta coefficients up to $a_{1000}(K)$.}
    \label{fig:sextic_zeta}
\end{figure}

A decision tree model trained on zeta coefficients with \( n = 1000 \) achieves an accuracy of 98.96\% in classifying the Galois groups of sextic fields.  
Similar to the quartic and nonic cases, we found that the top nodes of the tree correspond to square- or cube-indexed coefficients (see Figure~\ref{fig:sextic_zeta}, where the first few nodes are \( a_{2^8} \), \( a_{2^6} \), \( a_{3^6} \), \( a_{3^4} \), \( a_{18^2} \), \( a_{20^2} \), etc.).
Based on this observation, we trained a model using only the square- and cube-indexed coefficients up to \( n = 1000 \) (there are 38 such numbers), and obtained an even higher accuracy of 99.31\%.

When we realized importance of square- and cube-indexed coefficients in decision trees' performances, we performed exploratory data analysis (EDA) on the dataset we used\footnote{This can be reproduced, for small primes, with \texttt{verify\_galois.sage} in the GitHub repository.}. In particular, for each prime $p$, we investigated the possible values of $a_{p^2}(K)$ and $a_{p^3}(K)$ over all sextic Galois extension in \texttt{LMFDB}.
Interestingly, we found that the possible values of $a_{p^2}(K)$ depends on the Galois group \emph{and} $p \pmod{3}$.
Especially, we found that, when $p \equiv 1 \pmod{3}$, then $a_{p^2}(K) = 1$ only if $K / \bQ$ is a cyclic extension, and when $p \equiv 2 \pmod{3}$, then $a_{p^2}(K) = 1$ only if $K / \bQ$ is a sextic extension.
Also, we observed that $a_{p^3}(K) = 1$ only if $K / \bQ$ is a cyclic extension.

 In the remainder of this subsection, we prove the above observation (Corollary \ref{cor:sextic_zc}) and explain how it can be used to show that the bold branch of the tree (Figure \ref{fig:sextic_zeta}) give \emph{provably} correct predictions (Example \ref{ex:sextic_zeta_branch}). 
As in previous cases, we analyze the possible values of the zeta coefficients based on the decomposition types of primes.  
Table~\ref{tab:sexticeuler} lists the Euler factors and the corresponding zeta coefficients for each possible decomposition type.

\begin{table}[h]
    \begin{center}
        \begin{tabular}{c|c|c|c|c}
            \toprule
            $(e, f, g)$ & Euler factor & $a_p(K)$ & $a_{p^2}(K)$ & $a_{p^3}(K)$\\
            \midrule
            $(1, 1, 6)$ & $(1 - p^{-s})^{-6} = \sum_{k \ge 0}\binom{k+5}{5} p^{-ks}$ & 6 & 21 & 56 \\
            $(1, 2, 3)$ & $(1 - p^{-2s})^{-3} = \sum_{k \ge 0} \binom{k + 2}{2}p^{-2ks}$ & 0 & 3 & 0\\
            $(1, 3, 2)$ & $(1 - p^{-3s})^{-2} = \sum_{k \ge 0} (k + 1)p^{-3ks}$ & 0 & 0 & 2 \\
            $(1, 6, 1)$ & $(1 - p^{-6s})^{-1} = \sum_{k \ge 0} p^{-6ks}$ & 0 & 0 & 0 \\
            $(2, 1, 3)$ & $(1 - p^{-s})^{-3} = \sum_{k \ge 0} \binom{k+2}{2}p^{-ks}$ & 3 & 6 & 10 \\
            $(2, 3, 1)$ & $(1 - p^{-3s})^{-1} = \sum_{k \ge 0}p^{-3ks}$ & 0 & 0 & 1\\
            $(3, 1, 2)$ & $(1 - p^{-s})^{-2} = \sum_{k \ge 0} (k + 1) p^{-ks}$ & 2 & 3 & 4\\
            $(3, 2, 1)$ & $(1 - p^{-2s})^{-1} = \sum_{k \ge 0} p^{-2ks}$ & 0 & 1 & 0\\
            $(6, 1, 1)$ & $(1 - p^{-s})^{-1} = \sum_{k \ge 0} p^{-ks}$ & 1 & 1 & 1 \\
            \bottomrule
        \end{tabular}
        \caption{Decomposition types, Euler factors, $p$, $p^2$, $p^3$-th zeta coefficients of sextic extensions.}
        \label{tab:sexticeuler}
    \end{center}
\end{table}

\begin{proposition}
    \label{prop:C6ram}
    Let $K / \bQ$ be a sextic extension with $\Gal(K / \bQ) \simeq C_6$ and let $p$ be a prime with decomposition type $(e, f, g)$ in $K$.
    If $p \equiv 5 \pmod{6}$, then $(e, f, g) \not \in \{(3, 1, 2), (3, 2, 1), (6, 1, 1)\}$.
\end{proposition}
\begin{proof}
   
Assume \( (e, f, g) \in \{(3, 1, 2),\, (3, 2, 1),\, (6, 1, 1)\} \).   In all cases, the decomposition group has order \( ef = 3 \) or \( 6 \), and the corresponding extension \( K_\mathfrak{p} / \mathbb{Q}_p \) is a totally ramified \( C_3 \)- or \( C_6 \)-extension.   (For \( (e, f, g) = (3, 1, 2) \), we have \( (p) = \mathfrak{p}_1^3 \mathfrak{p}_2^3 \), and we choose either \( \mathfrak{p}_1 \) or \( \mathfrak{p}_2 \) as \( \mathfrak{p} \).) Let \( \tilde{K} \subset K_\mathfrak{p} \) be the unique subextension of degree~3 over \( \mathbb{Q}_p \).   Then \( \tilde{K} / \mathbb{Q}_p \) is a totally and tamely ramified \( C_3 \)-extension (note that $p\equiv 5 \pmod{6}$ implies $p \ne 2, 3$, so the extension is tame), and Proposition~\ref{prop:newton} implies that \( 3 \mid (p - 1) \), which contradicts the assumption \(  p \equiv 5 \pmod{6} \).
\end{proof}

\begin{lemma}
    \label{lem:p16_cyclic}
    Let $p \equiv 1 \pmod{6}$ be a prime. Then any degree 6 Galois extension of $\bQ_p$ is cyclic.
\end{lemma}
\begin{proof}
    Let \( K / \mathbb{Q}_p \) be a sextic Galois extension with \( G = \Gal(K / \mathbb{Q}_p) \simeq S_3 \).   Since \( (p, 6) = 1 \), the extension is tamely ramified, so the inertia group \( I \) is cyclic.   Then, \( G \) fits into the short exact sequence 
    \begin{equation}
        1 \to C_e \to G \to C_f \to 1, \label{eqn:tameram_exact_seq}
    \end{equation}
    and the assumption \( G \simeq S_3 \) implies \( (e, f) = (3, 2) \). Let \( \tau \) be a generator of the inertia group \( I \simeq C_3 \), and let \( \sigma \) be a lift of Frobenius, i.e., a generator of \( C_f \).   By Proposition~\ref{prop:tamefrobaction} and \( p \equiv 1 \pmod{3} \), we have \[ \sigma \tau \sigma^{-1} = \tau^p = \tau, \] which implies that \( \tau \) and \( \sigma \) commute.   Thus, \( G \simeq C_3 \times C_2 \simeq C_6 \), contradicting  \( G \simeq S_3 \).
\end{proof}

\begin{corollary}
    \label{cor:S3ram}
    Let $K / \bQ$ be a sextic extension with $G = \Gal(K / \bQ) \simeq S_3$ and let $p$ be a rational prime with decomposition type $(e, f, g)$ in $K$. Then: 
    \begin{enumerate}
        \item For any $p$, $(e, f, g) \not \in \{(1, 6, 1), (2, 3, 1)\}$.
        \item If $p \equiv 1 \pmod{6}$, then $(e, f, g) \not \in \{(3, 2, 1), (6, 1, 1)\}$.
        \item If $p \equiv 2 \pmod{3}$, then $(e, f, g) \not \in \{(3, 1, 2), (6, 1, 1)\}$.
    \end{enumerate}
\end{corollary}
\begin{proof}

1) The cases \( (e, f, g) = (1, 6, 1) \) and \( (2, 3, 1) \) cannot occur for any prime \( p \).  
Indeed, in either case,  \(G= D_\mathfrak{p} \) (for \( \mathfrak{p} \mid p \)) and \( D_\mathfrak{p} / I_\mathfrak{p} \simeq C_6 \) or \( C_3 \).  
However, this is impossible since \( S_3 \) admits no such quotients.

2) When $p \equiv 1 \pmod{6}$, Lemma \ref{lem:p16_cyclic} implies that the decomposition group cannot have order 6 with $G \cong S_3$, i.e. $ef \ne 6$. This excludes $(3, 2, 1)$ and $(6, 1, 1)$.

3) Lastly, assume \( p \equiv 5 \pmod{6} \).  
Then the decomposition type \( (e, f, g) = (6, 1, 1) \) is not possible, since the inertia group of a tamely ramified extension must be cyclic, and \( C_6 \) is not a subgroup of \( S_3 \).  
If \( (e, f, g) = (3, 1, 2) \), then \( K_\mathfrak{p} / \mathbb{Q}_p \) is a totally ramified \( C_3 \)-extension, which cannot exist when \( p \equiv 5 \pmod{6} \), as explained in the proof of Proposition~\ref{prop:C6ram}.

 For $p = 2$, assume it has a decomposition type of $(e, f, g) = (3, 1, 2)$.
Since $2 \nmid 3$, the corresponding local extension $K_\frp / \bQ_2$ is tamely ramified.
By Proposition \ref{prop:newton}, we should have $e \mid p - 1$, a contradiction.
If it has decomposition type $(e, f, g) = (6, 1, 1)$, one can consider a cubic subextension of the corresponding local extension, which reduces to the case $(e, f) = (3, 1)$.

\end{proof}

By Proposition~\ref{prop:C6ram}, Corollary~\ref{cor:S3ram}, and Table~\ref{tab:sexticeuler}, we obtain the following result.
\begin{corollary}
    \label{cor:sextic_zc}
    Let $K / \bQ$ be a sextic Galois extension and $p \ge 5$ be a prime.
    \begin{enumerate}
        \item If $p \equiv 1 \Mod{6}$, then the set of possible values of $a_{p^2}(K)$ is given by 
        \begin{equation}
            a_{p^2}(K) \in \begin{cases} \{0, 1, 3, 6, 21\} & \text{ if } \Gal(K/\bQ) \simeq C_6,  \\ \{0, 3, 6, 21\} & \text{ if } \Gal(K / \bQ) \simeq S_3. \end{cases}
        \end{equation}
        Especially, if there exists a prime $p$ with $p \equiv 1 \Mod{6}$ and $a_{p^2}(K) = 1$, then $\Gal(K / \bQ) \simeq C_6$.
        \item If $p \equiv 2 \Mod{3}$, then the set of possible values of $a_{p^2}(K)$ is given by
        \begin{equation}
            a_{p^2}(K) \in \begin{cases} \{0, 3, 6, 21\} & \text{ if }  \Gal(K/\bQ) \simeq C_6,\\ \{0, 1, 3, 6, 21\} & \text{ if } \Gal(K / \bQ) \simeq S_3. \end{cases}
        \end{equation}
        Especially, if there exists a prime $p$ with $p \equiv 5 \Mod{6}$ and $a_{p^2}(K) = 1$, then $\Gal(K / \bQ) \simeq S_3$.
        \item The set of possible values of $a_{p^3}(K)$ is given by
        \begin{equation}
            a_{p^3}(K) \in \begin{cases}
                \{0, 1, 2, 4, 10, 56\} & \text{ if } \Gal(K / \bQ) \simeq C_6,\\
                \{0, 2, 4, 10, 56\} & \text{ if }  \Gal(K / \bQ) \simeq S_3.
            \end{cases}
        \end{equation}
        Especially, if there exists a prime $p$ with $a_{p^3}(K) = 1$, then $\Gal(K / \bQ) \simeq C_6$.
    \end{enumerate}
\end{corollary}

\begin{example}[$C_6$, sufficient condition]
\label{ex:sextic_zeta_branch}
    Consider the bold branch in Figure \ref{fig:sextic_zeta}.
    If a sextic field satisfies the inequalities, then we have $ 0 < a_{19^2} \le  2$, hence $a_{19^2} = 1$  since $a_{19^2}$ cannot be 2 (Table \ref{tab:sexticeuler}).
    Since $19 \equiv 1 \pmod{6}$, Corollary \ref{cor:sextic_zc} implies that the field has to be a $C_6$-extension, which shows that the branch always gives a correct prediction.
\end{example}

One may ask whether Corollary \ref{cor:sextic_zc} can distinguish between $C_6$- and $S_3$-extensions by searching for primes $p$ such that $a_{p^2}(K)=1$. Table \ref{tab:galois_6} shows that $a_{p^2}(K)=1$ occurs only when the decomposition type of $p$ is $(3,2,1)$ or $(6,1,1)$, implying that $p$ must ramify in $K$. Consequently, for any Galois sextic extension $K$, verifying that $a_{p^2}(K)=1$ for a ramifying prime $p$ with $p \equiv 1 \pmod{6}$ suffices to conclude that $K$ is a $C_6$-extension. Table \ref{tab:sextic_psq_dist} presents the distribution of sextic fields according to the value of $a_{p^2}(K)$ for a fixed prime $p$ and Galois group $G\in \{C_6,S_3\}$. For example, about 7\% of $C_6$-extensions  on \texttt{LMFDB} satisfy $a_{19^2}(K)=1$. In such cases, the above procedure correctly identifies the Galois group.

\paragraph{Logistic regression}

Using the first 1000 zeta coefficients (resp. with standard normalization), the logistic regression model achieves an accuracy of 98.40\% (resp. 97.40\%) (Table~\ref{tab:galois_6}).   Table~\ref{tab:sextic_lr_weights} lists the indices with the largest and smallest weights, along with their corresponding values. Notably, most of the indices with large weights (except for 52 among the top 10) are products of prime powers with exponents that are multiples of 2 or 3.

\begin{table}[h]
    \begin{center}
        \resizebox{\columnwidth}{!}{%
        \begin{tabular}{c|c|c|c|c|c|c|c|c|c|c}
            \toprule
              top & \textbf{125} (14.4) & \textbf{25} (12.0) & \textbf{4} (7.8) & \textbf{512} (7.4) & 128 (6.9) & 32 (6.3) & \textbf{27} (5.5) & \textbf{100} (4.8) & \textbf{36} (4.3) & \textbf{196} (4.2) \\
            \cline{1-11}
             bottom & \textbf{625} (-20.9) & \textbf{16} (-15.5) & 243 (-9.5) & \textbf{64} (-9.3) & 496 (-6.6) & 5 (-6.4) & 112 (-6.0) & 304 (-6.0) & \textbf{400} (-5.9) & 48 (-5.9) \\
            \bottomrule
        \end{tabular}
        }
        \caption{Indices of the 10 largest and smallest weights (with actual weight values in parentheses) among 1000 logistic regression coefficients used to classify Galois groups of sextic fields  with normalized zeta coefficients. The bold numbers correspond to indices that are products of prime powers with exponents that are multiples of 2 or 3.}
        \label{tab:sextic_lr_weights}
    \end{center}
\end{table}

As in the case of quartic and nonic fields (Section~\ref{subsec:galois_quartic_nonic}), we also compared the distribution of the absolute values of model weights for indices divisible by squares or cubes with the overall distribution.  
We found that the former is, on average, higher than the latter (Figure~\ref{fig:galois_6_10_lr_weights} (a) and (b)).  
The average absolute value of all weights is  0.63, while the average for indices divisible by square or cube factors is  3.94.
This indicates that the logistic regression model aligns with the decision tree model in identifying the importance of such features.

\subsubsection{Learning Galois groups from polynomial coefficients}
\label{sec:sextic_poly}

As in the case of quartic and nonic fields (Section \ref{subsec:galois_quartic_nonic}), the decision tree model achieved a high accuracy of  97.56\% with polynomial coefficients, although the trained model is very complicated and not easy to interpret the meaning of nodes.
However, we found that it tends to classify a given sextic field as a $S_3$-extension when coefficients are ``small''.
For example, there are 125,734 fields in the training set, where 17,029 (resp. 108,785) of them have a Galois group $C_6$ (resp. $S_3$).
The first node of the tree asks whether $c_2 \le 520,127$ or not, and 108442 of $S_3$-extensions (99.76\% among 108785 $S_3$-extensions) satisfy the condition, while 5831 $C_6$-extensions (34.24\% among 17,029 $C_6$-extensions) satisfy the opposite condition of $c_2 > 520,127$.
Table~\ref{tab:sextic_nonab_poly} shows the mean and median of the polynomial coefficients \( c_0, c_1, \dots, c_5 \) for each Galois group.  
The magnitudes \( |c_0|, |c_1|, \dots, |c_5| \) are, on average, much larger for \( C_6 \)-extensions.

In contrast, the logistic regression model performed very poorly when using polynomial coefficients, achieving an accuracy of only  26.12\%. Similarly to the discussion at the end of Section \ref{subsubsec:49_poly}, this is due to the large magnitude of polynomial coefficients. After normalization the model achieves  86.73\% accuracy, which is slightly above the majority baseline (86.55\%, Table \ref{tab:nf_count_galois}). In particular, the balanced accuracy is only 50.72\%, where the confusion matrix (d) of Figure \ref{fig:galois_6_cm} shows that logistic regression models' predictions are highly biased towards $S_3$.

\begin{table}[h]
    \begin{center}
        \resizebox{\columnwidth}{!}{%
        \begin{tabular}{c|c|c|c|c|c|c|c|c|c|c|c|c}
            \toprule
             & \multicolumn{2}{c|}{$c_0$} & \multicolumn{2}{c|}{$c_1$} & \multicolumn{2}{c|}{$c_2$} & \multicolumn{2}{c|}{$c_3$} & \multicolumn{2}{c|}{$c_4$} & \multicolumn{2}{c}{$c_5$}  \\
            \midrule
            $C_6$ & $4.4 \cdot 10^{17}$ & $1.5 \cdot 10^5$ & $-3.7 \cdot 10^{14}$ & $-2.3 \cdot 10^3$ & $7.6 \cdot 10^{10}$ & $1.3 \cdot 10^5$ & $9.1 \cdot 10^6$ & $-2$ & $-3.1 \cdot 10^3$ & $-77$ & $-1.1$ & $-1$ \\
            \midrule
            $S_3$ & $2.2 \cdot 10^{9}$ & $7.2 \cdot 10^5$ & $-2.4 \cdot 10^{6}$ & $0$ & $2.0 \cdot 10^{5}$ & $2.3 \cdot 10^4$& $8.6 \cdot 10^2$ & $0$ & $-1.4 \cdot 10^2$ & $-50$ & $-0.7$ & $0$ \\
            \bottomrule
        \end{tabular}
        }
        \caption{Mean (left) and median (right) of polynomial coefficients of sextic extensions.}
        \label{tab:sextic_nonab_poly}
    \end{center}
\end{table}

Table \ref{tab:galois_6} summarizes the results for sextic extensions.

\subsection{Decic fields}

As in the sextic case, there are two possible Galois groups for decic Galois extensions \( K/\mathbb{Q} \): \( C_{10} \) (cyclic, abelian) and \( D_5 \) (dihedral, nonabelian).  
There are few prior works that focus on computing Galois groups of decic extensions, but several studies consider number fields of higher degrees, including~\cite{kluners2001database, fieker2014computation, cohen2007number}.

\subsubsection{Learning Galois groups from zeta coefficients}

\paragraph{Decision tree}

\begin{figure}[t]
    \centering
    \footnotesize
    \begin{forest}
      label L/.style={edge label={node[midway,left,font=\scriptsize]{#1}}},
      label R/.style={edge label={node[midway,right,font=\scriptsize]{#1}}},
      for tree={font=\normalsize,forked edge, child anchor=north, for descendants={edge=->}}
      [{$a_{2^6} \le 17$}, draw
        [{$a_{2^5 \cdot 3^2} \le 7$}, draw, label L=Y, edge=very thick
          [{$a_{2^5 \cdot 5^2} \le 8$}, draw, label L=Y, edge=very thick
            [{$a_{31^2} \le 4$}, draw, label L=Y, edge=very thick
              [{${\cdots}$}, rectangle, draw, label L=Y, tier=bottom]
              [{$a_{2^5} \le 1$}, draw, label R=N, edge=very thick
                [{$C_{10}$}, rectangle, thick, draw, label L=Y, tier=bottom, line width=1.5pt, edge=very thick, fill=gray!40]
                [{${\cdots}$}, rectangle, draw, label R=N, tier=bottom]
              ]
            ]
            [{$a_{3^5} \le 1$}, draw, label R=N, edge=very thick
              [{$C_{10}$}, rectangle, thick, draw, label L=Y, tier=bottom, line width=1.5pt]
              [{$a_{2^5} \le 1$}, draw, label R=N, edge=very thick
                [{$C_{10}$}, rectangle, thick, draw, label L=Y, tier=bottom, line width=1.5pt, edge=very thick, fill=gray!40]
                [{${\cdots}$}, rectangle, draw, label R=N, tier=bottom]
              ]
            ]
          ]
          [{$a_{2^5} \le 1$}, draw, label R=N, edge=very thick
            [{$C_{10}$}, rectangle, thick, draw, label L=Y, tier=bottom, line width=1.5pt, edge=very thick, fill=gray!40]
            [{${\cdots}$}, rectangle, draw, label R=N, tier=bottom]
          ]
        ]
        [{$a_{3^6} \le 17$}, draw, label R=N
          [{$a_{3^5} \le 1$}, draw, label L=Y
            [{$C_{10}$}, rectangle, thick, draw, label L=Y, tier=bottom, line width=1.5pt]
            [{${\cdots}$}, rectangle, draw, label R=N, tier=bottom]
          ]
          [{$a_{2 \cdot 3^4} \le 5362$}, draw, label R=N
            [{$a_{149} \le 7.5$}, draw, label L=Y
              [{$D_5$}, rectangle, thick, draw, label L=Y, tier=bottom, line width=1.5pt]
              [{$a_{2^2 \cdot 229} \le 37$}, draw, label R=N
                [{$D_5$}, rectangle, thick, draw, label L=Y, tier=bottom, line width=1.5pt]
                [{$a_{2^2 \cdot 3^2} \le 50$}, draw, label R=N
                  [{$C_{10}$}, rectangle, thick, draw, label L=Y, tier=bottom, line width=1.5pt]
                  [{$D_5$}, rectangle, thick, draw, label R=N, tier=bottom, line width=1.5pt]
                ]
              ]
            ]
            [{$C_{10}$}, rectangle, thick, draw, label R=N, tier=bottom, line width=1.5pt]
          ]
        ]
      ]
    \end{forest}
    \caption{A decision tree predicts Galois groups of decic fields from zeta coefficients up to $a_{1000}(K)$.}
    \label{fig:decic_zeta}
\end{figure}

A decision tree model trained on the first 1000 zeta coefficients achieves an accuracy of 97.19\% in classifying the Galois groups of decic fields.  
Figure~\ref{fig:decic_zeta} shows the resulting tree, where the top nodes correspond to indices divisible by squares or fifth powers.  
Based on this observation, we trained a decision tree using only the 114 indices up to \( n = 1000 \) that are divisible by squares or fifth powers, and obtained a higher accuracy of 98.56\%.

As in previous cases, we analyze the possible values of the zeta coefficients based on the decomposition types of primes.
 In particular, by performing EDA as in the sextic extension case, we found that $a_{p^2}(K)$ can be 1 only if $K$ has a certain Galois group, \emph{depending on $p \equiv 1 \pmod{10}$}.
Table~\ref{tab:deciceuler} lists the Euler factors and the $p^2$- and  $p^5$-th zeta coefficients corresponding to each possible decomposition type.
 The observation is proved in Corollary \ref{cor:decic_zeta}, and this can be used to prove that the bold branches of Figure \ref{fig:decic_zeta} give \emph{provably correct predictions} (see Example \ref{ex:decic_zeta_branch}).

\begin{table}[h]
    \begin{center}
        \begin{tabular}{c|c|c|c|c}
            \toprule
            $(e, f, g)$ & Euler factor & $a_{p}(K)$ & $a_{p^2}(K)$ & $a_{p^5}(K)$   \\
            \midrule
            $(1, 1, 10)$ & $(1 - p^{-s})^{-10} = \sum_{k \ge 0} \binom{k+9}{9} p^{-ks}$ & 10 & 55 & 2002 \\
            $(1, 2, 5)$ & $(1 - p^{-2s})^{-5} = \sum_{k \ge 0} \binom{k+4}{4} p^{-2ks}$ & 0 & 5 & 0\\
            $(1, 5, 2)$ & $(1 - p^{-5s})^{-2} = \sum_{k \ge 0} (k+1)p^{-5ks}$ & 0 & 0 & 2\\
            $(1, 10, 1)$ & $(1 - p^{-10s})^{-1} = \sum_{k \ge 0} p^{-10ks}$ & 0 & 0 & 0 \\
            $(2, 1, 5)$ & $(1 - p^{-s})^{-5} = \sum_{k \ge 0} \binom{k+4}{4} p^{-ks}$ & 5 &  15 & 126 \\
            $(2, 5, 1)$ & $(1 - p^{-5s})^{-1} = \sum_{k \ge 0} p^{-5ks}$ & 0 & 0 & 1 \\
            $(5, 1, 2)$ & $(1 - p^{-s})^{-2} = \sum_{k \ge 0} (k+1)p^{-ks}$ & 2 & 3 & 6 \\
            $(5, 2, 1)$ & $(1 - p^{-2s})^{-1} = \sum_{k \ge 0} p^{-2ks}$ & 0 & 1 & 0 \\
            $(10, 1, 1)$ & $(1 - p^{-s})^{-1} = \sum_{k \ge 0} p^{-ks}$ & 1 & 1 & 1\\
            \bottomrule
        \end{tabular}
        \caption{Decomposition types, Euler factors, and $p$, $p^2$, $p^5$-th zeta coefficients of decic extensions.}
        \label{tab:deciceuler}
    \end{center}
\end{table}

\begin{proposition}
\label{prop:C10ram}
    Let $K / \bQ$ be a decic Galois extension with Galois group $\Gal(K / \bQ) \simeq C_{10}$ and $p$ be a rational prime.
Denote by $(e, f, g)$ the decomposition type of $p$ in $K$.
    If $p = 2$ or $p \equiv 3, 7, 9 \pmod{10}$, then $(e, f, g) \not\in  \{(5, 1, 2), (5, 2, 1), (10, 1, 1)\}$.
\end{proposition}
\begin{proof}
    The proof is similar to that of Proposition \ref{prop:C6ram} and Corollary \ref{cor:S3ram}.
    
    For $p = 2$ and $(e, f, g) = (5, 1, 2)$, assume such extension exists.
    Since $2 \nmid 5$, the corresponding local extension $K_\frp / \bQ_2$ is tamely ramified, and Proposition \ref{prop:newton} implies $e \mid p - 1$, a contradiction.
    For $(e, f, g) = (5, 2, 1)$ or $(10, 1, 1)$, one can consider the unique quintic subextension of the local extension, and reduce to the previous case $(e, f) = (5, 1)$.

    For $p \equiv 3,7,9 \pmod{10}$, assume $(e, f, g) \in \{(5, 1, 2), (5, 2, 1), (10, 1, 1)\}$.
    Then the corresponding extension $K_\frp$ of $\bQ_p$ (for some prime $\frp | p$) is totally ramified $C_{5}$- or $C_{10}$-extension. In either case, there exists a totally ramified $C_{5}$-subextension $\tilde K \subset K_\frp$.
    Since $(p, 5) = 1$, the extension $\tilde K / \mathbb Q_p$ is tamely ramified, and Proposition \ref{prop:newton} implies that $p \equiv 1 \pmod{5}$, which contradicts the assumption $p \equiv 3,7,9 \pmod{10}$.
\end{proof}

\begin{lemma}
    \label{lem:p110_cyclic}
    Let $p \equiv 1 \pmod{10}$ be a prime. Then any decic Galois extension of $\bQ_p$ is cyclic.    
\end{lemma}
\begin{proof}
    The proof is similar to that of Lemma \ref{lem:p16_cyclic}.
     If $K / \bQ_p$ is a decic Galois extension with $G = \Gal(K / \bQ_p) \simeq D_{10}$, then the extension is tamely ramified from $(p, 10) = 1$, hence the inertia group $I$ is cyclic. By considering the same exact sequence \eqref{eqn:tameram_exact_seq} in Lemma \ref{lem:p16_cyclic}, we should have $(e, f) = (5, 2)$. Now applying Proposition \ref{prop:tamefrobaction} gives a contradiction. 
\end{proof}

\begin{proposition}
    \label{prop:D5ram}
    Let $K / \bQ$ be a decic $D_5$-extension and $p$ be a rational prime.
    Let $(e, f, g)$ be the decomposition type of $p$ in $K$. Then:
    \begin{enumerate}
        \item If $p = 2$, then $(e, f, g) \not \in \{(1, 10, 1), (2, 5, 1), (5, 1, 2), (5, 2, 1), (10, 1, 1)\}$.
        \item If $p = 5$, then $(e, f, g) \not\in \{(1, 10, 1), (2, 5, 1)\}$.
        \item If $p \equiv 1 \pmod{10}$, then $(e, f, g) \not \in \{(1, 10, 1), (2, 5, 1), (5, 2, 1), (10, 1, 1)\}$.
        \item If $p \equiv 3, 7 \pmod{10}$, then $(e, f, g) \not\in  \{(1, 10, 1), (2, 5, 1), (5, 1, 2), (5, 2, 1), (10, 1, 1)\}$.
        \item If $p \equiv 9 \pmod{10}$, then $(e, f, g) \not\in  \{(1, 10, 1), (2, 5, 1), (5, 1, 2), (10, 1, 1)\}$.
    \end{enumerate}
\end{proposition}
\begin{proof}

The proof is similar to that of Corollary~\ref{cor:S3ram}.  

First, the case of $(1, 10, 1)$ can be easily excluded for all $p$, since $C_{10}$ cannot arise as a quotient of $D_5$.
Also, if a prime $p$ has decomposition type $(2, 5, 1)$, then the decomposition group is isomorphic to $D_5$ and has a normal subgroup of order $2$. This is impossible, since any such normal subgroup must lie in the center, whereas the center of $D_5$ is trivial as $5$ is odd.

If \( p \equiv 3, 7, 9 \pmod{10} \), then \( (e, f, g) \ne (10, 1, 1) \), since the inertia group of a tamely ramified extension must be cyclic, but \( C_{10} \) is not a subgroup of \( D_5 \).  
Similarly, the decomposition types \( (1, 10, 1) \) and \( (2, 5, 1) \) are excluded, as neither \( C_{10} \) nor \( C_5 \) can occur as a quotient of \( D_5 \).
If \( (e, f, g) = (5, 1, 2) \), then the corresponding extension \( K_\mathfrak{p} / \mathbb{Q}_p \) is a totally and tamely ramified \( C_5 \)-extension and Proposition~\ref{prop:newton} gives $p \equiv 1 \pmod{5}$.
Thus \( (e, f, g) = (5, 1, 2) \) is excluded for $p \equiv 3, 7, 9 \pmod{10}$.
If $(e, f, g) = (5, 2, 1)$, a corresponding $D_5$-extension $K_\frp / \bQ_p$ has an unramified quadratic extension $\widetilde{K} \subset K_\frp$ where $K_\frp / \widetilde{K}$ is a totally and tamely ramified $C_5$-extension.
Since the residue field of $\widetilde{K}$ is $\bF_{p^2}$, Proposition \ref{prop:newton} implies that $p^2 \equiv 1 \pmod{5}$, hence $p \not\equiv 3, 7 \pmod{10}$.
\end{proof}

As a result, we obtain the following information about the values of $a_{p^2}(K)$ and $a_{p^5}(K)$ for a decic field $K$.

\begin{corollary} \label{cor:decic_zeta} \hfill
    \begin{enumerate}
        \item If $p \equiv 1 \Mod{10}$, then the set of possible values of $a_{p^2}(K)$ is given by 
        \begin{equation}
            a_{p^2}(K) \in \begin{cases}
                \{0, 1, 3, 5, 15, 55\} & \text{ if } \Gal(K / \bQ) \simeq C_{10},\\
                \{0, 3, 5, 15, 55\} & \text{ if }  \Gal(K / \bQ) \simeq D_5.
            \end{cases}
        \end{equation}
        Especially, if there exists a prime $p$ with $p \equiv 1 \Mod{10}$ and $a_{p^2}(K) = 1$, then $\Gal(K / \bQ) \simeq C_{10}$.
        \item If $p \equiv 9 \Mod{10}$, then the set of possible values of $a_{p^2}(K)$ is given by
        \begin{equation}
            a_{p^2}(K) \in \begin{cases}
                \{0, 5, 15, 55\} & \text{ if }  \Gal(K / \bQ) \simeq C_{10},\\
                \{0, 1, 5, 15, 55\} & \text{ if }  \Gal(K / \bQ) \simeq D_5.
            \end{cases}
        \end{equation}
        Especially, if there exists a prime $p$ with $p \equiv 9 \Mod{10}$ and $a_{p^2}(K) = 1$, then $\Gal(K / \bQ) \simeq D_{5}$.
        \item For $p \ne 5$, the set of possible values of $a_{p^5}(K)$ is given by
        \begin{equation}
            a_{p^5}(K) \in \begin{cases} \{0, 1, 2, 126, 2002\} & \text{ if }  \mathrm{Gal}(K/\mathbb{Q}) \simeq C_{10},\\ \{0, 2, 126, 2002\} & \text{ if }  \mathrm{Gal}(K / \mathbb{Q}) \simeq D_5.\end{cases}
        \end{equation}
        Especially, if there exists a prime $p \ne 5$ such that $a_{p^5}(K) = 1$, then $\Gal(K / \bQ) \simeq C_{10}$.
    \end{enumerate}

\end{corollary}

\begin{example}[$C_{10}$, sufficient condition]
\label{ex:decic_zeta_branch}
    Assume that a decic extension $K / \bQ$ follows one of the bold branches in Figure \ref{fig:decic_zeta}.
    Then $a_{2^6}(K) \le 17$ and $a_{2^5}(K) \le  1$. If \( (e, f, g) = (1, 2, 5) \), then by Corollary~\ref{cor:prime_pow_zeta_coeff}, we have \( a_{2^6}(K) = 35 > 17 \), so this decomposition type cannot occur.  Furthermore, from Table \ref{tab:deciceuler}, the condition $a_{2^5}(K) \le 1$ implies $(e, f, g) \in \{(1, 10, 1), (2, 5, 1), (5, 2, 1), (10, 1, 1)\}$, which cannot be the case when $K / \bQ$ is a $D_{5}$-extension by Proposition \ref{prop:D5ram}.
    Hence $K / \bQ$ must be a cyclic extension.
\end{example}

\paragraph{Logistic regression}

Using 1000 zeta coefficients (resp. with standard normalization), the logistic regression model achieves an accuracy of 93.81\% (reps. 94.70\%)  (Table~\ref{tab:galois_10}).  
Table~\ref{tab:decic_lr_weights} lists the indices with the largest and smallest weights, along with their corresponding values.  
One can observe that all the top ten indices with largest weights correspond to square indices.
This reflects the importance of square- and quintic-indexed coefficients in the classification.

\begin{table}[H]
    \begin{center}
        \resizebox{\columnwidth}{!}{%
        \begin{tabular}{c|c|c|c|c|c|c|c|c|c|c}
            \toprule
             top & \textbf{25} (3.9) & \textbf{49} (3.4) & \textbf{961} (2.4) & \textbf{121} (2.3) & \textbf{9} (2.3) & \textbf{841} (1.7) & \textbf{529} (1.7) & \textbf{289} (1.6) & \textbf{169} (1.4) & \textbf{361} (1.3) \\
            \cline{1-11}
             bottom & 343 (-2.7) & 31 (-2.1) & 11 (-2.1) & 5 (-1.6) & 125 (-1.6) & 17 (-1.5) & 29 (-1.5) & 23 (-1.3) & 13 (-1.3) & 19 (-1.2) \\
            \bottomrule
        \end{tabular}
        }
        \caption{Indices of the 10 largest and smallest weights (with actual weight values in parentheses) among 1000 logistic regression coefficients used to classify Galois groups of decic fields  with normalized zeta coefficients. The bold numbers correspond to indices that are squares.}
        \label{tab:decic_lr_weights}
    \end{center}
\end{table}

We also compared the distribution of the absolute values of model weights for indices divisible by squares or fifth powers with the overall distribution, and found that the former is higher on average (Figure~\ref{fig:galois_6_10_lr_weights} (c) and (d)).  
The average of all weights is  0.22, while the average of the weights for indices with square or quintic factors is  0.99.
This indicates that the logistic regression model is able to learn the importance of such coefficients for prediction, in agreement with the decision tree model.

\subsubsection{Learning Galois groups from polynomial coefficients}

Using polynomial coefficients, the decision tree model achieved an accuracy of 86.71\%.  
As in previous cases, the model's performance may stem from discrepancies in the distributions of the coefficients across Galois groups (Table~\ref{tab:decic_poly}).  
For example, the first node asks whether \( c_2 \le 8{,}101{,}335 \), and approximately 99\% of the \( D_5 \)-extensions in the training set satisfy this condition.  
In contrast, the logistic regression model achieved only 46.25\% accuracy, which improves to 66.18\% after normalization, close to the proportion of \(D_5\)-extensions. Recall the discussion at the end of Section~\ref{subsubsec:49_poly}.

\begin{table}[h]
\centering
    \resizebox{0.85\columnwidth}{!}{%
    \begin{tabular}{c|c|c|c|c|c|c|c|c|c|c}
        \toprule
         & \multicolumn{2}{c|}{$c_0$} & \multicolumn{2}{c|}{$c_1$} & \multicolumn{2}{c|}{$c_2$} & \multicolumn{2}{c|}{$c_3$} & \multicolumn{2}{c}{$c_4$} \\
        \midrule
        $C_{10}$ & $3.3\cdot 10^{19}$ & $1.8 \cdot 10^5$ & $-1.0 \cdot 10^{18}$ & $-5.0 \cdot 10^5$ & $7.7 \cdot 10^{15}$ & $1.5 \cdot 10^7$ & $3.7 \cdot 10^{13}$ & $0$ & $-5.7 \cdot 10^{11}$ & $7.8 \cdot 10^2$  \\
        \midrule
        $D_5$ & $8.1 \cdot 10^6$ & $8.1 \cdot 10^5$ & $-2.0 \cdot 10^5$ & $0$ & $7.7 \cdot 10^5$ & $1.3 \cdot 10^5$ & $-6.0 \cdot 10^3$ & $0$ & $2.9 \cdot 10^4$ & $1.4 \cdot 10^4$ \\
        \bottomrule
    \end{tabular}
    }
    
    \vspace{0.5em}  

    \resizebox{0.75\columnwidth}{!}{%
    \begin{tabular}{c|c|c|c|c|c|c|c|c|c|c}
        \toprule
          & \multicolumn{2}{c|}{$c_5$} & \multicolumn{2}{c|}{$c_6$} & \multicolumn{2}{c|}{$c_7$} & \multicolumn{2}{c|}{$c_8$} & \multicolumn{2}{c}{$c_9$} \\
        \midrule
        $C_{10}$ & $-2.7 \cdot 10^8$ & $-8.4 \cdot 10^3$ & $1.6 \cdot 10^7$ & $2.5 \cdot 10^4$ & $-4.5 \cdot 10^2$ & $0$ & $-2.5 \cdot 10^2$ & $-55$ & $-1.6$ & $-1$ \\
        \midrule
        $D_5$ & $-1.0 \cdot 10^3$ & $0$ & $1.7 \cdot 10^3$ & $7.0 \cdot 10^2$ & $-9.2$ & $0$ & $-1.2$ & $-4$ & $-1.3$ & $0$ \\
        \bottomrule
    \end{tabular}
    }
    \caption{Mean (left) and median (right) of polynomial coefficients for  decic extensions.}
    \label{tab:decic_poly}
\end{table}

Table \ref{tab:galois_10} summarizes the results for decic extensions.

\subsection{Octic fields}

For degree 8 Galois extensions, there are five possible Galois groups: three  abelian groups ($C_8$, $C_4 \times C_2$ and $C_2^3$) and two are nonabelian groups ($D_4$ and $Q_8$).

There are several works that study the Galois groups of octic fields. Most of these focus on fields defined by specific types of degree~8 polynomials, although they are not limited to Galois extensions. For example, Chen, Chin, and Tan provide a complete classification of possible Galois groups for \emph{doubly even octic polynomials} (i.e., \( f(x) = x^8 + ax^4 + b \)) and \emph{palindromic even octic polynomials} (i.e., \( f(x) = x^8 + ax^6 + bx^4 + ax^2 + 1 \) with \( a \ne 0 \)) by analyzing linear resolvents~\cite{chen2023galois}. Awtrey and Patane generalized their result for palindromic even octic polynomials over \( \mathbb{Q} \) to arbitrary characteristic zero fields~\cite{awtrey2024galois}. However, these results are not enough to give a complete characterization of Galois groups for arbitrary defining polynomials; for example, the quaternion group cannot occur as the Galois group of an octic field defined by a doubly even octic polynomial~\cite[Theorem 3.1]{chen2023galois}.

\subsubsection{Distinguishing between abelian extensions}
\label{subsubsec:octic_ab}

\begin{figure}[t]
    \centering
    \footnotesize
    \begin{forest}
      label L/.style={edge label={node[midway,left,font=\scriptsize]{#1}}},
      label R/.style={edge label={node[midway,right,font=\scriptsize]{#1}}},
      for tree={forked edge, child anchor=north, for descendants={edge=->}}
      [{$a_{2^2 \cdot 3^2 \cdot 5^2} \le 7$}, draw
        [{$a_{5^4} =0$}, draw, label L=Y, edge=very thick
          [{$C_8$}, rectangle, thick, draw, label L=Y, tier=bottom, line width=1.5pt, edge=very thick, fill=gray!40]
          [{$a_{3^4} =0$}, draw, label R=N, edge=very thick
            [{$C_8$}, rectangle, thick, draw, label L=Y, tier=bottom, line width=1.5pt, edge=very thick, fill=gray!40]
            [{${\cdots}$}, rectangle, draw, label R=N, tier=bottom]
          ]
        ]
        [{$a_{23^2} \le 1$}, draw, label R=N, edge={very thick, densely dashed}
          [{$a_{2^2 \cdot 3^2 \cdot 5^2} \le 342$}, draw, label L=Y
            [{$a_{2 \cdot 3^2 \cdot 5^2} \le 104$}, draw, label L=Y
              [{$a_{2 \cdot 389} \le  12$}, draw, label L=Y
                [{${\cdots}$}, rectangle, draw, label L=Y, tier=bottom]
                [{$a_{881} \le 4$}, draw, label R=N
                  [{$a_{2^4 \cdot 5^3} \le 140$}, draw, label L=Y
                    [{$C_8$}, rectangle, thick, draw, label L=Y, tier=bottom, line width=1.5pt]
                    [{$C_4 \times C_2$}, rectangle, thick, draw, label R=N, tier=bottom, line width=1.5pt]
                  ]
                  [{$C_4 \times C_2$}, rectangle, thick, draw, label R=N, tier=bottom, line width=1.5pt]
                ]
              ]
              [{${\cdots}$}, rectangle, draw, label R=N, tier=bottom]
            ]
            [{${\cdots}$}, rectangle, draw, label R=N, tier=bottom]
          ]
          [{$a_{31^2} \le  1$}, draw, label R=N, edge={very thick, densely dashed}
            [{${\cdots}$}, rectangle, draw, label L=Y, tier=bottom]
            [{$a_{29^2} \le  1$}, draw, label R=N, edge={very thick, densely dashed}
              [{$C_4 \times C_2$}, rectangle, thick, draw, label L=Y, tier=bottom, line width=1.5pt]
              [{$a_{17^2} \le 1$}, draw, label R=N, edge={very thick, densely dashed}
                [{$C_4 \times C_2$}, rectangle, thick, draw, label L=Y, tier=bottom, line width=1.5pt]
                [{$a_{11^2} \le 1$}, draw, label R=N, edge={very thick, densely dashed}
                  [{$C_4 \times C_2$}, rectangle, thick, draw, label L=Y, tier=bottom, line width=1.5pt]
                  [{$a_{7^2} \le 1$}, draw, label R=N, edge={very thick, densely dashed}
                    [{$C_4 \times C_2$}, rectangle, thick, draw, label L=Y, tier=bottom, line width=1.5pt]
                    [{$C_2^3$}, rectangle, thick, draw, label R=N, tier=bottom, line width=1.5pt, edge={very thick, densely dashed}, fill=gray!20]
                  ]
                ]
              ]
            ]
          ]
        ]
      ]
    \end{forest}
    \caption{A decision tree predicts Galois groups of abelian octic fields from zeta coefficients up to $a_{1000}(K)$.}
    \label{fig:octic_ab_zeta_dt}
\end{figure}

In~\cite[Section~5.2]{he2022machine}, the authors found that a random forest classifier trained on zeta coefficients up to \( n = 1000 \) can distinguish between different \emph{abelian} octic extensions with high accuracy (95.47\%). The three possible Galois groups are \( C_8 \), \( C_4 \times C_2 \), and \( C_2^3 \).
 We use a single decision tree for the classification task and find that it achieves a higher accuracy of  98.61\%. As in the case of quartic extensions, we observe that the first few nodes of the decision tree focus on square-indexed zeta coefficients. Motivated by this, we train a decision tree using only square-indexed coefficients with \( n \le 1000 \), which results in an even higher accuracy of 98.91\%.

Figure~\ref{fig:octic_ab_zeta_dt} shows that the model predicts the Galois group to be \( C_8 \) whenever a fourth-power-indexed coefficient \( a_{m^4}(K) \) vanishes for some  \( m \). This strongly suggests that the Galois group is cyclic if there exists a prime \( p \) such that \( a_{p^4}(K) = 0 \). Similarly, the right part of Figure~\ref{fig:octic_ab_zeta_dt} shows that the model predicts the Galois group to be \( C_4 \times C_2 \) when the field is not cyclic and the square-indexed coefficients are small. Motivated by these observations, we studied the distribution of \( a_{p^2}(K) \) and \( a_{p^4}(K) \) for fixed primes \( p \), and formulated a conjecture, which we are able to prove rigorously (Corollary~\ref{cor:octic_zeta_ab}). The proof follows from a couple of theorems established below.

\begin{table}[h]
    \begin{center}
        \begin{tabular}{c|c|c|c}
            \toprule
            $(e, f, g)$ & Euler factor & $a_{p^\ell}(K)$ & $a_{p^{\ell^2}}(K)$ \\
            \midrule
            $(1, 1, \ell^3)$ & $(1 - p^{-s})^{-\ell^3} = \sum_{k \ge 0} \binom{\ell^3 + k - 1}{\ell^3 - 1}p^{-ks}$ & $\binom{\ell^3 + \ell - 1}{\ell^3 - 1}$ & $\binom{\ell^3 + \ell^2 - 1}{\ell^3 - 1}$ \\
            $(1, \ell, \ell^2)$ & $(1 - p^{-\ell s})^{-\ell^2} = \sum_{k \ge 0} \binom{\ell^2 + k - 1}{\ell^2 - 1} p^{-\ell k s}$ & $\ell^2$ & $\binom{\ell^2 + \ell - 1}{\ell^2 - 1}$ \\
            $(1, \ell^2, \ell)$ & $(1 - p^{-\ell^2 s})^{-\ell} = \sum_{k \ge 0} \binom{\ell + k - 1}{\ell - 1} p^{-\ell^2 k s}$ & $0$ & $\ell + 1$ \\
            $(1, \ell^3, 1)$ & $(1 - p^{-\ell^3 s})^{-1} = \sum_{k \ge 0}p^{-\ell^3 k s}$ & $0$ & $0$ \\
            $(\ell, 1, \ell^2)$ & $(1 - p^{-s})^{-\ell^2} = \sum_{k \ge 0} \binom{\ell^2 + k - 1}{\ell^2- 1} p^{-ks}$ & $\binom{\ell^2 + \ell - 1}{\ell^2 - 1}$ & $\binom{2\ell^2 - 1}{\ell^2 - 1}$ \\
            $(\ell, \ell, \ell)$ & $(1 - p^{-\ell s})^{-\ell} = \sum_{k \ge 0} \binom{\ell + k - 1}{\ell - 1} p^{-\ell ks}$ & $\ell$ & $\binom{2\ell - 1}{\ell - 1}$ \\
            $(\ell, \ell^2, 1)$ & $(1 - p^{-\ell^2 s})^{-1} = \sum_{k \ge 0} p^{-\ell^2 ks}$ & $0$ & $1$ \\
            $(\ell^2, 1, \ell)$ & $(1 - p^{-s})^{-\ell} = \sum_{k \ge 0} \binom{\ell + k - 1}{\ell - 1} p^{-ks}$ & $\binom{2\ell - 1}{\ell - 1}$ & $\binom{\ell + \ell^2 - 1}{\ell - 1}$\\
            $(\ell^2, \ell, 1)$ & $(1 - p^{-\ell s})^{-1} = \sum_{k \ge 0} p^{-\ell ks}$ & $1$ & $1$ \\
            $(\ell^3, 1, 1)$ & $(1 - p^{-s})^{-1} = \sum_{k \ge 0} p^{-ks}$ & $1$ & $1$ \\
            \bottomrule
        \end{tabular}
        \caption{Decomposition types, Euler factors, and $p^{\ell}$ and $p^{\ell^2}$-th zeta coefficients for degree $\ell^3$ extensions.}
        \label{tab:lcbeuler}
    \end{center}
\end{table}

\begin{theorem}
    \label{thm:Clcube_zeta_nonzero}
    Let \( \ell \) be a prime, and let \( K/\bQ \) be a Galois extension of degree \( \ell^3 \) with \( G = \Gal(K/\bQ) \simeq C_\ell^3 \). Then for any prime \( p \) and any integer \( b \ge 1 \), we have \( a_{p^{b\ell}}(K) \neq 0 \); that is, the \( p^{b\ell}\)-th zeta coefficients are always nonzero.
\end{theorem}
\begin{proof}
    Since every element of \( C_\ell^3 \) has order at most \( \ell \), and \( D_p / I_p \subset G / I_p \) is a cyclic group of order \( f \) (when \( p \) has decomposition type \( (e, f, g) \) in \( K \)), we must have \( f = 1 \) or \( \ell \). Therefore, only 7 cases from Table~\ref{tab:lcbeuler} are possible, and the corresponding Euler factors are all of the form \( (1 - p^{-s})^{-g} \) or \( (1 - p^{-\ell s})^{-g} \). The nonvanishing of the \( p^{b\ell} \)-th coefficient then follows immediately.
\end{proof}

\begin{theorem}
    \label{thm:C4C2}

    Let \( K / \bQ \) be a degree \( \ell^3 \) abelian extension with \( G = \Gal(K / \bQ) \simeq C_{\ell^2} \times C_\ell \). Then the set of rational primes \( p \) for which \( a_{p^\ell}(K) = 0 \) has positive (natural) density. Moreover, \( a_{p^{\ell^2}}(K) \ne 0 \) for all primes \( p \); that is, the \(p^{\ell^2}\)-th zeta coefficients are always nonzero.

\end{theorem}
\begin{proof}
    From Table~\ref{tab:lcbeuler}, any unramified prime \( p \) with decomposition type \( (1, \ell^2, \ell) \) or \( (1, \ell^3, 1) \) (i.e., inert in \( K \)) satisfies \( a_{p^\ell}(K) = 0 \). The set of such primes has positive density by the Chebotarev density theorem. For the second part, from Table~\ref{tab:lcbeuler} again, it suffices to show that the decomposition type \( (1, \ell^3, 1) \) cannot occur in \( K \). This follows from the fact that \( C_{\ell^3} \) is not a subquotient of \( C_{\ell^2} \times C_\ell \). 
\end{proof}

\begin{corollary}
    \label{cor:octic_zeta_ab}
    Let $K / \bQ$ be an abelian Galois extension of degree $8$.
    Then the following holds:
    \begin{enumerate}
        \item We have $a_{p^2}(K) \ne 0$ for all primes $p$ if and only if $\Gal(K / \bQ) \simeq C_2^3$.
        \item We have $a_{p^4}(K) = 0$ for some prime $p$ if and only if $\Gal(K / \bQ) \simeq C_8$.
    \end{enumerate}
\end{corollary}

\begin{proof}
        1. ($\Rightarrow$) If \( \Gal(K / \mathbb{Q}) \simeq C_8 \), then Theorem~\ref{thm:cyclic_degree_lpower} with \( \ell = 2 \), \( m = 3 \), and \( b = 2 \) shows that there are infinitely many primes \( p \) such that \( a_{p^2}(K) = 0 \).  
Theorem~\ref{thm:C4C2} implies the same conclusion when \( \Gal(K / \mathbb{Q}) \simeq C_4 \times C_2 \).
($\Leftarrow$) This is precisely Theorem~\ref{thm:Clcube_zeta_nonzero} with \( \ell = 2 \) and \( b = 1 \).

       2. ($\Rightarrow$) If \( \Gal(K / \mathbb{Q}) \simeq C_2^3 \), then Theorem~\ref{thm:Clcube_zeta_nonzero} shows that \( a_{p^4}(K) \) is always nonzero, by taking \( \ell = b = 2 \).  
The case \( \Gal(K / \mathbb{Q}) \simeq C_4 \times C_2 \) follows from Theorem~\ref{thm:C4C2}.
($\Leftarrow$) This follows from Theorem~\ref{thm:cyclic_degree_lpower} with \( \ell = 2 \), \( m = 3 \), and \( b = 4 \). In particular, there are infinitely many such primes \( p \).
\end{proof}

\begin{example}[$C_8$,  sufficient condition]
   Consider the bold branch in Figure~\ref{fig:octic_ab_zeta_dt}.  
If an abelian octic field satisfies the inequalities along this branch, then either \( a_{5^4} = 0 \) or \( a_{3^4} = 0 \).  
By Corollary~\ref{cor:octic_zeta_ab}, this implies that the field must be a \( C_8 \)-extension.  
Therefore, the branch always yields a correct prediction.
\end{example}

As before, we can provide a list of the decomposition types that do not occur for each Galois group and prime~$p$. (Propositions \ref{prop:C8ram}--\ref{prop:C2C2C2ram})

\begin{proposition}
    \label{prop:C8ram}
    Let $K / \bQ$ be a $C_8$-extension and $p$ be a rational prime with decomposition type $(e, f, g)$ in $K$. Then: 
    \begin{enumerate}
        \item If $p = 2$, then $(e, f, g) \not \in \{(1, 8, 1), (2, 4, 1), (4, 2, 1)\}$.
        \item If $p \equiv 5 \pmod{8}$, then $(e, f, g) \ne (8, 1, 1)$.
        \item If $p \equiv 3 \pmod{4}$, then $(e, f, g) \not \in \{(4, 1, 2), (4, 2, 1), (8, 1, 1)\}$.
    \end{enumerate}
 For $p \equiv1 \pmod{8}$, each decomposition type $(e, f, g)$ occurs in some example.
\end{proposition}

\begin{proof}

$p = 2$ case is a direct consequence of Proposition~\ref{thm:wangcounterex}; if $g = 1$, then $2$ totally ramifies in $K$ and the only possibility is $(e, f, g) = (8, 1, 1)$.

For odd primes $p$, Proposition~\ref{prop:newton} shows that a totally ramified $C_8$-extension of $\bQ_p$ exists only if $p \equiv 1 \pmod{8}$. Hence, such an extension is not possible when $p \equiv 5 \pmod{8}$ or $p \equiv 3 \pmod{4}$. 
Similarly, if $p \equiv 3 \pmod{4}$, there is no totally and tamely ramified $C_4$-extension of $\bQ_p$, which excludes $(e, f, g) = (4, 1, 2)$ and $(4, 2, 1)$.
\end{proof}

\begin{proposition}
    \label{prop:C4C2ram}
    Let $K / \bQ$ be a $C_4 \times C_2$-extension and $p$ be a rational prime with a decomposition type $(e, f, g)$ in $K$. Then:
    \begin{enumerate}
        \item $(e, f, g) \ne (1, 8, 1)$.
        \item If $p \equiv 1 \pmod{4}$, then $(e, f, g) \ne (8, 1, 1)$.
        \item If $p \equiv 3 \pmod{4}$, then $(e, f, g) \not \in \{(4, 1, 2), (4, 2, 1), (8, 1, 1) \}$.
    \end{enumerate}
\end{proposition}
\begin{proof}

The largest possible cyclic quotient of \( C_4 \times C_2 \) is \( C_4 \). Hence \(f \le 4 \) and this excludes \( (e, f, g) = (1, 8, 1) \).

For odd primes \( p \), the corresponding extension \( K_\frp / \bQ_p \) is tamely ramified, so both the inertia subgroup \( I_\frp \) and the quotient \( D_\frp / I_\frp \) are cyclic, of orders \( e \) and \( f \), respectively. In particular, the decomposition type \( (e, f, g) = (8, 1, 1) \) is impossible, since \( C_8 \) cannot embed into \( C_4 \times C_2 \). If \( (e, f, g) = (4, 1, 2) \) or \( (4, 2, 1) \), then Proposition~\ref{prop:newton} implies \( p \equiv 1 \pmod{4} \), so such types are excluded when \( p \equiv 3 \pmod{4} \).
\end{proof}

\begin{proposition}
    \label{prop:C2C2C2ram}
    Let $K / \bQ$ be a $C_2^3$-extension and $p$ be a rational prime with decomposition type $(e, f, g)$ in $K$. Then:
    \begin{enumerate}
        \item If $p = 2$, then $(e, f, g) \not \in \{(1, 4, 2), (1, 8, 1), (2, 4, 1), (8, 1, 1)\}$.
        \item If $p$ is odd, then $(e, f, g) \not \in \{(1, 4, 2), (1, 8, 1), (2, 4, 1), (4, 1, 2), (4, 2, 1), (8, 1, 1)\}$.
    \end{enumerate}
\end{proposition}
\begin{proof}
 For $p = 2$ and $(e, f) = (1, 4)$, there is no such extension, since decomposition group of such extension should be $C_4$, which cannot be a subgroup of $C_2^3$.
Similarly, if $ef = 8$, then the decomposition group is $C_2^3$.
Since $\Gal(\bF_{2^f} / \bF_{2}) \simeq C_f$ has to be a quotient of the decomposition group, we have $f \le 2$ and this rules out $(e, f) = (1, 8)$ and $(2, 4)$.

In the case of $(8, 1)$, assume that there exists a totally ramified $C_2^3$-extension $K_\frp$ of $\bQ_2$.
There are 7 quadratic extensions of $\bQ_2$, which are $\bQ_2(\sqrt{-1}), \bQ_2(\sqrt{\pm 2}), \bQ_2(\sqrt{\pm 3}), \bQ_2(\sqrt{\pm 6})$, including the unique unramified extension $\bQ_2(\sqrt{-3})$.
Since there are exactly 7 (normal) subgroups of $C_2^3$ of index 2 (which can be counted by the number of nonzero vectors in $\bF_2^3$ corresponding to the normal vectors of the corresponding hyperplanes), all of the quadratic extensions of $\bQ_2$ appear as a subextension of $K_\frp$.
In particular, it contains the uniramified extension $\bQ_2(\sqrt{-3})$ and cannot be totally ramified.


When \( p \) is odd, the corresponding local field extensions \( K_\frp / \bQ_p \) (for \( \frp \mid p \)) are tamely ramified. The assertion then follows from the fact that the only cyclic subgroups or quotients of \( C_2^3 \) are trivial or isomorphic to \( C_2 \), so we must have \( e, f \le 2 \).
\end{proof}

\begin{example}[$C_4 \times C_2$]
    Based on Propositions~\ref{prop:C8ram}--\ref{prop:C2C2C2ram}, one way to distinguish \( C_4 \times C_2 \) from the other two possibilities using zeta coefficients is to show that the decomposition type at \( p = 2 \) is \( (2, 4, 1) \).  
For instance, this is enforced when \( a_{2^2} = a_{2^3} = 0 \) and \( a_{2^4} = 1 \), as one can see from Table~\ref{tab:lcbeuler}.  
However, we were not able to identify a branch in the decision tree (Figure~\ref{fig:octic_ab_zeta_dt}) that uses this condition to predict the Galois group as \( C_4 \times C_2 \).
\end{example}

\begin{example}[ False positive for $C_2^3$]
\label{ex:C2C2C2}
 We also have an example of a ``false positive'' branch in Figure~\ref{fig:octic_ab_zeta_dt}, similar to Example~\ref{ex:C3C3}.

The only branch that predicts the Galois group as \( C_2^3 \) in Figure~\ref{fig:octic_ab_zeta_dt} is the right-most branch (dashed bold branch).   In fact, the model achieves \(100\%\) recall and precision on the test set when the Galois group is \( C_2^3 \), i.e., there are no false positives or negatives for $C_2^3$.  
However, there exist non-\( C_2^3 \) abelian octic extensions that satisfy all the inequalities along this path but are not included in our dataset. Indeed, using a brute-force search, we find a prime \( q = \num{14836487689} \) such that \( q \equiv 1 \pmod{8} \) and \( p^{\frac{q-1}{8}} \equiv 1 \pmod{q} \) for all \( p = 2, 3, 5, 7, 11, 17, 23, 29, 31 \).  
Then all these primes \( p \) completely split in the unique \( C_8 \)-subextension \( K \) of \( \bQ(\zeta_q) \), so \( a_{p^2} = 36 \) from Table \ref{tab:lcbeuler} and the field satisfies all the inequalities on the dashed path, but is not a \( C_2^3 \)-extension.  
Note that this field is not listed in \texttt{LMFDB} \cite{lmfdb}; there is no number field ramified at \( q \)).
As in Example~\ref{ex:C3C3}, there are infinitely many such \( q \) by the Chebotarev density theorem.
 Although these are the only examples we found from the trees, it would be interesting to find more such examples in other branches or tasks.
\end{example}

\paragraph{Logistic regression}

With 1,000 zeta coefficients (resp. with standard normalization) as input, the logistic regression model achieved 93.81\% (resp. 92.88\%) accuracy in classifying abelian octic Galois groups (see Table~\ref{tab:galois_8_ab}). Table~\ref{tab:octic_ab_lr_weights} lists, for each Galois group, the indices corresponding to the 10 largest and 10 smallest model weights, along with the weight values. Since this is a 3-way classification task, the model has a total of $3 \times 1000$ weights and 3 bias terms. In comparison, all the other logistic regression experiments in this study were binary classification tasks, each involving 1,000 weights and one bias.

A notable pattern is that many indices associated with significantly large or small weights are squares. For instance, among the 10 largest weights for the Galois group $C_2^3$, all of them correspond to square indices.
Also,  25 appears as the index of the  seventh and sixth smallest weight for $C_8$ and $C_4 \times C_2$, while it is the index of the second largest weight for $C_2^3$.
This indicates that a large value of $a_{25}$ pushes the model toward predicting $C_2^3$. The average absolute values of all weights are  0.32, 0.20, and 0.27 for the three Galois groups. However, when restricted to square-indexed coefficients, these averages increase significantly to  2.36, 1.27, and 2.50. This indicates that the model identifies square-indexed coefficients as particularly informative, consistent with the interpretations drawn from the decision tree model.

\begin{table}[H]
    \begin{center}
        \resizebox{\columnwidth}{!}{%
        \begin{tabular}{c|c|c|c|c|c|c|c|c|c|c|c}
            \toprule
            \multirow{2}{*}{$C_8$} & top & 125 (6.5) & 343 (5.7) & 31 (5.2) & \textbf{16} (4.4) & \textbf{729} (4.2) & 23 (4.0) & 19 (3.9) & 29 (3.8) & 11 (2.7) & 243 (2.6) \\
            \cline{2-12}
            & bottom & \textbf{81} (-6.2) & \textbf{961} (-5.3) & \textbf{49} (-5.2) & \textbf{529} (-4.3) & \textbf{841} (-4.2) & \textbf{361} (-4.2) & \textbf{25} (-3.8) & \textbf{289} (-2.8) & \textbf{121} (-2.8) & 162 (-2.7) \\
            \midrule
            \multirow{2}{*}{$C_4 \times C_2$} & top & \textbf{81} (9.2) & \textbf{625} (5.4) & 162 (3.8) & \textbf{324} (3.2) & 567 (2.8) & 5 (1.8) & 891 (1.8) & 405 (1.5) & 648 (1.3) & 488 (1.3) \\
            \cline{2-12}
            & bottom & 125 (-5.1) & \textbf{729} (-3.2) & 243 (-2.8) & 189 (-2.7) & 27 (-2.0) & \textbf{25} (-2.0) & 486 (-1.9) & \textbf{16} (-1.8) & 135 (-1.6) & 54 (-1.5) \\
            \midrule
            \multirow{2}{*}{$C_2^3$} & top & \textbf{961} (5.8) & \textbf{25} (5.8) & \textbf{49} (5.7) & \textbf{841} (5.0) & \textbf{529} (4.8) & \textbf{361} (4.6) & \textbf{9} (3.7) & \textbf{289} (3.4) & \textbf{121} (2.9) & \textbf{36} (2.6) \\
            \cline{2-12}
            & bottom & 343 (-6.4) & 31 (-5.9) & 29 (-4.7) & 23 (-4.4) & 19 (-4.4) & \textbf{625} (-4.1) & 17 (-3.3) & 11 (-3.1) & \textbf{81} (-3.0) & \textbf{16} (-2.6) \\
            \bottomrule
        \end{tabular}
        }
        \caption{Indices of the 10 largest and smallest weights (with actual weight values in parentheses) among 1000 logistic regression coefficients used to classify Galois groups of abelian octic fields  with normalized zeta coefficients.}
        \label{tab:octic_ab_lr_weights}
    \end{center}
\end{table}

\paragraph{Polynomial coefficients}

Using polynomial coefficients, both the decision tree and logistic regression models perform poorly, achieving accuracies of 77.21\% and 19.11\%, respectively.
The latter improves to  64.89\% after normalization.
Recall the discussion at the end of Section~\ref{subsubsec:49_poly}.

\subsubsection{Distinguishing between nonabelian extensions}
\label{subsubsec:octic_nab}

We trained a decision tree to classify non-abelian octic Galois extensions, i.e., distinguishing between \( D_4 \)- and \( Q_8 \)-extensions. The model achieved a high accuracy of  98.67\%, and we observed that the top nodes in the tree correspond to square-indexed coefficients.  However, the previous approach does not yield a criterion based on zeta coefficients to distinguish between the two groups, since the possible decomposition types for a given prime are identical for both \( D_4 \)- and \( Q_8 \)-extensions.
 In particular, for primes $p \le 17$, we checked that all the decomposition types that are \emph{not} ruled out by Proposition~\ref{prop:D4Q8ram} can occur for some number field in \texttt{LMFDB}.\footnote{See \texttt{src/verify\_galois.sage} of the GitHub repository.}

Proposition~\ref{prop:D4Q8ram} below describes the decomposition types that cannot occur in either \( D_4 \)- or \( Q_8 \)-extensions. We first present two lemmas needed for the proof.

\begin{lemma}
    \label{lem:octic_p14_ab}
    Let \( p \equiv 1 \pmod{4} \) be a prime. Then any degree 8 Galois extension of \( \bQ_p \) has Galois group either \( C_8 \) or \( C_4 \times C_2 \).
\end{lemma}
\begin{proof}
    The proof is similar to that of Lemma~\ref{lem:p16_cyclic}.  
Let \( K / \bQ_p \) be an octic Galois extension with Galois group \( G \).  
Since \( (p, 2) = 1 \), the extension is tamely ramified, and the argument in Lemma~\ref{lem:p16_cyclic} shows that \( G \simeq C_e \times C_f \).  
Thus, \( G \) must be abelian. Moreover, \( G \) cannot be isomorphic to \( C_2^3 \), since at least one of \( e \) or \( f \) must be 4, and hence the exponent of \( G \) must be at least 4.
\end{proof}

By a similar argument, we also obtain the following:
\begin{lemma}
    \label{lem:octic_p12_ab}
    Let \( p \) be an odd prime. Then any degree \( 8 \) Galois extension of \( \mathbb{Q}_p \) with decomposition type \( (e, f) = (2, 4) \) is abelian.
\end{lemma}
\begin{proof}
    If \( p \equiv 1 \pmod{4} \), then Lemma~\ref{lem:octic_p14_ab} shows the claim.  
If \( p \equiv 3 \pmod{4} \), the extension is still tamely ramified, and the same argument as in Lemma~\ref{lem:p16_cyclic} shows that the short exact sequence must split, hence \( G \) must be abelian.
\end{proof}

\begin{proposition}
    \label{prop:D4Q8ram}
    Let $K / \bQ$ be an octic Galois extension with Galois group $D_4$ or $Q_8$. Let $p$ be a rational prime with decomposition type $(e, f, g)$ in $K$. Then:
    \begin{enumerate}
        \item  For all prime $p$, $(e, f, g) \not \in \{(1, 8, 1), (2, 4, 1)\}$.
        \item If $p \equiv 1 \pmod{4}$, then $(e, f, g) \not \in \{(4, 2, 1), (8, 1, 1)\}$.
        \item If $p \equiv 3 \pmod{4}$, then $(e, f, g) \not \in \{(4, 1, 2), (8, 1, 1)\}$.
    \end{enumerate}
\end{proposition}
\begin{proof}
    First of all, we have \( (e, f, g) \ne (1, 8, 1) \) for any prime \(p\), since \(C_8\) cannot arise as a quotient of either \(D_4\) or \(Q_8\).
     Assume that a prime $p$ has decomposition type $(e, f, g) = (2, 4, 1)$. Then the decomposition group $D_p$ has a normal subgroup $I_p$ of order $e = 2$, whose quotient has to be cyclic (isomorphic to $\Gal(\bF_{p^4} / \bF_p) \simeq C_4$). However, the only normal subgroup of order 2 in $D_4$ (resp. $Q_8$) is the center, i.e. $\langle r^2 \rangle$ where $r$ is the rotation by 90 degree (resp. $\langle -1 \rangle$), and the quotients by the subgroup is isomorphic to $C_2^2$, since the square of any element belongs to the center.

    This proves $(e, f, g) \not \in \{(1, 8, 1), (2, 4, 1)\}$ for  any $p$.

    For odd $p$, a corresponding local extension $K_\frp / \bQ_p$ has to be tame and so both $I_\frp$ and $D_\frp / I_\frp$ are cyclic of order $e$ and $f$ respectively, which excludes the possibilities $(1, 8, 1)$ and $(8, 1, 1)$.
    For $p \equiv 1 \pmod{4}$, $(e, f, g) \ne (4, 2, 1)$ follows from Lemma \ref{lem:octic_p14_ab}.
    For $p \equiv 3 \pmod{4}$, Lemma \ref{lem:octic_p12_ab} shows $(e, f, g) \ne (2, 4, 1)$.
    If $p \equiv 3 \pmod{4}$ has a decomposition type of $(e, f, g) = (4, 1, 2)$, then a corresponding local extension is totally and tamely ramified quartic extension, so Proposition \ref{prop:newton} shows $p \equiv 1 \pmod{4}$ and gives a contradiction.
\end{proof}

Since the proposition above applies to both nonabelian octic extensions, it does not help distinguish between the two cases. Nonetheless, we observe that the distributions of square-indexed coefficients differ significantly between \( D_4 \)- and \( Q_8 \)-extensions (Table~\ref{tab:octic_nab_psq_dist}). This discrepancy partially explains why the decision tree model achieved high accuracy using square-indexed zeta coefficients. It would be interesting to derive an explicit formula for the probability \( \mathbb{P}[a_{p^2}(K) = a \mid \Gal(K / \mathbb{Q}) \simeq G] \) and \( \mathbb{P}[\Gal(K / \mathbb{Q}) \simeq G \mid a_{p^2}(K) = a]\)  over all Galois extensions $K$ with Galois group $G$ (ordered by their discriminants or heights) in terms of $a$ and $p$.

\paragraph{Polynomial coefficients}

With polynomial coefficients, the decision tree model yields a high accuracy of  98.76\%. By examining the first few nodes, we reach a conclusion similar to the sextic case (Section~\ref{sec:sextic_poly}): the model tends to classify a given octic field as a \( D_4 \)-extension when the coefficients are ``small.''
For example, among the \num{62442} fields in the training set, \num{22307} (resp. \num{40135}) have Galois group \( D_4 \) (resp. \( Q_8 \)). The first node of the decision tree checks whether \( c_4 \le \num{167573.5} \); \num{20953} \( D_4 \)-extensions (93.93\% of \num{22307}) satisfy this condition, while \num{39594} \( Q_8 \)-extensions (98.65\% of \num{40135}) satisfy the opposite inequality \( c_4 > \num{167573.5} \).
We also computed the mean and median of the polynomial coefficients \( c_0, c_1, \dots, c_7 \) for each Galois group, and found that the magnitudes \( |c_0|, |c_1|, \dots, |c_6| \) are significantly larger for \( Q_8 \)-extensions (Table~\ref{tab:octic_nonab_poly}).

In contrast, the logistic regression model achieves only  68.21\% (resp. 84.33\%) accuracy with polynomial coefficients (resp. normalized polynomial coefficients); see Table~\ref{tab:galois_8_nab}.

\begin{table}[H]
    \begin{center}
        \resizebox{\columnwidth}{!}{%
        \begin{tabular}{c|c|c|c|c|c|c|c|c|c|c|c|c|c|c|c|c}
            \toprule
             & \multicolumn{2}{c|}{$c_0$} & \multicolumn{2}{c|}{$c_1$} & \multicolumn{2}{c|}{$c_2$} & \multicolumn{2}{c|}{$c_3$} & \multicolumn{2}{c|}{$c_4$} & \multicolumn{2}{c|}{$c_5$} & \multicolumn{2}{c|}{$c_6$} & \multicolumn{2}{c}{$c_7$} \\
            \midrule
            $D_4$ & $1.3 \cdot 10^{11}$ & $1.2 \cdot 10^5$ & $5.5 \cdot 10^8$ & $0$ & $-1.7 \cdot 10^8$ & $1.8 \cdot 10^3$ & $2.0 \cdot 10^6$ & $0$ & $3.3 \cdot 10^5$ & $1.7 \cdot 10^3$ & $78.4$ & $0$ & $-104.6$ & $4.0$ & $-1.2$ & $0$ \\
            \midrule
            $Q_8$ & $2.3 \cdot 10^{18}$ & $1.5 \cdot 10^{14}$ & $-2.3 \cdot 10^{11}$ & $0$ & $-1.8 \cdot 10^{11}$ & $-1.0 \cdot 10^9$& $1.1 \cdot 10^8$ & $0$ & $7.9 \cdot 10^9$ & $7.7 \cdot 10^8$ & $2.4 \cdot 10^3$ & $0$ & $-8.8 \cdot 10^3$ & $-1.3 \cdot 10^4$ & $-0.6$ & $0$ \\
            \bottomrule
        \end{tabular}
        }
        \caption{Mean (left) and median (right) of polynomial coefficients of nonabelian octic extensions.}
        \label{tab:octic_nonab_poly}
    \end{center}
\end{table}

\subsubsection{Distinguishing between abelian and nonabelian extensions}
\label{subsubsec:octic_ab_nab}

\begin{figure}[t]
    \centering
    \footnotesize
    \begin{forest}
      label L/.style={edge label={node[midway,left,font=\scriptsize]{#1}}},
      label R/.style={edge label={node[midway,right,font=\scriptsize]{#1}}},
      for tree={forked edge, child anchor=north, for descendants={edge=->}}
      [{$a_{5^4} \le 1$}, draw
        [{$\mathrm{ab}$}, rectangle, thick, draw, label L=Y, tier=bottom, line width=1.5pt, edge=very thick, fill=gray!40]
        [{$a_{23^2} \le 3$}, draw, label R=N, edge=very thick
          [{$a_{3^4} =0$}, draw, label L=Y, edge=very thick
            [{$\mathrm{ab}$}, rectangle, thick, draw, label L=Y, tier=bottom, line width=1.5pt]
            [{$a_{11^2} \le 3$}, draw, label R=N, edge=very thick
              [{$a_{3^4} \le 1$}, draw, label L=Y, edge=very thick
                [{$a_{3^2} =0$}, draw, label L=Y, edge=very thick
                  [{$\mathrm{ab}$}, rectangle, thick, draw, label L=Y, tier=bottom, line width=1.5pt]
                  [{$\mathrm{nab}$}, rectangle, thick, draw, label R=N, tier=bottom, line width=1.5pt, edge=very thick, fill=gray!40]
                ]
                [{${\cdots}$}, rectangle, draw, label R=N, tier=bottom]
              ]
              [{${\cdots}$}, rectangle, draw, label R=N, tier=bottom]
            ]
          ]
          [{$a_{17^2} \le 3$}, draw, label R=N, edge=very thick
            [{$a_{3^4} =0$}, draw, label L=Y
              [{$\mathrm{ab}$}, rectangle, thick, draw, label L=Y, tier=bottom, line width=1.5pt]
              [{${\cdots}$}, rectangle, draw, label R=N, tier=bottom]
            ]
            [{$a_{2^2 \cdot 3^2 \cdot 5^2} \le 7$}, draw, label R=N, edge=very thick
              [{${\cdots}$}, rectangle, draw, label L=Y, tier=bottom]
              [{$a_{13^2} =0$}, draw, label R=N, edge=very thick
                [{${\cdots}$}, rectangle, draw, label L=Y, tier=bottom]
                [{$a_{11^2} \le 1$}, draw, label R=N, edge=very thick
                  [{$a_{13^2} \le 1$}, draw, label L=Y, edge=very thick
                    [{$\mathrm{ab}$}, rectangle, thick, draw, label L=Y, tier=bottom, line width=1.5pt, edge=very thick, fill=gray!40]
                    [{${\cdots}$}, rectangle, draw, label R=N, tier=bottom]
                  ]
                  [{${\cdots}$}, rectangle, draw, label R=N, tier=bottom]
                ]
              ]
            ]
          ]
        ]
      ]
    \end{forest}
    \caption{A decision tree predicts commutativity of Galois groups of octic fields from zeta coefficients up to $a_{1000}(K)$.}
    \label{fig:octic_ab_nab_zeta_dt}
\end{figure}

We also conducted experiments on distinguishing between abelian (\( C_8 \), \( C_4 \times C_2 \), \( C_2^3 \)) and nonabelian (\( D_4 \), \( Q_8 \)) extensions as a binary classification task. Using 1000 zeta coefficients, the decision tree model achieved an accuracy of  96.58\%. We observed that the model primarily focuses on square-indexed coefficients (Figure~\ref{fig:octic_ab_nab_zeta_dt}). By combining Propositions~\ref{prop:C8ram}--\ref{prop:C2C2C2ram} and~\ref{prop:D4Q8ram}, one can derive criterias that distinguish abelian from nonabelian extensions.

\begin{corollary}
    \label{cor:octic_ab_nab}
    Let \( K / \bQ \) be an octic Galois extension, and let \( p \) be a rational prime.
    \begin{enumerate}
        \item If \( a_{p^4}(K) = 0 \), then \( \Gal(K / \bQ) \) is a \(C_8\)-extension (hence abelian).
        \item For \( p \equiv 1 \pmod{4} \), if  \( a_{p^4}(K) = 1 \) or \( a_{p^2}(K)=1  \), then \( \Gal(K / \bQ) \) is abelian.
        \item If \( p \equiv 3 \pmod{4} \), \( a_{p^4}(K) = 1 \), and \( a_{p^2}(K) > 0 \), then \( \Gal(K / \bQ) \) is nonabelian.
    \end{enumerate}
\end{corollary}
\begin{proof}
    Let \( (e, f, g) \) be the decomposition type of \( p \) in \( K \).

    \begin{enumerate}
        \item  If \(a_{p^4}(K) = 0\), Table~\ref{tab:lcbeuler} shows that we must have \((e, f, g) = (1, 8, 1)\), which implies  that $\Gal(K / \bQ)$ is isomorphic to the decomposition group ($g = 1$), and the decomposition group is isomorphic to $\Gal(\bF_{p^8} / \bF_p) \simeq C_8$ ($e = 1$).
        This proves the claim.
         (Note that there is no $C_8$-extension of $\bQ$ with $p = 2$ inerts or unramifies by Proposition \ref{thm:wangcounterex}.)
        \item  From Table~\ref{tab:lcbeuler}, the condition \(a_{p^2}(K) = 1\) or \(a_{p^4}(K) = 1\) implies that the possible decomposition types of \(p\) are \((1, 8, 1)\), \((2, 4, 1)\), \((4, 2, 1)\), or \((8, 1, 1)\). By Proposition~\ref{prop:D4Q8ram}, all of these possibilities are excluded for nonabelian extensions, and thus the extension must be abelian.

        \item The condition \( a_{p^4}(K) = 1 \) implies that the possible decomposition types of \( p \) are \( (2, 4, 1) \), \( (4, 2, 1) \), or \( (8, 1, 1) \).  
        Also, \( a_{p^2}(K) > 0 \) implies that the decomposition type cannot be \( (2, 4, 1) \), hence we must have \( (e, f, g) \in \{(4, 2, 1), (8, 1, 1)\} \).  
        However, by Propositions~\ref{prop:C8ram}--\ref{prop:C2C2C2ram}, such decomposition types cannot occur for abelian octic extensions when \( p \equiv 3 \pmod{4} \), so \( \Gal(K / \mathbb{Q}) \) must be nonabelian. \qedhere
    \end{enumerate}
\end{proof}

\begin{example}[abelian, sufficient condition]
    Assume that an octic Galois extension \( K / \mathbb{Q} \) follows the leftmost bold branch of Figure~\ref{fig:octic_ab_nab_zeta_dt}, so that \( a_{5^4}(K) \le 1.5 \).   
    Then $a_{5^4}(K) = 0$ or $1$ and Corollary~\ref{cor:octic_ab_nab} shows that the extension must be abelian.

    Similarly, consider the rightmost bold branch of Figure~\ref{fig:octic_ab_nab_zeta_dt}.  
    If an octic Galois extension satisfies all the inequalities along this path, then \( a_{13^2}(K) = 1 \), which again forces \( K \) to be an abelian extension by Corollary~\ref{cor:octic_ab_nab}.
\end{example}

\begin{example}[nonabelian, sufficient condition]
    Consider the middle bold branch of Figure~\ref{fig:octic_ab_nab_zeta_dt}.  
If an octic Galois extension \( K / \mathbb{Q} \) satisfies all the corresponding inequalities, then we have \( a_{3^4}(K) = 1 \) and \( a_{3^2}(K) > 0.5 \).  
Then Corollary~\ref{cor:octic_ab_nab} implies that \( K \) must be a nonabelian extension, i.e., the branch always yields a correct prediction.
\end{example}

In contrast, we found that logistic regression models perform poorly in distinguishing between abelian and nonabelian extensions, whether using zeta coefficients or polynomial coefficients.  
The model achieved only 76.27\% (resp. 75.83\%) accuracy with 1000 zeta coefficients (resp. with standard normalization).

Table~\ref{tab:galois_8_ab}, \ref{tab:galois_8_nab}, and \ref{tab:galois_8_ab_nab} in the Appendix summarize the results for octic extensions.


\section{Conclusion}
\label{sec:conclusion}

In this paper, we apply interpretable machine learning (ML) methods to the problem of classifying Galois groups and prove several new theorems based on the resulting interpretations.
These results offer different perspectives that emerged only through the analysis of ML outputs.  Our study highlights the potential of ML not merely as a predictive tool, but as a means to generate mathematical conjectures and contribute to the discovery of rigorous mathematical theorems.
We hope that this emerging methodology will continue to gain momentum and contribute to meaningful advances in mathematics.

We conclude this paper by proposing possible directions for future research on Galois groups.
First, it is natural to ask if there are similar criteria in other degrees for determining Galois groups from Dedekind zeta coefficients.
For example, Yang-Hui He asked whether the following statement is true:
\emph{For a given degree \(d\), there exist (possibly infinitely many) pairs \((n, a)\) of natural numbers such that, for any Galois extension \(K / \bQ\) of degree \(d\), the condition \(a_n(K) = a\) implies that \(K\) is a cyclic extension}.

Next, one can consider more general extensions---including non-Galois extensions---and attempt to predict the Galois groups of their splitting fields. 
For example, up to isomorphism there are five possible Galois groups of quartic fields (\(C_4\), \(C_2\times C_2\), \(D_4\), \(A_4\), and \(S_4\)), and it would be interesting to see whether simple ML algorithms, such as decision trees, can still classify these groups effectively.
One can also study extensions of number fields whose base field is not necessarily \(\mathbb{Q}\), using the zeta-function coefficients of both the base and extension fields to predict the corresponding Galois groups.  Ultimately, it would be very interesting if these ML experiments led to new efficient algorithms for computing Galois groups of number fields.

\paragraph{Data availability}

All the data used in the experiments come from \texttt{LMFDB} \cite{lmfdb} with further computations using SageMath \cite{sagemath}. Anyone can get the same data using \texttt{nf\_query.sage} file in the GitHub repository.

\paragraph{Conflict of Interest Statement}

The authors declare no conflicts of interest regarding this manuscript.


\appendix

\section{Tables and Confusion Matrices}

\subsection{Quartic and nonic fields}

\begin{table}[H]
    \begin{center}
        \begin{tabular}{c|c|c|c|c|c|c|c}
            \toprule
            Model & \multicolumn{2}{c|}{Feature} & ACC & BACC & Precision & Recall & F1 \\
            \midrule
            \multirow{10}{*}{DT} & \multirow{8}{*}{ZC} & \multirow{2}{*}{all (1000)} & \multirow{2}{*}{\meanstd{100.00}{0.00}} & \multirow{2}{*}{\meanstd{99.99}{0.02}} & \meanstd{100.00}{0.00} & \meanstd{99.98}{0.04} & \meanstd{99.99}{0.02} \\
             \cline{6-8}
             & & & & & \meanstd{100.00}{0.00} & \meanstd{100.00}{0.00} & \meanstd{100.00}{0.00} \\
             \cline{3-8}
             & & \multirow{2}{*}{squares (31)} & \multirow{2}{*}{\meanstd{99.99}{0.01}} & \multirow{2}{*}{\meanstd{99.97}{0.02}} & \meanstd{99.97}{0.03} & \meanstd{99.95}{0.04} & \meanstd{99.96}{0.03} \\
             \cline{6-8}
             & & & & & \meanstd{99.99}{0.01} & \meanstd{100.00}{0.00} & \meanstd{99.99}{0.00} \\
             \cline{3-8}
             & & \multirow{2}{*}{primes (168)} & \multirow{2}{*}{\meanstd{91.11}{0.09}} & \multirow{2}{*}{\meanstd{79.27}{0.30}} & \meanstd{62.64}{0.94} & \meanstd{63.69}{0.54} & \meanstd{63.16}{0.69} \\
             \cline{6-8}
             & & & & & \meanstd{95.06}{0.06} & \meanstd{94.84}{0.11} & \meanstd{94.95}{0.05} \\
             \cline{3-8}
             & & \multirow{2}{*}{prime powers (193)} & \multirow{2}{*}{\meanstd{100.00}{0.00}} & \multirow{2}{*}{\meanstd{99.99}{0.02}} & \meanstd{100.00}{0.00} & \meanstd{99.98}{0.04} & \meanstd{99.99}{0.02} \\
             \cline{6-8}
             & & & & & \meanstd{100.00}{0.00} & \meanstd{100.00}{0.00} & \meanstd{100.00}{0.00} \\
             \cline{2-8}
             & \multicolumn{2}{c|}{\multirow{2}{*}{PC (4)}} & \multirow{2}{*}{\meanstd{95.16}{0.18}} & \multirow{2}{*}{\meanstd{87.82}{0.32}} & \meanstd{80.73}{0.93} & \meanstd{78.18}{0.61} & \meanstd{79.43}{0.65} \\
             \cline{6-8}
             & \multicolumn{1}{l}{} & & & & \meanstd{97.05}{0.06} & \meanstd{97.46}{0.16} & \meanstd{97.25}{0.11} \\
            \midrule
            \multirow{14}{*}{LR} & \multirow{10}{*}{ZC} & \multirow{2}{*}{all (1000)} & \multirow{2}{*}{\meanstd{99.60}{0.03}} & \multirow{2}{*}{\meanstd{98.88}{0.09}} & \meanstd{98.76}{0.19} & \meanstd{97.93}{0.19} & \meanstd{98.34}{0.09} \\
            \cline{6-8}
            & & & & & \meanstd{99.72}{0.03} & \meanstd{99.83}{0.03} & \meanstd{99.78}{0.01} \\
            \cline{3-8}
            & & \multirow{2}{*}{all (normalized, 1000)} & \multirow{2}{*}{\meanstd{99.41}{0.03}} & \multirow{2}{*}{\meanstd{98.27}{0.09}} & \meanstd{98.26}{0.02} & \meanstd{96.77}{0.16} & \meanstd{97.51}{0.14} \\
            \cline{6-8}
            & & & & & \meanstd{99.56}{0.02} & \meanstd{99.77}{0.03} & \meanstd{99.66}{0.02} \\
            \cline{3-8}
            & & \multirow{2}{*}{squares (31)} & \multirow{2}{*}{\meanstd{98.05}{0.06}} & \multirow{2}{*}{\meanstd{92.29}{0.25}} & \meanstd{98.85}{0.08} & \meanstd{84.72}{0.05} & \meanstd{91.24}{0.27} \\
            \cline{6-8}
            & & & & & \meanstd{97.96}{0.08} & \meanstd{99.87}{0.01} & \meanstd{98.91}{0.04} \\
            \cline{3-8}
            & & \multirow{2}{*}{primes (168)} & \multirow{2}{*}{\meanstd{88.04}{0.17}} & \multirow{2}{*}{\meanstd{50.00}{0.00}} & \meanstd{0.00}{0.00} & \meanstd{0.00}{0.00} & \meanstd{0.00}{0.00} \\
            \cline{6-8}
            & & & & & \meanstd{88.04}{0.00} & \meanstd{100.00}{0.00} & \meanstd{0.00}{0.00} \\
            \cline{3-8}
            & & \multirow{2}{*}{prime powers (193)} & \multirow{2}{*}{\meanstd{99.65}{0.02}} & \multirow{2}{*}{\meanstd{98.92}{0.01}} & \meanstd{99.16}{0.22} & \meanstd{97.94}{0.06} & \meanstd{98.55}{0.08} \\
            \cline{6-8}
            & & & & & \meanstd{99.72}{0.01} & \meanstd{99.89}{0.03} & \meanstd{99.80}{0.01} \\
            \cline{2-8}
            & \multicolumn{2}{c|}{\multirow{2}{*}{PC (4)}} & \multirow{2}{*}{\meanstd{22.45}{0.57}} & \multirow{2}{*}{\meanstd{44.34}{0.43}} & \meanstd{10.53}{0.17} & \meanstd{73.12}{1.41} & \meanstd{18.40}{0.30} \\
            \cline{6-8}
            & \multicolumn{1}{l}{} & & & & \meanstd{80.98}{0.66} & \meanstd{15.56}{0.86} & \meanstd{26.10}{1.23} \\
            \cline{2-8}
            & \multicolumn{2}{c|}{\multirow{2}{*}{PC (normalized, 4)}} & \multirow{2}{*}{\meanstd{88.06}{0.18}} & \multirow{2}{*}{\meanstd{50.08}{0.05}} & \meanstd{100.00}{0.00} & \meanstd{0.16}{0.09} & \meanstd{0.32}{0.18} \\
            \cline{6-8}
            & \multicolumn{1}{l}{} & & & & \meanstd{88.06}{0.18} & \meanstd{100.00}{0.00} & \meanstd{93.65}{0.10} \\
            \bottomrule
        \end{tabular}
         \caption{ Performance of decision tree (DT) and logistic regression (LR) models on distinguishing Galois groups of quartic extensions, using zeta coefficients (ZC) or polynomial coefficients (PC) as features, averaged over 5 different random splits. We report accuracy (ACC), balanced accuracy (BACC), per-class precision, recall, and F1. For the latter three metrics, top (resp. bottom) row corresponds to the class of $C_4$ (resp. $C_2^2$). The majority baseline is 88.04\% ($C_2^2$).}
        \label{tab:galois_4}
    \end{center}
\end{table}

\begin{table}[H]
    \begin{center}
        \begin{tabular}{c|c|c|c|c|c|c|c}
            \toprule
            Model & \multicolumn{2}{c|}{Feature} & ACC & BACC & Precision & Recall & F1 \\
            \midrule
            \multirow{10}{*}{DT} & \multirow{8}{*}{ZC} & \multirow{2}{*}{all (1000)} & \multirow{2}{*}{\meanstd{100.00}{0.00}} & \multirow{2}{*}{\meanstd{100.00}{0.00}} & \meanstd{100.00}{0.00} & \meanstd{100.00}{0.00} & \meanstd{100.00}{0.00} \\
             \cline{6-8}
             & & & & & \meanstd{100.00}{0.00} & \meanstd{100.00}{0.00} & \meanstd{100.00}{0.00} \\
             \cline{3-8}
             & & \multirow{2}{*}{cubes (10)} & \multirow{2}{*}{\meanstd{100.00}{0.00}} & \multirow{2}{*}{\meanstd{100.00}{0.00}} & \meanstd{100.00}{0.00} & \meanstd{100.00}{0.00} & \meanstd{100.00}{0.00} \\
             \cline{6-8}
             & & & & & \meanstd{100.00}{0.00} & \meanstd{100.00}{0.00} & \meanstd{100.00}{0.00} \\
             \cline{3-8}
             & & \multirow{2}{*}{primes (168)} & \multirow{2}{*}{\meanstd{92.50}{0.66}} & \multirow{2}{*}{\meanstd{88.60}{1.68}} & \meanstd{84.27}{3.78} & \meanstd{81.57}{4.21} & \meanstd{82.76}{2.16} \\
             \cline{6-8}
             & & & & & \meanstd{94.75}{1.25} & \meanstd{95.62}{1.23} & \meanstd{95.17}{0.51} \\
             \cline{3-8}
             & & \multirow{2}{*}{prime powers (193)} & \multirow{2}{*}{\meanstd{100.00}{0.00}} & \multirow{2}{*}{\meanstd{100.00}{0.00}} & \meanstd{100.00}{0.00} & \meanstd{100.00}{0.00} & \meanstd{100.00}{0.00} \\
             \cline{6-8}
             & & & & & \meanstd{100.00}{0.00} & \meanstd{100.00}{0.00} & \meanstd{100.00}{0.00} \\
             \cline{2-8}
             & \multicolumn{2}{c|}{\multirow{2}{*}{PC (9)}} & \multirow{2}{*}{\meanstd{94.15}{2.58}} & \multirow{2}{*}{\meanstd{92.16}{3.48}} & \meanstd{86.19}{6.74} & \meanstd{88.38}{5.38} & \meanstd{87.18}{5.29} \\
             \cline{6-8}
             & \multicolumn{1}{l}{} & & & & \meanstd{96.48}{1.81} & \meanstd{95.94}{2.12} & \meanstd{96.20}{1.70} \\
            \midrule
            \multirow{14}{*}{LR} & \multirow{10}{*}{ZC} & \multirow{2}{*}{all (1000)} & \multirow{2}{*}{\meanstd{97.16}{0.92}} & \multirow{2}{*}{\meanstd{94.66}{1.67}} & \meanstd{97.24}{3.90} & \meanstd{90.01}{3.70} & \meanstd{93.36}{1.93} \\
            \cline{6-8}
            & & & & & \meanstd{97.10}{1.38} & \meanstd{99.30}{0.89} & \meanstd{98.18}{0.63} \\
            \cline{3-8}
            & & \multirow{2}{*}{all (normalized, 1000)} & \multirow{2}{*}{\meanstd{78.51}{2.84}} & \multirow{2}{*}{\meanstd{67.19}{2.42}} & \meanstd{52.57}{4.16} & \meanstd{46.52}{4.72} & \meanstd{49.13}{2.91} \\
            \cline{6-8}
            & & & & & \meanstd{84.91}{3.15} & \meanstd{87.86}{2.29} & \meanstd{86.32}{2.16} \\
            \cline{3-8}
            & & \multirow{2}{*}{cubes (10)} & \multirow{2}{*}{\meanstd{95.97}{0.52}} & \multirow{2}{*}{\meanstd{91.12}{0.48}} & \meanstd{99.59}{0.82} & \meanstd{82.33}{0.98} & \meanstd{90.14}{0.62} \\
            \cline{6-8}
            & & & & & \meanstd{95.13}{0.77} & \meanstd{99.90}{0.21} & \meanstd{97.45}{0.42} \\
            \cline{3-8}
            & & \multirow{2}{*}{primes (168)} & \multirow{2}{*}{\meanstd{73.70}{1.80}} & \multirow{2}{*}{\meanstd{56.77}{2.15}} & \meanstd{37.88}{9.01} & \meanstd{26.13}{2.87} & \meanstd{30.73}{4.79} \\
            \cline{6-8}
            & & & & & \meanstd{80.41}{2.34} & \meanstd{87.42}{2.23} & \meanstd{83.73}{1.35} \\
            \cline{3-8}
            & & \multirow{2}{*}{prime powers (193)} & \multirow{2}{*}{\meanstd{96.52}{0.77}} & \multirow{2}{*}{\meanstd{93.57}{1.05}} & \meanstd{96.13}{1.93} & \meanstd{88.16}{2.12} & \meanstd{91.94}{1.30} \\
            \cline{6-8}
            & & & & & \meanstd{96.59}{1.00} & \meanstd{98.98}{0.56} & \meanstd{97.77}{0.55} \\
            \cline{2-8}
            & \multicolumn{2}{c|}{\multirow{2}{*}{PC (9)}} & \multirow{2}{*}{\meanstd{84.52}{2.95}} & \multirow{2}{*}{\meanstd{68.65}{2.50}} & \meanstd{83.28}{5.46} & \meanstd{39.65}{5.08} & \meanstd{53.44}{4.55} \\
            \cline{6-8}
            & \multicolumn{1}{l}{} & & & & \meanstd{84.71}{3.31} & \meanstd{97.65}{1.01} & \meanstd{90.69}{1.98} \\
            \cline{2-8}
            & \multicolumn{2}{c|}{\multirow{2}{*}{PC (normalized, 9)}} & \multirow{2}{*}{\meanstd{85.46}{3.37}} & \multirow{2}{*}{\meanstd{68.72}{3.27}} & \meanstd{95.69}{3.81} & \meanstd{37.93}{6.79} & \meanstd{53.86}{6.74} \\
            \cline{6-8}
            & \multicolumn{1}{l}{} & & & & \meanstd{84.54}{3.81} & \meanstd{99.50}{0.45} & \meanstd{91.36}{2.15} \\
            \bottomrule
        \end{tabular}
        \caption{ Performance of decision tree (DT) and logistic regression (LR) models on distinguishing Galois groups of nonic extensions, using zeta coefficients (ZC) or polynomial coefficients (PC) as features, averaged over 5 different random splits. We report accuracy (ACC), balanced accuracy (BACC), per-class precision, recall, and F1. For the latter three metrics, top (resp. bottom) row corresponds to the class of $C_9$ (resp. $C_3^2$). The majority baseline is 77.57\% ($C_3^2$).}
        \label{tab:galois_9}
    \end{center}
\end{table}

\subsection{Sextic and decic fields}

\begin{table}[H]
    \begin{center}
        \begin{tabular}{c|c|c|c|c|c|c|c|c|c|c|c|c}
            \toprule
             & \multicolumn{2}{c|}{$a_{2^2}$} & \multicolumn{2}{c|}{$a_{3^2}$} & \multicolumn{2}{c|}{$a_{5^2}$} & \multicolumn{2}{c|}{$a_{7^2}$} & \multicolumn{2}{c|}{$a_{17^2}$} & \multicolumn{2}{c}{$a_{19^2}$} \\
            \midrule
            $G$ & $C_6$ & $S_3$ & $C_6$ & $S_3$ & $C_6$ & $S_3$ & $C_6$ & $S_3$ & $C_6$ & $S_3$ & $C_6$ & $S_3$ \\
            \midrule
            0 & 0.78 & 0.04 & 0.53 & 0.22 & 0.75 & 0.27 & 0.50 & 0.33 & 0.72 & 0.31 & 0.64 & 0.36 \\
            1 & - & 0.26 & 0.29 & 0.11 & - & 0.06 & 0.23 & - & - & 0.01 & 0.07 & - \\
            3 & 0.08 & 0.16 & 0.10 & 0.31 & 0.11 & 0.42 & 0.17 & 0.42 & 0.15 & 0.50 & 0.17 & 0.45 \\
            6 & 0.09 & 0.50 & 0.05 & 0.30 & 0.05 & 0.17 & 0.02 & 0.13 & 0.02 & 0.06 & 0.01 & 0.05 \\
            21 & 0.05 & 0.03 & 0.03 & 0.06 & 0.10 & 0.08 & 0.07 & 0.12 & 0.11 & 0.12 & 0.12 & 0.15 \\
            \bottomrule
        \end{tabular}
        \caption{Distribution of $a_{p^2}(K)$ with Galois group $G \in \{C_6, S_3\}$ for sextic extensions and $p = 2, 3, 5, 7, 17, 19$.}
        \label{tab:sextic_psq_dist}
    \end{center}
\end{table}

\begin{table}[H]
   \begin{center}
        \begin{tabular}{c|c|c|c|c|c|c|c|c|c|c|c|c}
            \toprule
            & \multicolumn{2}{c|}{$a_{4}$} & \multicolumn{2}{c|}{$a_{9}$} & \multicolumn{2}{c|}{$a_{25}$} & \multicolumn{2}{c|}{$a_{49}$} & \multicolumn{2}{c|}{$a_{289}$} & \multicolumn{2}{c}{$a_{361}$} \\
            \midrule
            $G$ & $C_{6}$ & $S_{3}$ & $C_{6}$ & $S_{3}$ & $C_{6}$ & $S_{3}$ & $C_{6}$ & $S_{3}$ & $C_{6}$ & $S_{3}$ & $C_{6}$ & $S_{3}$ \\
            \midrule
            0 & 0.75 & 0.25 & 0.27 & 0.73 & 0.30 & 0.70 & 0.19 & 0.81 & 0.27 & 0.73 & 0.22 & 0.78 \\
            1 & - & 1.00 & 0.30 & 0.70 & - & 1.00 & 1.00 & - & - & 1.00 & 1.00 & - \\
            3 & 0.07 & 0.93 & 0.05 & 0.95 & 0.04 & 0.96 & 0.06 & 0.94 & 0.05 & 0.95 & 0.05 & 0.95 \\
            6 & 0.03 & 0.97 & 0.02 & 0.98 & 0.04 & 0.96 & 0.03 & 0.97 & 0.04 & 0.96 & 0.04 & 0.96 \\
            21 & 0.18 & 0.82 & 0.07 & 0.93 & 0.15 & 0.85 & 0.09 & 0.91 & 0.12 & 0.88 & 0.11 & 0.89 \\
            \bottomrule
        \end{tabular}
        \caption{Distribution of Galois group $G \in \{C_6, S_3\}$ for sextic extensions and given $a_{p^2}(K)$ with $p = 2, 3, 5, 7, 17, 19$.}
        \label{tab:sextic_psq_dist_G}
   \end{center}
\end{table}

\begin{table}[H]
    \begin{center}
        \begin{tabular}{c|c|c|c|c|c|c|c|c|c|c|c|c}
            \toprule
             & \multicolumn{2}{c|}{$a_{2^3}$} & \multicolumn{2}{c|}{$a_{3^3}$} & \multicolumn{2}{c|}{$a_{5^3}$} & \multicolumn{2}{c|}{$a_{7^3}$} & \multicolumn{2}{c|}{$a_{17^3}$} & \multicolumn{2}{c}{$a_{19^3}$} \\
            \midrule
            $G$ & $C_6$ & $S_3$ & $C_6$ & $S_3$ & $C_6$ & $S_3$ & $C_6$ & $S_3$ & $C_6$ & $S_3$ & $C_6$ & $S_3$ \\
            \midrule
            0 & 0.37 & 0.42 & 0.32 & 0.31 & 0.45 & 0.47 & 0.37 & 0.39 & 0.51 & 0.51 & 0.47 & 0.44 \\
            1 & 0.26 & - & 0.35 & 0.09 & 0.12 & - & 0.18 & - & 0.03 & - & 0.06 & - \\
            2 & 0.22 & 0.04 & 0.19 & 0.23 & 0.28 & 0.27 & 0.24 & 0.33 & 0.33 & 0.31 & 0.29 & 0.36 \\
            4 & - & - & 0.06 & 0.01 & - & - & 0.11 & 0.04 & - & - & 0.04 & 0.01 \\
            10 & 0.09 & 0.50 & 0.05 & 0.30 & 0.05 & 0.17 & 0.02 & 0.13 & 0.02 & 0.06 & 0.01 & 0.05 \\
            56 & 0.05 & 0.04 & 0.03 & 0.06 & 0.10 & 0.08 & 0.07 & 0.12 & 0.11 & 0.12 & 0.12 & 0.15 \\
            \bottomrule
        \end{tabular}
        \caption{Distribution of $a_{p^3}(K)$ with Galois group $G \in \{C_6, S_3\}$ for sextic extensions and $p = 2, 3, 5, 7, 17, 19$.}
        \label{tab:sextic_pcb_dist}
    \end{center}
\end{table}

\begin{table}[H]
   \begin{center}
        \begin{tabular}{c|c|c|c|c|c|c|c|c|c|c|c|c}
            \toprule
            & \multicolumn{2}{c|}{$a_{2^3}$} & \multicolumn{2}{c|}{$a_{3^3}$} & \multicolumn{2}{c|}{$a_{5^3}$} & \multicolumn{2}{c|}{$a_{7^3}$} & \multicolumn{2}{c|}{$a_{17^3}$} & \multicolumn{2}{c}{$a_{19^3}$} \\
            \midrule
            $G$ & $C_{6}$ & $S_{3}$ & $C_{6}$ & $S_{3}$ & $C_{6}$ & $S_{3}$ & $C_{6}$ & $S_{3}$ & $C_{6}$ & $S_{3}$ & $C_{6}$ & $S_{3}$ \\
            \midrule
            0 & 0.12 & 0.88 & 0.14 & 0.86 & 0.13 & 0.87 & 0.13 & 0.87 & 0.13 & 0.87 & 0.14 & 0.86 \\
            1 & 1.00 & - & 0.37 & 0.63 & 1.00 & - & 1.00 & - & 1.00 & - & 1.00 & - \\
            2 & 0.46 & 0.54 & 0.11 & 0.89 & 0.14 & 0.86 & 0.10 & 0.90 & 0.14 & 0.86 & 0.11 & 0.89 \\
            4 & - & - & 0.48 & 0.52 & - & - & 0.32 & 0.68 & - & - & 0.39 & 0.61 \\
            10 & 0.03 & 0.97 & 0.02 & 0.98 & 0.04 & 0.96 & 0.03 & 0.97 & 0.04 & 0.96 & 0.04 & 0.96 \\
            56 & 0.18 & 0.82 & 0.07 & 0.93 & 0.15 & 0.85 & 0.09 & 0.91 & 0.12 & 0.88 & 0.11 & 0.89 \\
            \bottomrule
        \end{tabular}
        \caption{Distribution of Galois group $G \in \{C_6, S_3\}$ for sextic extensions and given $a_{p^3}(K)$ with $p = 2, 3, 5, 7, 17, 19$.}
        \label{tab:sextic_pcb_dist_G}
   \end{center}
\end{table}

\begin{table}[H]
    \begin{center}
        \begin{tabular}{c|c|c|c|c|c|c|c}
            \toprule
            Model & \multicolumn{2}{c|}{Feature} & ACC & BACC & Precision & Recall & F1 \\
            \midrule
            \multirow{10}{*}{DT} & \multirow{8}{*}{ZC} & \multirow{2}{*}{all (1000)} & \multirow{2}{*}{\meanstd{98.96}{0.02}} & \multirow{2}{*}{\meanstd{97.95}{0.07}} & \meanstd{95.75}{0.18} & \meanstd{96.57}{0.14} & \meanstd{96.16}{0.08} \\
             \cline{6-8}
             & & & & & \meanstd{99.47}{0.03} & \meanstd{99.33}{0.02} & \meanstd{99.40}{0.01} \\
             \cline{3-8}
             & & \multirow{2}{*}{squares and cubes (38)} & \multirow{2}{*}{\meanstd{99.31}{0.04}} & \multirow{2}{*}{\meanstd{98.60}{0.13}} & \meanstd{97.24}{0.32} & \meanstd{97.64}{0.28} & \meanstd{97.44}{0.15} \\
             \cline{6-8}
             & & & & & \meanstd{99.63}{0.05} & \meanstd{99.57}{0.05} & \meanstd{99.60}{0.03} \\
             \cline{3-8}
             & & \multirow{2}{*}{primes (168)} & \multirow{2}{*}{\meanstd{89.66}{0.17}} & \multirow{2}{*}{\meanstd{80.38}{0.54}} & \meanstd{60.33}{0.69} & \meanstd{67.69}{1.11} & \meanstd{63.79}{0.63} \\
             \cline{6-8}
             & & & & & \meanstd{94.88}{0.22} & \meanstd{93.08}{0.13} & \meanstd{93.97}{0.10} \\
             \cline{3-8}
             & & \multirow{2}{*}{prime powers (193)} & \multirow{2}{*}{\meanstd{99.04}{0.04}} & \multirow{2}{*}{\meanstd{98.21}{0.09}} & \meanstd{95.81}{0.37} & \meanstd{97.09}{0.22} & \meanstd{96.45}{0.13} \\
             \cline{6-8}
             & & & & & \meanstd{99.55}{0.04} & \meanstd{99.34}{0.06} & \meanstd{99.44}{0.02} \\
             \cline{2-8}
             & \multicolumn{2}{c|}{\multirow{2}{*}{PC (6)}} & \multirow{2}{*}{\meanstd{97.56}{0.08}} & \multirow{2}{*}{\meanstd{94.65}{0.18}} & \meanstd{91.13}{0.35} & \meanstd{90.67}{0.36} & \meanstd{90.90}{0.26} \\
             \cline{6-8}
             & \multicolumn{1}{l}{} & & & & \meanstd{98.55}{0.06} & \meanstd{98.63}{0.06} & \meanstd{98.59}{0.05} \\
            \midrule
            \multirow{14}{*}{LR} & \multirow{10}{*}{ZC} & \multirow{2}{*}{all (1000)} & \multirow{2}{*}{\meanstd{98.40}{0.07}} & \multirow{2}{*}{\meanstd{95.90}{0.19}} & \meanstd{95.47}{0.38} & \meanstd{92.48}{0.40} & \meanstd{93.95}{0.24} \\
            \cline{6-8}
            & & & & & \meanstd{98.84}{0.06} & \meanstd{99.32}{0.07} & \meanstd{99.08}{0.04} \\
            \cline{3-8}
            & & \multirow{2}{*}{all (normalized, 1000)} & \multirow{2}{*}{\meanstd{97.80}{0.10}} & \multirow{2}{*}{\meanstd{94.35}{0.21}} & \meanstd{93.75}{0.45} & \meanstd{89.63}{0.40} & \meanstd{91.65}{0.33} \\
            \cline{6-8}
            & & & & & \meanstd{98.40}{0.07} & \meanstd{99.07}{0.08} & \meanstd{98.73}{0.06} \\
            \cline{3-8}
            & & \multirow{2}{*}{squares and cubes (38)} & \multirow{2}{*}{\meanstd{96.51}{0.13}} & \multirow{2}{*}{\meanstd{89.79}{0.50}} & \meanstd{92.50}{0.18} & \meanstd{80.60}{1.02} & \meanstd{86.14}{0.59} \\
            \cline{6-8}
            & & & & & \meanstd{97.04}{0.15} & \meanstd{98.98}{0.03} & \meanstd{98.00}{0.08} \\
            \cline{3-8}
            & & \multirow{2}{*}{primes (168)} & \multirow{2}{*}{\meanstd{91.34}{0.15}} & \multirow{2}{*}{\meanstd{69.36}{0.29}} & \meanstd{91.56}{0.67} & \meanstd{39.27}{0.58} & \meanstd{54.97}{0.56} \\
            \cline{6-8}
            & & & & & \meanstd{91.33}{0.19} & \meanstd{99.44}{0.04} & \meanstd{95.21}{0.09} \\
            \cline{3-8}
            & & \multirow{2}{*}{prime powers (193)} & \multirow{2}{*}{\meanstd{98.48}{0.03}} & \multirow{2}{*}{\meanstd{95.99}{0.13}} & \meanstd{95.98}{0.29} & \meanstd{92.58}{0.30} & \meanstd{94.25}{0.07} \\
            \cline{6-8}
            & & & & & \meanstd{98.85}{0.05} & \meanstd{99.40}{0.05} & \meanstd{99.12}{0.02} \\
            \cline{2-8}
            & \multicolumn{2}{c|}{\multirow{2}{*}{PC (6)}} & \multirow{2}{*}{\meanstd{26.22}{4.75}} & \multirow{2}{*}{\meanstd{50.69}{2.70}} & \meanstd{13.58}{0.74} & \meanstd{84.15}{12.45} & \meanstd{23.36}{1.54} \\
            \cline{6-8}
            & \multicolumn{1}{l}{} & & & & \meanstd{90.48}{6.47} & \meanstd{17.22}{7.35} & \meanstd{27.97}{10.44} \\
            \cline{2-8}
            & \multicolumn{2}{c|}{\multirow{2}{*}{PC (normalized, 6)}} & \multirow{2}{*}{\meanstd{86.73}{0.50}} & \multirow{2}{*}{\meanstd{50.72}{1.22}} & \meanstd{99.26}{1.48} & \meanstd{1.44}{2.45} & \meanstd{2.73}{4.59} \\
            \cline{6-8}
            & \multicolumn{1}{l}{} & & & & \meanstd{86.71}{0.47} & \meanstd{99.99}{0.01} & \meanstd{92.88}{0.26} \\
            \bottomrule
        \end{tabular}
        \caption{ Performance of decision tree (DT) and logistic regression (LR) models on distinguishing Galois groups of sextic extensions, using zeta coefficients (ZC) or polynomial coefficients (PC) as features, averaged over 5 different random splits. We report accuracy (ACC), balanced accuracy (BACC), per-class precision, recall, and F1. For the latter three metrics, top (resp. bottom) row corresponds to the class of $C_6$ (resp. $S_3$). The majority baseline is 86.55\% ($S_3$).}
        \label{tab:galois_6}
    \end{center}
\end{table}

\begin{table}[H]
    \begin{center}
        \begin{tabular}{c|c|c|c|c|c|c|c|c|c|c|c|c}
            \toprule
             & \multicolumn{2}{c|}{$a_{2^2}$} & \multicolumn{2}{c|}{$a_{3^2}$} & \multicolumn{2}{c|}{$a_{5^2}$} & \multicolumn{2}{c|}{$a_{7^2}$} & \multicolumn{2}{c|}{$a_{11^2}$} & \multicolumn{2}{c}{$a_{19^2}$} \\
            \midrule
            $G$ & $C_{10}$ & $D_5$ & $C_{10}$ & $D_5$ & $C_{10}$ & $D_5$ & $C_{10}$ & $D_5$ & $C_{10}$ & $D_5$ & $C_{10}$ & $D_5$ \\
            \midrule
            0 & 0.92 & 0.30 & 0.91 & 0.33 & 0.66 & 0.38 & 0.86 & 0.41 & 0.39 & 0.40 & 0.92 & 0.41 \\
            1 & - & - & - & - & 0.18 & 0.06 & - & - & 0.39 & - & - & 0.02 \\
            3 & - & - & - & - & 0.06 & 0.01 & - & - & 0.16 & 0.04 & - & - \\
            5 & 0.03 & 0.34 & 0.04 & 0.38 & 0.04 & 0.38 & 0.06 & 0.43 & 0.03 & 0.44 & 0.04 & 0.46 \\
            15 & 0.03 & 0.32 & 0.02 & 0.24 & 0.02 & 0.13 & 0.02 & 0.11 & 0.00 & 0.06 & 0.00 & 0.05 \\
            55 & 0.02 & 0.04 & 0.03 & 0.04 & 0.04 & 0.04 & 0.05 & 0.05 & 0.03 & 0.06 & 0.04 & 0.06 \\
            \bottomrule
        \end{tabular}
        \caption{Distribution of $a_{p^2}(K)$ with Galois group $G \in \{C_{10}, D_5\}$ for decic extensions and $p = 2, 3, 5, 7, 11, 19$.}
        \label{tab:decic_psq_dist}
    \end{center}
\end{table}

\begin{table}[H]
    \begin{center}
        \begin{tabular}{c|c|c|c|c|c|c|c|c|c|c|c|c}
            \toprule
             & \multicolumn{2}{c|}{$a_{2^2}$} & \multicolumn{2}{c|}{$a_{3^2}$} & \multicolumn{2}{c|}{$a_{5^2}$} & \multicolumn{2}{c|}{$a_{7^2}$} & \multicolumn{2}{c|}{$a_{11^2}$} & \multicolumn{2}{c}{$a_{19^2}$} \\
            \midrule
            $G$ & $C_{10}$ & $D_5$ & $C_{10}$ & $D_5$ & $C_{10}$ & $D_5$ & $C_{10}$ & $D_5$ & $C_{10}$ & $D_5$ & $C_{10}$ & $D_5$ \\
            \midrule
            0 & 0.61 & 0.39 & 0.58 & 0.42 & 0.47 & 0.53 & 0.52 & 0.48 & 0.33 & 0.67 & 0.53 & 0.47 \\
            1 & - & - & - & - & 0.61 & 0.39 & - & - & 1.00 & - & - & 1.00 \\
            3 & - & - & - & - & 0.77 & 0.23 & - & - & 0.67 & 0.33 & - & - \\
            5 & 0.04 & 0.96 & 0.05 & 0.95 & 0.05 & 0.95 & 0.07 & 0.93 & 0.04 & 0.96 & 0.05 & 0.95 \\
            15 & 0.04 & 0.96 & 0.05 & 0.95 & 0.08 & 0.92 & 0.07 & 0.93 & 0.03 & 0.97 & 0.01 & 0.99 \\
            55 & 0.25 & 0.75 & 0.29 & 0.71 & 0.34 & 0.66 & 0.37 & 0.63 & 0.18 & 0.82 & 0.25 & 0.75 \\
            \bottomrule
        \end{tabular}
        \caption{Distribution of Galois group $G \in \{C_{10}, D_5\}$ for  decic extensions and given $a_{ p^2}(K)$ with $p = 2, 3, 5, 7, 11, 19$.}
        \label{tab:decic_psq_dist_G}
    \end{center}
\end{table}

\begin{table}[H]
    \begin{center}
        \begin{tabular}{c|c|c|c|c|c|c|c|c}
            \toprule
             & \multicolumn{2}{c|}{$a_{2^5}$} & \multicolumn{2}{c|}{$a_{3^5}$} & \multicolumn{2}{c|}{$a_{5^5}$} & \multicolumn{2}{c}{$a_{7^5}$} \\
            \midrule
            $G$ & $C_{10}$ & $D_5$ & $C_{10}$ & $D_5$ & $C_{10}$ & $D_5$ & $C_{10}$ & $D_5$ \\
            \midrule
            0 & 0.35 & 0.34 & 0.39 & 0.38 & 0.35 & 0.39 & 0.47 & 0.43 \\
            1 & 0.32 & - & 0.22 & - & 0.26 & 0.05 & 0.09 & - \\
            2 & 0.28 & 0.30 & 0.33 & 0.33 & 0.26 & 0.38 & 0.37 & 0.41 \\
            6 & - & - & - & - & 0.06 & 0.01 & - & - \\
            126 & 0.03 & 0.32 & 0.02 & 0.24 & 0.02 & 0.13 & 0.02 & 0.11 \\
            2002 & 0.02 & 0.04 & 0.03 & 0.04 & 0.04 & 0.04 & 0.06 & 0.05 \\
            \bottomrule
        \end{tabular}
        \caption{Distribution of $a_{p^5}(K)$ with Galois group $G \in \{C_{10}, D_5\}$ for decic extensions and $p = 2, 3, 5, 7$.}
        \label{tab:decic_pqu_dist}
    \end{center}
\end{table}

\begin{table}[H]
    \begin{center}
        \begin{tabular}{c|c|c|c|c|c|c|c|c}
            \toprule
             & \multicolumn{2}{c|}{$a_{2^5}$} & \multicolumn{2}{c|}{$a_{3^5}$} & \multicolumn{2}{c|}{$a_{5^5}$} & \multicolumn{2}{c}{$a_{7^5}$} \\
            \midrule
            $G$ & $C_{10}$ & $D_5$ & $C_{10}$ & $D_5$ & $C_{10}$ & $D_5$ & $C_{10}$ & $D_5$ \\
            \midrule
            0 & 0.35 & 0.65 & 0.35 & 0.65 & 0.32 & 0.68 & 0.36 & 0.64 \\
            1 & 1.00 & - & 1.00 & - & 0.72 & 0.28 & 1.00 & - \\
            2 & 0.32 & 0.68 & 0.34 & 0.66 & 0.26 & 0.74 & 0.32 & 0.68 \\
            6 & - & - & - & - & 0.77 & 0.23 & - & - \\
            126 & 0.04 & 0.96 & 0.05 & 0.95 & 0.08 & 0.92 & 0.07 & 0.93 \\
            2002 & 0.25 & 0.75 & 0.29 & 0.71 & 0.34 & 0.66 & 0.37 & 0.63 \\
            \bottomrule
        \end{tabular}
        \caption{Distribution of Galois group $G \in \{C_{10}, D_5\}$ for sextic extensions and given $a_{ p^5}(K)$ with $p = 2, 3, 5, 7$.}
        \label{tab:decic_pqu_dist_G}
    \end{center}
\end{table}

\begin{table}[H]
    \begin{center}
        \begin{tabular}{c|c|c|c|c|c|c|c}
            \toprule
            Model & \multicolumn{2}{c|}{Feature} & ACC & BACC & Precision & Recall & F1 \\
            \midrule
            \multirow{10}{*}{DT} & \multirow{8}{*}{ZC} & \multirow{2}{*}{all ($n \le 1000$)} & \multirow{2}{*}{\meanstd{97.19}{0.47}} & \multirow{2}{*}{\meanstd{96.92}{0.62}} & \meanstd{95.61}{0.66} & \meanstd{96.11}{1.18} & \meanstd{95.85}{0.71} \\
             \cline{6-8}
             & & & & & \meanstd{98.01}{0.57} & \meanstd{97.73}{0.39} & \meanstd{97.87}{0.36} \\
             \cline{3-8}
             & & \multirow{2}{*}{squares and quintics (33)} & \multirow{2}{*}{\meanstd{98.56}{0.24}} & \multirow{2}{*}{\meanstd{98.46}{0.26}} & \meanstd{97.60}{0.52} & \meanstd{98.16}{0.45} & \meanstd{97.88}{0.35} \\
             \cline{6-8}
             & & & & & \meanstd{99.06}{0.21} & \meanstd{98.76}{0.30} & \meanstd{98.91}{0.18} \\
             \cline{3-8}
             & & \multirow{2}{*}{primes (168)} & \multirow{2}{*}{\meanstd{85.50}{0.43}} & \multirow{2}{*}{\meanstd{84.77}{0.46}} & \meanstd{76.50}{0.66} & \meanstd{82.53}{0.67} & \meanstd{79.40}{0.56} \\
             \cline{6-8}
             & & & & & \meanstd{90.68}{0.48} & \meanstd{87.02}{0.40} & \meanstd{88.81}{0.39} \\
             \cline{3-8}
             & & \multirow{2}{*}{prime powers (193)} & \multirow{2}{*}{\meanstd{98.14}{0.43}} & \multirow{2}{*}{\meanstd{97.85}{0.57}} & \meanstd{97.50}{0.58} & \meanstd{96.97}{1.02} & \meanstd{97.23}{0.67} \\
             \cline{6-8}
             & & & & & \meanstd{98.46}{0.50} & \meanstd{98.73}{0.28} & \meanstd{98.60}{0.32} \\
             \cline{2-8}
             & \multicolumn{2}{c|}{\multirow{2}{*}{PC (10)}} & \multirow{2}{*}{\meanstd{86.71}{0.96}} & \multirow{2}{*}{\meanstd{85.32}{1.07}} & \meanstd{79.96}{1.58} & \meanstd{81.06}{1.69} & \meanstd{80.50}{1.40} \\
             \cline{6-8}
             & \multicolumn{1}{l}{} & & & & \meanstd{90.24}{0.83} & \meanstd{89.59}{0.98} & \meanstd{89.91}{0.77} \\
            \midrule
            \multirow{14}{*}{LR} & \multirow{10}{*}{ZC} & \multirow{2}{*}{all (1000)} & \multirow{2}{*}{\meanstd{93.81}{0.54}} & \multirow{2}{*}{\meanstd{93.51}{0.60}} & \meanstd{89.53}{0.77} & \meanstd{92.57}{0.91} & \meanstd{91.02}{0.65} \\
            \cline{6-8}
            & & & & & \meanstd{96.12}{0.55} & \meanstd{94.45}{0.54} & \meanstd{95.28}{0.46} \\
            \cline{3-8}
            & & \multirow{2}{*}{all (normalized, 1000)} & \multirow{2}{*}{\meanstd{94.70}{0.63}} & \multirow{2}{*}{\meanstd{94.20}{0.70}} & \meanstd{91.77}{1.01} & \meanstd{92.65}{0.94} & \meanstd{92.21}{0.97} \\
            \cline{6-8}
            & & & & & \meanstd{96.23}{0.44} & \meanstd{95.75}{0.51} & \meanstd{95.99}{0.47} \\
            \cline{3-8}
            & & \multirow{2}{*}{squares and quintics (33)} & \multirow{2}{*}{\meanstd{94.47}{1.23}} & \multirow{2}{*}{\meanstd{93.91}{1.36}} & \meanstd{91.55}{1.91} & \meanstd{92.18}{1.89} & \meanstd{91.86}{1.78} \\
            \cline{6-8}
            & & & & & \meanstd{95.98}{0.99} & \meanstd{95.64}{1.02} & \meanstd{95.81}{0.94} \\
            \cline{3-8}
            & & \multirow{2}{*}{primes (168)} & \multirow{2}{*}{\meanstd{85.88}{0.50}} & \multirow{2}{*}{\meanstd{82.53}{0.80}} & \meanstd{83.91}{1.83} & \meanstd{72.16}{1.90} & \meanstd{77.56}{1.16} \\
            \cline{6-8}
            & & & & & \meanstd{86.71}{0.72} & \meanstd{92.91}{0.89} & \meanstd{89.70}{0.31} \\
            \cline{3-8}
            & & \multirow{2}{*}{prime powers (193)} & \multirow{2}{*}{\meanstd{96.57}{0.45}} & \multirow{2}{*}{\meanstd{96.14}{0.53}} & \meanstd{95.06}{1.23} & \meanstd{94.82}{1.26} & \meanstd{94.93}{0.61} \\
            \cline{6-8}
            & & & & & \meanstd{97.36}{0.63} & \meanstd{97.46}{0.72} & \meanstd{97.40}{0.35} \\
            \cline{2-8}
            & \multicolumn{2}{c|}{\multirow{2}{*}{PC (10)}} & \multirow{2}{*}{\meanstd{46.25}{4.62}} & \multirow{2}{*}{\meanstd{56.92}{4.79}} & \meanstd{37.71}{2.99} & \meanstd{89.82}{8.35} & \meanstd{53.05}{3.96} \\
            \cline{6-8}
            & \multicolumn{1}{l}{} & & & & \meanstd{83.46}{10.91} & \meanstd{24.02}{5.97} & \meanstd{36.84}{7.05} \\
            \cline{2-8}
            & \multicolumn{2}{c|}{\multirow{2}{*}{PC (normalized, 10)}} & \multirow{2}{*}{\meanstd{66.18}{0.79}} & \multirow{2}{*}{\meanstd{50.06}{0.11}} & \meanstd{20.00}{40.00} & \meanstd{0.11}{0.22} & \meanstd{0.22}{0.45} \\
            \cline{6-8}
            & \multicolumn{1}{l}{} & & & & \meanstd{66.17}{0.77} & \meanstd{100.00}{0.00} & \meanstd{79.64}{0.55} \\
            \bottomrule
        \end{tabular}
        \caption{ Performance of decision tree (DT) and logistic regression (LR) models on distinguishing Galois groups of decic extensions, using zeta coefficients (ZC) or polynomial coefficients (PC) as features, averaged over 5 different random splits. We report accuracy (ACC), balanced accuracy (BACC), per-class precision, recall, and F1. For the latter three metrics, top (resp. bottom) row corresponds to the class of $C_{10}$ (resp. $D_5$). The majority baseline is 66.14\% ($D_5$).}
        \label{tab:galois_10}
    \end{center}
\end{table}

\subsection{Octic fields} 

\begin{table}[H]
    \begin{center}
        \begin{tabular}{c|c|c|c|c|c|c|c|c|c|c|c|c}
            \toprule
             & \multicolumn{2}{c|}{$a_{2^2}$} & \multicolumn{2}{c|}{$a_{3^2}$} & \multicolumn{2}{c|}{$a_{5^2}$} & \multicolumn{2}{c|}{$a_{7^2}$} & \multicolumn{2}{c|}{$a_{11^2}$} & \multicolumn{2}{c}{$a_{13^2}$} \\
            \midrule
            $G$ & $D_4$ & $Q_8$ & $D_4$ & $Q_8$ & $D_4$ & $Q_8$ & $D_4$ & $Q_8$ & $D_4$ & $Q_8$ & $D_4$ & $Q_8$ \\
            \midrule
            0 & 0.11 & 0.24 & 0.21 & 0.30 & 0.26 & 0.38 & 0.24 & 0.41 & 0.28 & 0.46 & 0.30 & 0.50 \\
            1 & 0.32 & 0.48 & 0.10 & 0.50 & 0.13 & 0.09 & 0.04 & 0.38 & 0.02 & 0.32 & 0.06 & 0.07 \\
            2 & 0.17 & 0.06 & 0.18 & 0.10 & 0.05 & 0.41 & 0.10 & 0.08 & 0.07 & 0.07 & 0.01 & 0.29 \\
            3 & 0.10 & 0.15 & - & - & - & - & - & - & - & - & - & - \\
            4 & 0.17 & 0.03 & 0.37 & 0.03 & 0.41 & 0.05 & 0.50 & 0.05 & 0.52 & 0.07 & 0.51 & 0.06 \\
            10 & 0.11 & 0.01 & 0.11 & 0.03 & 0.09 & 0.02 & 0.07 & 0.02 & 0.04 & 0.02 & 0.04 & 0.02 \\
            36 & 0.02 & 0.03 & 0.04 & 0.03 & 0.05 & 0.05 & 0.06 & 0.06 & 0.07 & 0.07 & 0.08 & 0.06 \\
            \bottomrule
        \end{tabular}
        \caption{Distribution of $a_{p^2}(K)$ for nonabelian octic extensions and $p = 2, 3, 5, 7, 11, 13$. Note that $a_{p^2}(K) \ne 3$ for odd $p$.}
        \label{tab:octic_nab_psq_dist}
    \end{center}
\end{table}

\begin{table}[H]
    \begin{center}
        \begin{tabular}{c|c|c|c|c|c|c|c|c|c|c|c|c}
            \toprule
             & \multicolumn{2}{c|}{$a_{2^2}$} & \multicolumn{2}{c|}{$a_{3^2}$} & \multicolumn{2}{c|}{$a_{5^2}$} & \multicolumn{2}{c|}{$a_{7^2}$} & \multicolumn{2}{c|}{$a_{11^2}$} & \multicolumn{2}{c}{$a_{13^2}$} \\
            \midrule
            $G$ & $D_4$ & $Q_8$ & $D_4$ & $Q_8$ & $D_4$ & $Q_8$ & $D_4$ & $Q_8$ & $D_4$ & $Q_8$ & $D_4$ & $Q_8$ \\
            \midrule
            0 & 0.21 & 0.79 & 0.28 & 0.72 & 0.28 & 0.72 & 0.24 & 0.76 & 0.25 & 0.75 & 0.25 & 0.75 \\
            1 & 0.27 & 0.73 & 0.10 & 0.90 & - & - & 0.05 & 0.95 & 0.03 & 0.97 & - & - \\
            2 & 0.62 & 0.38 & 0.49 & 0.51 & 0.45 & 0.55 & 0.41 & 0.59 & 0.35 & 0.65 & 0.33 & 0.67 \\
            3 & 0.27 & 0.73 & - & - & 0.07 & 0.93 & - & - & - & - & 0.03 & 0.97 \\
            4 & 0.77 & 0.23 & 0.86 & 0.14 & 0.83 & 0.17 & 0.84 & 0.16 & 0.81 & 0.19 & 0.82 & 0.18 \\
            10 & 0.80 & 0.20 & 0.68 & 0.32 & 0.66 & 0.34 & 0.64 & 0.36 & 0.52 & 0.48 & 0.54 & 0.46 \\
            36 & 0.29 & 0.71 & 0.38 & 0.62 & 0.38 & 0.62 & 0.36 & 0.64 & 0.36 & 0.64 & 0.41 & 0.59 \\
            \bottomrule
        \end{tabular}
        \caption{Distribution of Galois group $G \in \{D_4, Q_8\}$ for octic nonabelian extensions and given $a_{p^2}(K)$ with $p = 2, 3, 5, 7, 11, 13$.}
        \label{tab:octic_nab_psq_dist_G}
    \end{center}
\end{table}

\begin{table}[H]
    \begin{center}
        \begin{tabular}{c|c|c|c|c|c|c|c}
            \toprule
            Model & \multicolumn{2}{c|}{Feature} & ACC & BACC & Precision & Recall & F1 \\
            \midrule
            \multirow{15}{*}{DT} & \multirow{12}{*}{ZC} & \multirow{3}{*}{all (1000)} & \multirow{3}{*}{\meanstd{98.61}{0.16}} & \multirow{3}{*}{\meanstd{98.17}{0.26}} & \meanstd{96.06}{0.84} & \meanstd{95.79}{0.79} & \meanstd{95.92}{0.48} \\
             \cline{6-8}
             & & & & & \meanstd{98.66}{0.27} & \meanstd{98.74}{0.22} & \meanstd{98.70}{0.14} \\
             \cline{6-8}
             & & & & & \meanstd{99.96}{0.07} & \meanstd{100.00}{0.00} & \meanstd{99.98}{0.04} \\
             \cline{3-8}
             & & \multirow{3}{*}{squares (31)} & \multirow{3}{*}{\meanstd{98.91}{0.18}} & \multirow{3}{*}{\meanstd{98.66}{0.25}} & \meanstd{96.59}{0.55} & \meanstd{97.11}{0.63} & \meanstd{96.85}{0.49} \\
             \cline{6-8}
             & & & & & \meanstd{99.08}{0.22} & \meanstd{98.88}{0.17} & \meanstd{98.98}{0.17} \\
             \cline{6-8}
             & & & & & \meanstd{99.94}{0.07} & \meanstd{100.00}{0.00} & \meanstd{99.97}{0.03} \\
             \cline{3-8}
             & & \multirow{3}{*}{primes (168)} & \multirow{3}{*}{\meanstd{77.43}{0.49}} & \multirow{3}{*}{\meanstd{74.65}{0.45}} & \meanstd{70.84}{0.66} & \meanstd{65.27}{1.37} & \meanstd{67.93}{0.78} \\
             \cline{6-8}
             & & & & & \meanstd{80.50}{0.30} & \meanstd{81.34}{0.81} & \meanstd{80.91}{0.48} \\
             \cline{6-8}
             & & & & & \meanstd{75.37}{1.41} & \meanstd{77.34}{0.78} & \meanstd{76.34}{0.98} \\
             \cline{3-8}
             & & \multirow{3}{*}{prime powers (193)} & \multirow{3}{*}{\meanstd{98.82}{0.16}} & \multirow{3}{*}{\meanstd{98.48}{0.29}} & \meanstd{96.49}{0.71} & \meanstd{96.56}{0.94} & \meanstd{96.52}{0.45} \\
             \cline{6-8}
             & & & & & \meanstd{98.91}{0.32} & \meanstd{98.89}{0.23} & \meanstd{98.90}{0.15} \\
             \cline{6-8}
             & & & & & \meanstd{100.00}{0.00} & \meanstd{100.00}{0.00} & \meanstd{100.00}{0.00} \\
             \cline{2-8}
             & \multicolumn{2}{c|}{\multirow{3}{*}{PC (8)}} & \multirow{3}{*}{\meanstd{77.21}{0.56}} & \multirow{3}{*}{\meanstd{77.48}{0.60}} & \meanstd{83.02}{0.74} & \meanstd{82.86}{0.74} & \meanstd{82.94}{0.51} \\
             \cline{6-8}
             & \multicolumn{1}{l}{} & & & & \meanstd{78.72}{0.55} & \meanstd{79.16}{0.43} & \meanstd{78.94}{0.42} \\
             \cline{6-8}
             & \multicolumn{1}{l}{} & & & & \meanstd{71.05}{1.07} & \meanstd{70.42}{1.50} & \meanstd{70.73}{1.17} \\
            \midrule
            \multirow{21}{*}{LR} & \multirow{15}{*}{ZC} & \multirow{3}{*}{all (1000)} & \multirow{3}{*}{\meanstd{93.81}{0.13}} & \multirow{3}{*}{\meanstd{92.04}{0.19}} & \meanstd{89.03}{1.10} & \meanstd{83.45}{0.65} & \meanstd{86.14}{0.21} \\
             \cline{6-8}
             & & & & & \meanstd{93.75}{0.31} & \meanstd{94.94}{0.41} & \meanstd{94.34}{0.16} \\
             \cline{6-8}
             & & & & & \meanstd{96.46}{0.19} & \meanstd{97.71}{0.37} & \meanstd{97.08}{0.24} \\
             \cline{3-8}
             & & \multirow{3}{*}{all (normalized, 1000)} & \multirow{3}{*}{\meanstd{92.88}{0.20}} & \multirow{3}{*}{\meanstd{90.51}{0.15}} & \meanstd{85.31}{1.25} & \meanstd{78.99}{0.74} & \meanstd{82.02}{0.29} \\
             \cline{6-8}
             & & & & & \meanstd{92.56}{0.28} & \meanstd{94.38}{0.58} & \meanstd{93.46}{0.22} \\
             \cline{6-8}
             & & & & & \meanstd{97.50}{0.28} & \meanstd{98.16}{0.28} & \meanstd{97.83}{0.21} \\
             \cline{3-8}
             & & \multirow{3}{*}{squares (31)} & \multirow{3}{*}{\meanstd{85.08}{0.73}} & \multirow{3}{*}{\meanstd{81.87}{0.95}} & \meanstd{83.45}{2.29} & \meanstd{64.86}{1.13} & \meanstd{72.98}{1.37} \\
             \cline{6-8}
             & & & & & \meanstd{86.29}{1.12} & \meanstd{86.36}{0.52} & \meanstd{86.32}{0.64} \\
             \cline{6-8}
             & & & & & \meanstd{83.83}{0.45} & \meanstd{94.41}{1.96} & \meanstd{88.79}{0.66} \\
             \cline{3-8}
             & & \multirow{3}{*}{primes (168)} & \multirow{3}{*}{\meanstd{57.99}{0.39}} & \multirow{3}{*}{\meanstd{46.31}{0.45}} & \meanstd{49.11}{2.42} & \meanstd{19.04}{0.54} & \meanstd{27.43}{0.80} \\
             \cline{6-8}
             & & & & & \meanstd{59.73}{0.36} & \meanstd{80.72}{0.64} & \meanstd{68.65}{0.28} \\
             \cline{6-8}
             & & & & & \meanstd{54.81}{1.33} & \meanstd{39.18}{1.00} & \meanstd{45.69}{1.10} \\
             \cline{3-8}
             & & \multirow{3}{*}{prime powers (193)} & \multirow{3}{*}{\meanstd{94.86}{0.28}} & \multirow{3}{*}{\meanstd{92.99}{0.35}} & \meanstd{90.70}{1.38} & \meanstd{84.06}{0.91} & \meanstd{87.24}{0.56} \\
             \cline{6-8}
             & & & & & \meanstd{94.39}{0.42} & \meanstd{96.15}{0.46} & \meanstd{95.26}{0.29} \\
             \cline{6-8}
             & & & & & \meanstd{97.93}{0.30} & \meanstd{98.77}{0.28} & \meanstd{98.35}{0.22} \\
             \cline{2-8}
             & \multicolumn{2}{c|}{\multirow{3}{*}{PC (8)}} & \multirow{3}{*}{\meanstd{19.11}{2.19}} & \multirow{3}{*}{\meanstd{32.43}{0.42}} & \meanstd{16.45}{0.64} & \meanstd{77.52}{20.78} & \meanstd{26.85}{1.94} \\
             \cline{6-8}
             & \multicolumn{1}{l}{} & & & & \meanstd{31.72}{25.96} & \meanstd{0.28}{0.25} & \meanstd{0.55}{0.50} \\
             \cline{6-8}
             & \multicolumn{1}{l}{} & & & & \meanstd{34.93}{5.83} & \meanstd{19.51}{19.91} & \meanstd{17.39}{13.96} \\
             \cline{2-8}
             & \multicolumn{2}{c|}{\multirow{3}{*}{PC (normalized, 8)}} & \multirow{3}{*}{\meanstd{64.89}{0.44}} & \multirow{3}{*}{\meanstd{54.28}{0.42}} & \meanstd{90.86}{0.31} & \meanstd{34.75}{1.16} & \meanstd{50.26}{1.23} \\
             \cline{6-8}
             & \multicolumn{1}{l}{} & & & & \meanstd{62.31}{0.63} & \meanstd{88.26}{0.47} & \meanstd{73.05}{0.46} \\
             \cline{6-8}
             & \multicolumn{1}{l}{} & & & & \meanstd{66.40}{1.01} & \meanstd{39.82}{0.45} & \meanstd{49.78}{0.44} \\
            \bottomrule
        \end{tabular}
        \caption{ Performance of decision tree (DT) and logistic regression (LR) models on distinguishing Galois groups of abelian octic extensions, using zeta coefficients (ZC) or polynomial coefficients (PC) as features, averaged over 5 different random splits. We report accuracy (ACC), balanced accuracy (BACC), per-class precision, recall, and F1. For the latter three metrics, each of three rows correspond to the class of $C_8$, $C_4 \times C_2$, $C_2^3$. The majority baseline is 54.03\% ($C_4 \times C_2$).}
        \label{tab:galois_8_ab}
    \end{center}
\end{table}

\begin{table}[H]
    \begin{center}
        \begin{tabular}{c|c|c|c|c|c|c|c}
            \toprule
            Model & \multicolumn{2}{c|}{Feature} & ACC & BACC & Precision & Recall & F1 \\
            \midrule
            \multirow{10}{*}{DT} & \multirow{8}{*}{ZC} & \multirow{2}{*}{all (1000)} & \multirow{2}{*}{\meanstd{98.67}{0.08}} & \multirow{2}{*}{\meanstd{98.52}{0.10}} & \meanstd{98.25}{0.05} & \meanstd{98.00}{0.19} & \meanstd{98.13}{0.10} \\
             \cline{6-8}
             & & & & & \meanstd{98.89}{0.12} & \meanstd{99.04}{0.04} & \meanstd{98.96}{0.07} \\
             \cline{3-8}
             & & \multirow{2}{*}{squares (31)} & \multirow{2}{*}{\meanstd{99.17}{0.09}} & \multirow{2}{*}{\meanstd{99.09}{0.11}} & \meanstd{98.82}{0.08} & \meanstd{98.84}{0.20} & \meanstd{98.83}{0.11} \\
             \cline{6-8}
             & & & & & \meanstd{99.36}{0.12} & \meanstd{99.35}{0.05} & \meanstd{99.35}{0.07} \\
             \cline{3-8}
             & & \multirow{2}{*}{primes (168)} & \multirow{2}{*}{\meanstd{84.50}{0.43}} & \multirow{2}{*}{\meanstd{83.16}{0.48}} & \meanstd{78.07}{0.42} & \meanstd{78.52}{0.87} & \meanstd{78.30}{0.51} \\
             \cline{6-8}
             & & & & & \meanstd{88.08}{0.58} & \meanstd{87.80}{0.31} & \meanstd{87.94}{0.38} \\
             \cline{3-8}
             & & \multirow{2}{*}{prime powers (193)} & \multirow{2}{*}{\meanstd{98.85}{0.06}} & \multirow{2}{*}{\meanstd{98.73}{0.07}} & \meanstd{98.46}{0.24} & \meanstd{98.32}{0.20} & \meanstd{98.39}{0.08} \\
             \cline{6-8}
             & & & & & \meanstd{99.07}{0.12} & \meanstd{99.15}{0.13} & \meanstd{99.11}{0.05} \\
             \cline{2-8}
             & \multicolumn{2}{c|}{\multirow{2}{*}{PC (8)}} & \multirow{2}{*}{\meanstd{98.76}{0.11}} & \multirow{2}{*}{\meanstd{98.64}{0.13}} & \meanstd{98.29}{0.14} & \meanstd{98.23}{0.21} & \meanstd{98.26}{0.16} \\
             \cline{6-8}
             & \multicolumn{1}{l}{} & & & & \meanstd{99.02}{0.12} & \meanstd{99.05}{0.07} & \meanstd{99.04}{0.08} \\
            \midrule
            \multirow{14}{*}{LR} & \multirow{10}{*}{ZC} & \multirow{2}{*}{all (1000)} & \multirow{2}{*}{\meanstd{97.51}{0.13}} & \multirow{2}{*}{\meanstd{97.16}{0.16}} & \meanstd{97.07}{0.14} & \meanstd{95.91}{0.25} & \meanstd{96.49}{0.19} \\
             \cline{6-8}
             & & & & & \meanstd{97.75}{0.15} & \meanstd{98.40}{0.07} & \meanstd{98.08}{0.11} \\
             \cline{3-8}
             & & \multirow{2}{*}{all (normalized, 1000)} & \multirow{2}{*}{\meanstd{97.54}{0.10}} & \multirow{2}{*}{\meanstd{97.16}{0.14}} & \meanstd{97.21}{0.21} & \meanstd{95.84}{0.33} & \meanstd{96.52}{0.12} \\
             \cline{6-8}
             & & & & & \meanstd{97.72}{0.21} & \meanstd{98.48}{0.10} & \meanstd{98.09}{0.08} \\
             \cline{3-8}
             & & \multirow{2}{*}{square (31)} & \multirow{2}{*}{\meanstd{77.73}{0.21}} & \multirow{2}{*}{\meanstd{72.99}{0.18}} & \meanstd{74.76}{0.66} & \meanstd{56.55}{0.46} & \meanstd{64.39}{0.31} \\
             \cline{6-8}
             & & & & & \meanstd{78.82}{0.39} & \meanstd{89.44}{0.30} & \meanstd{83.79}{0.21} \\
             \cline{3-8}
             & & \multirow{2}{*}{primes (168)} & \multirow{2}{*}{\meanstd{71.93}{0.23}} & \multirow{2}{*}{\meanstd{64.78}{0.10}} & \meanstd{68.04}{0.85} & \meanstd{39.94}{0.43} & \meanstd{50.33}{0.23} \\
             \cline{6-8}
             & & & & & \meanstd{72.96}{0.43} & \meanstd{89.62}{0.37} & \meanstd{80.43}{0.24} \\
             \cline{3-8}
             & & \multirow{2}{*}{prime powers (193)} & \multirow{2}{*}{\meanstd{97.76}{0.04}} & \multirow{2}{*}{\meanstd{97.39}{0.05}} & \meanstd{97.55}{0.14} & \meanstd{96.12}{0.15} & \meanstd{96.83}{0.03} \\
             \cline{6-8}
             & & & & & \meanstd{97.87}{0.09} & \meanstd{98.67}{0.09} & \meanstd{98.27}{0.04} \\
             \cline{2-8}
             & \multicolumn{2}{c|}{\multirow{2}{*}{PC (8)}} & \multirow{2}{*}{\meanstd{68.21}{0.47}} & \multirow{2}{*}{\meanstd{55.37}{0.18}} & \meanstd{99.57}{0.30} & \meanstd{10.78}{0.36} & \meanstd{19.45}{0.59} \\
             \cline{6-8}
             & \multicolumn{1}{l}{} & & & & \meanstd{66.95}{0.48} & \meanstd{99.97}{0.02} & \meanstd{80.20}{0.35} \\
             \cline{2-8}
             & \multicolumn{2}{c|}{\multirow{2}{*}{PC (normalized, 8)}} & \multirow{2}{*}{\meanstd{84.33}{0.38}} & \multirow{2}{*}{\meanstd{87.78}{0.30}} & \meanstd{69.51}{0.61} & \meanstd{99.76}{0.04} & \meanstd{81.93}{0.42} \\
             \cline{6-8}
             & \multicolumn{1}{l}{} & & & & \meanstd{99.82}{0.03} & \meanstd{75.80}{0.59} & \meanstd{86.17}{0.38} \\
            \bottomrule
        \end{tabular}
        \caption{ Performance of decision tree (DT) and logistic regression (LR) models on distinguishing Galois groups of nonabelian octic extensions, using zeta coefficients (ZC) or polynomial coefficients (PC) as features, averaged over 5 different random splits. We report accuracy (ACC), balanced accuracy (BACC), per-class precision, recall, and F1. For the latter three metrics, top (resp. bottom) row corresponds to the class of $D_4$ (resp. $Q_8$). The majority baseline is 64.39\% ($Q_8$).}
        \label{tab:galois_8_nab}
    \end{center}
\end{table}

\begin{table}[H]
    \begin{center}
        \begin{tabular}{c|c|c|c|c|c|c|c}
            \toprule
            Model & \multicolumn{2}{c|}{Feature} & ACC & BACC & Precision & Recall & F1 \\
            \midrule
            \multirow{10}{*}{DT} & \multirow{8}{*}{ZC} & \multirow{2}{*}{all (1000)} & \multirow{2}{*}{\meanstd{96.58}{0.10}} & \multirow{2}{*}{\meanstd{96.14}{0.15}} & \meanstd{94.37}{0.28} & \meanstd{94.93}{0.37} & \meanstd{94.65}{0.13} \\
             \cline{6-8}
             & & & & & \meanstd{97.62}{0.20} & \meanstd{97.35}{0.11} & \meanstd{97.48}{0.08} \\
             \cline{3-8}
             & & \multirow{2}{*}{squares (31)} & \multirow{2}{*}{\meanstd{97.20}{0.09}} & \multirow{2}{*}{\meanstd{96.91}{0.12}} & \meanstd{95.16}{0.29} & \meanstd{96.10}{0.25} & \meanstd{95.63}{0.16} \\
             \cline{6-8}
             & & & & & \meanstd{98.17}{0.13} & \meanstd{97.71}{0.11} & \meanstd{97.94}{0.07} \\
             \cline{3-8}
             & & \multirow{2}{*}{primes (168)} & \multirow{2}{*}{\meanstd{77.65}{0.21}} & \multirow{2}{*}{\meanstd{75.11}{0.33}} & \meanstd{64.06}{0.58} & \meanstd{68.09}{0.88} & \meanstd{66.01}{0.40} \\
             \cline{6-8}
             & & & & & \meanstd{84.62}{0.45} & \meanstd{82.13}{0.40} & \meanstd{83.35}{0.17} \\
             \cline{3-8}
             & & \multirow{2}{*}{prime powers (193)} & \multirow{2}{*}{\meanstd{96.64}{0.07}} & \multirow{2}{*}{\meanstd{96.25}{0.13}} & \meanstd{94.36}{0.19} & \meanstd{95.15}{0.33} & \meanstd{94.75}{0.14} \\
             \cline{6-8}
             & & & & & \meanstd{97.72}{0.14} & \meanstd{97.34}{0.10} & \meanstd{97.53}{0.05} \\
             \cline{2-8}
             & \multicolumn{2}{c|}{\multirow{2}{*}{PC (8)}} & \multirow{2}{*}{\meanstd{84.08}{0.26}} & \multirow{2}{*}{\meanstd{81.92}{0.28}} & \meanstd{74.59}{0.26} & \meanstd{75.94}{0.42} & \meanstd{75.26}{0.32} \\
             \cline{6-8}
             & \multicolumn{1}{l}{} & & & & \meanstd{88.64}{0.29} & \meanstd{87.89}{0.17} & \meanstd{88.27}{0.23} \\
            \midrule
            \multirow{14}{*}{LR} & \multirow{10}{*}{ZC} & \multirow{2}{*}{all (1000)} & \multirow{2}{*}{\meanstd{76.27}{0.25}} & \multirow{2}{*}{\meanstd{69.61}{0.47}} & \meanstd{66.61}{0.87} & \meanstd{51.23}{0.92} & \meanstd{57.91}{0.84} \\
             \cline{6-8}
             & & & & & \meanstd{79.41}{0.25} & \meanstd{87.99}{0.24} & \meanstd{83.48}{0.14} \\
             \cline{3-8}
             & & \multirow{2}{*}{all (normalized, 1000)} & \multirow{2}{*}{\meanstd{75.83}{0.26}} & \multirow{2}{*}{\meanstd{68.92}{0.43}} & \meanstd{66.00}{0.85} & \meanstd{49.87}{0.78} & \meanstd{56.81}{0.76} \\
             \cline{6-8}
             & & & & & \meanstd{78.95}{0.26} & \meanstd{87.98}{0.22} & \meanstd{83.22}{0.16} \\
             \cline{3-8}
             & & \multirow{2}{*}{squares (31)} & \multirow{2}{*}{\meanstd{67.88}{0.27}} & \multirow{2}{*}{\meanstd{51.12}{0.15}} & \meanstd{46.38}{2.25} & \meanstd{4.88}{0.20} & \meanstd{8.83}{0.37} \\
             \cline{6-8}
             & & & & & \meanstd{68.63}{0.32} & \meanstd{97.36}{0.14} & \meanstd{80.50}{0.19} \\
             \cline{3-8}
             & & \multirow{2}{*}{primes (168)} & \multirow{2}{*}{\meanstd{70.69}{0.23}} & \multirow{2}{*}{\meanstd{57.43}{0.29}} & \meanstd{61.94}{1.55} & \meanstd{20.87}{0.44} & \meanstd{31.21}{0.62} \\
             \cline{6-8}
             & & & & & \meanstd{71.74}{0.28} & \meanstd{94.00}{0.27} & \meanstd{81.37}{0.16} \\
             \cline{3-8}
             & & \multirow{2}{*}{prime powers (193)} & \multirow{2}{*}{\meanstd{76.01}{0.38}} & \multirow{2}{*}{\meanstd{69.02}{0.65}} & \meanstd{66.54}{0.40} & \meanstd{49.73}{1.56} & \meanstd{56.91}{1.07} \\
             \cline{6-8}
             & & & & & \meanstd{78.97}{0.58} & \meanstd{88.30}{0.29} & \meanstd{83.37}{0.23} \\
             \cline{2-8}
             & \multicolumn{2}{c|}{\multirow{2}{*}{PC (8)}} & \multirow{2}{*}{\meanstd{66.41}{0.26}} & \multirow{2}{*}{\meanstd{49.51}{0.10}} & \meanstd{25.85}{1.33} & \meanstd{2.88}{0.11} & \meanstd{5.18}{0.20} \\
             \cline{6-8}
             & \multicolumn{1}{l}{} & & & & \meanstd{67.90}{0.31} & \meanstd{96.13}{0.13} & \meanstd{79.59}{0.00} \\
             \cline{2-8}
             & \multicolumn{2}{c|}{\multirow{2}{*}{PC (normalized, 8)}} & \multirow{2}{*}{\meanstd{68.02}{0.31}} & \multirow{2}{*}{\meanstd{50.01}{0.08}} & \meanstd{29.98}{11.85} & \meanstd{0.31}{0.20} & \meanstd{0.61}{0.40} \\
             \cline{6-8}
             & \multicolumn{1}{l}{} & & & & \meanstd{68.13}{0.33} & \meanstd{99.71}{0.07} & \meanstd{80.95}{0.22} \\
            \bottomrule
        \end{tabular}
        \caption{ Performance of decision tree (DT) and logistic regression (LR) models on distinguishing abelian and nonabelian octic extensions, using zeta coefficients (ZC) or polynomial coefficients (PC) as features, averaged over 5 different random splits. We report accuracy (ACC), balanced accuracy (BACC), per-class precision, recall, and F1. For the latter three metrics, top (resp. bottom) row corresponds to the class of abelian (resp. nonabelian) extensions. The majority baseline is 67.89\% (nonabelian).}
        \label{tab:galois_8_ab_nab}
    \end{center}
\end{table}

\section{Figures}

\subsection{Quartic and nonic fields}

\begin{figure}[h]
    \centering
    \begin{subfigure}[t]{0.23\textwidth}
        \centering
        \scalebox{0.25}{
        \BinaryConfusionMatrix
          {$C_4$}{$C_2^2$}
          {4373.80}{0.80}
          {0.00}{32197.40}
        }
        \caption{DT, ZC}
        \label{fig:quartic_dt_zc}
    \end{subfigure}
    \hfill
    \begin{subfigure}[t]{0.23\textwidth}
        \centering
        \scalebox{0.25}{
        \BinaryConfusionMatrix
          {$C_4$}{$C_2^2$}
          {3420.40}{954.20}
          {817.20}{31380.20}
        }
        \caption{DT, PC}
        \label{fig:quartic_dt_pc}
    \end{subfigure}
    \hfill
    \begin{subfigure}[t]{0.23\textwidth}
        \centering
        \scalebox{0.25}{
        \BinaryConfusionMatrix
          {$C_4$}{$C_2^2$}
          {4233.40}{141.20}
          {74.80}{32212.60}
        }
        \caption{LR, normalized ZC}
        \label{fig:quartic_lr_zc_normalized}
    \end{subfigure}
    \hfill
    \begin{subfigure}[t]{0.23\textwidth}
        \centering
        \scalebox{0.25}{
        \BinaryConfusionMatrix
          {$C_4$}{$C_2^2$}
          {7.00}{4367.60}
          {0.00}{32197.40}
        }
        \caption{LR, normalized PC}
        \label{fig:quartic_lr_pc_normalized}
    \end{subfigure}
    \caption{Confusion matrices for the quartic Galois field experiments with (a) decision tree, using all ($n \le 1000$) zeta coefficients, (b) decision tree, using polynomial coefficients, (c) logistic regression, using all ($n \le 1000$) normalized zeta coefficients, (d) logistic regression, using normalized polynomial coefficients.}
    \label{fig:galois_4_cm}
\end{figure}

\begin{figure}[h]
    \centering
    \begin{subfigure}[t]{0.23\textwidth}
        \centering
        \scalebox{0.25}{
        \BinaryConfusionMatrix
          {9T1}{9T2}
          {56.80}{0.00}
          {0.00}{196.40}
        }
        \caption{DT, ZC}
        \label{fig:nonic_dt_zc}
    \end{subfigure}
    \hfill
    \begin{subfigure}[t]{0.23\textwidth}
        \centering
        \scalebox{0.25}{
        \BinaryConfusionMatrix
          {9T1}{9T2}
          {50.40}{6.80}
          {8.00}{188.40}
        }
        \caption{DT, PC}
        \label{fig:nonic_dt_pc}
    \end{subfigure}
    \hfill
    \begin{subfigure}[t]{0.23\textwidth}
        \centering
        \scalebox{0.25}{
        \BinaryConfusionMatrix
          {9T1}{9T2}
          {26.20}{30.60}
          {23.80}{172.60}
        }
        \caption{LR, normalized ZC}
        \label{fig:nonic_lr_zc_normalized}
    \end{subfigure}
    \hfill
    \begin{subfigure}[t]{0.23\textwidth}
        \centering
        \scalebox{0.25}{
        \BinaryConfusionMatrix
          {9T1}{9T2}
          {21.00}{35.80}
          {1.00}{195.40}
        }
        \caption{LR, normalized PC}
        \label{fig:nonic_lr_pc_normalized}
    \end{subfigure}
    \caption{Confusion matrices for the nonic Galois field experiments with (a) decision tree, using all ($n \le 1000$) zeta coefficients, (b) decision tree, using polynomial coefficients, (c) logistic regression, using all ($n \le 1000$) normalized zeta coefficients, (d) logistic regression, using normalized polynomial coefficients.}
    \label{fig:galois_9_cm}
\end{figure}

\begin{figure}[H]
    \centering
    \begin{subfigure}[b]{0.475\textwidth}
        \centering
        \includegraphics[width=\textwidth]{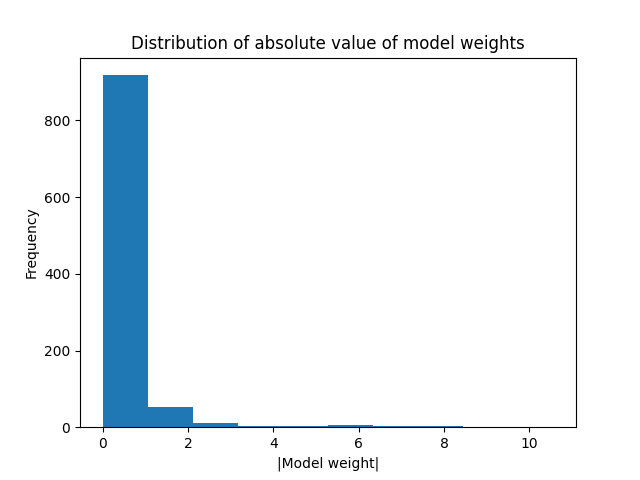}
        \caption[]%
        {{\small Quartic fields, the whole set of weights}}    
        \label{fig:mean and std of net14}
    \end{subfigure}
    \hfill
    \begin{subfigure}[b]{0.475\textwidth}  
        \centering 
        \includegraphics[width=\textwidth]{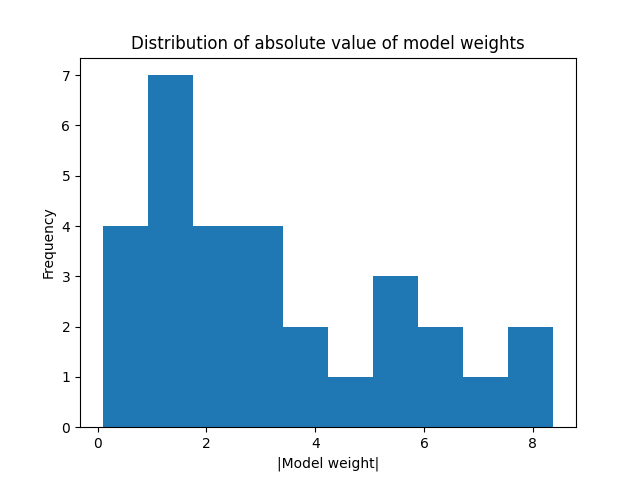}
        \caption[]%
        {{\small Quartic fields, square-indexed weights}}    
        \label{fig:mean and std of net24}
    \end{subfigure}
    \vskip\baselineskip
    \begin{subfigure}[b]{0.475\textwidth}   
        \centering 
        \includegraphics[width=\textwidth]{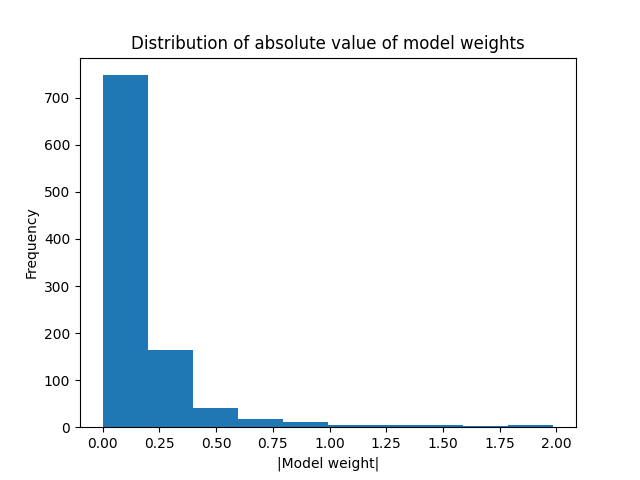}
        \caption[]%
        {{\small Nonic fields, the whole set of weights}}    
        \label{fig:mean and std of net34}
    \end{subfigure}
    \hfill
    \begin{subfigure}[b]{0.475\textwidth}   
        \centering 
        \includegraphics[width=\textwidth]{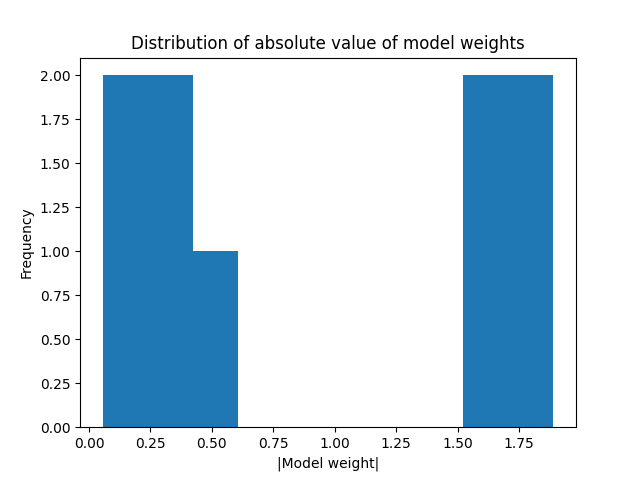}
        \caption[]%
        {{\small Nonic fields, cube-indexed weights}}    
        \label{fig:mean and std of net44}
    \end{subfigure}
    \caption[]
    {\small Distribution of the absolute values of the logistic regression model weights  trained on normalized zeta coefficients of quartic and nonic Galois extensions.} 
    \label{fig:galois_4_9_lr_weights}
\end{figure}

\subsection{Sextic and decic fields}

\begin{figure}[h]
    \centering
    \begin{subfigure}[t]{0.23\textwidth}
        \centering
        \scalebox{0.25}{
        \BinaryConfusionMatrix
          {$C_{6}$}{$S_{3}$}
          {4084.20}{145.00}
          {181.00}{27023.40}
        }
        \caption{DT, ZC}
        \label{fig:sextic_dt_zc}
    \end{subfigure}
    \hfill
    \begin{subfigure}[t]{0.23\textwidth}
        \centering
        \scalebox{0.25}{
        \BinaryConfusionMatrix
          {$C_{6}$}{$S_{3}$}
          {3834.80}{394.40}
          {373.40}{26831.00}
        }
        \caption{DT, PC}
        \label{fig:sextic_dt_pc}
    \end{subfigure}
    \hfill
    \begin{subfigure}[t]{0.23\textwidth}
        \centering
        \scalebox{0.25}{
        \BinaryConfusionMatrix
          {$C_{6}$}{$S_{3}$}
          {3790.80}{408.40}
          {252.80}{26951.60}
        }
        \caption{LR, normalized ZC}
        \label{fig:sextic_lr_zc_normalized}
    \end{subfigure}
    \hfill
    \begin{subfigure}[t]{0.23\textwidth}
        \centering
        \scalebox{0.25}{
        \BinaryConfusionMatrix
          {$C_{6}$}{$S_{3}$}
          {59.60}{4196.60}
          {2.00}{27202.40}
        }
        \caption{LR, normalized PC}
        \label{fig:sextic_lr_pc_normalized}
    \end{subfigure}
    \caption{Confusion matrices for the sextic Galois field experiments with (a) decision tree, using all ($n \le 1000$) zeta coefficients, (b) decision tree, using polynomial coefficients, (c) logistic regression, using all ($n \le 1000$) normalized zeta coefficients, (d) logistic regression, using normalized polynomial coefficients.}
    \label{fig:galois_6_cm}
\end{figure}

\begin{figure}[h]
    \centering
    \begin{subfigure}[t]{0.23\textwidth}
        \centering
        \scalebox{0.25}{
        \BinaryConfusionMatrix
          {$C_{10}$}{$D_{5}$}
          {356.40}{14.40}
          {16.40}{708.00}
        }
        \caption{DT, ZC}
        \label{fig:decic_dt_zc}
    \end{subfigure}
    \hfill
    \begin{subfigure}[t]{0.23\textwidth}
        \centering
        \scalebox{0.25}{
        \BinaryConfusionMatrix
          {$C_{10}$}{$D_{5}$}
          {300.60}{70.20}
          {75.40}{649.00}
        }
        \caption{DT, PC}
        \label{fig:decic_dt_pc}
    \end{subfigure}
    \hfill
    \begin{subfigure}[t]{0.23\textwidth}
        \centering
        \scalebox{0.25}{
        \BinaryConfusionMatrix
          {$C_{10}$}{$D_{5}$}
          {343.60}{27.20}
          {30.80}{693.60}
        }
        \caption{LR, normalized ZC}
        \label{fig:decic_lr_zc_normalized}
    \end{subfigure}
    \hfill
    \begin{subfigure}[t]{0.23\textwidth}
        \centering
        \scalebox{0.25}{
        \BinaryConfusionMatrix
          {$C_{10}$}{$D_{5}$}
          {0.40}{370.40}
          {0.00}{724.40}
        }
        \caption{LR, normalized PC}
        \label{fig:decic_lr_pc_normalized}
    \end{subfigure}
    \caption{Confusion matrices for the decic Galois field experiments with (a) decision tree, using all ($n \le 1000$) zeta coefficients, (b) decision tree, using polynomial coefficients, (c) logistic regression, using all ($n \le 1000$) normalized zeta coefficients, (d) logistic regression, using normalized polynomial coefficients.}
    \label{fig:galois_10_cm}
\end{figure}

\begin{figure}[H]
    \centering
    \begin{subfigure}[b]{0.475\textwidth}
        \centering
        \includegraphics[width=\textwidth]{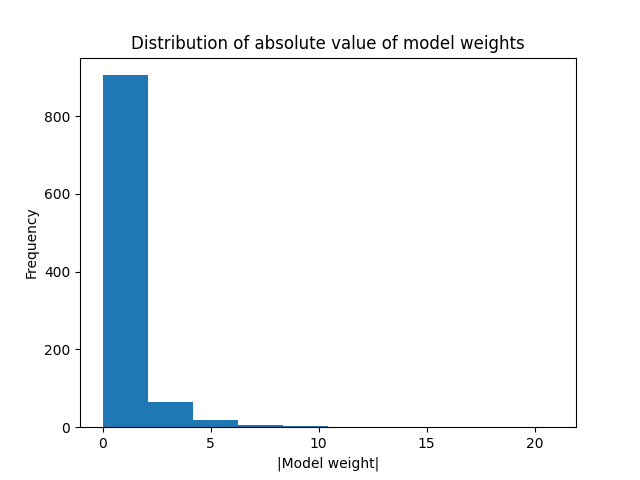}
        \caption[]%
        {{\small Sextic fields, the whole set of weights}}    
        \label{fig:mean and std of net14}
    \end{subfigure}
    \hfill
    \begin{subfigure}[b]{0.475\textwidth}  
        \centering 
        \includegraphics[width=\textwidth]{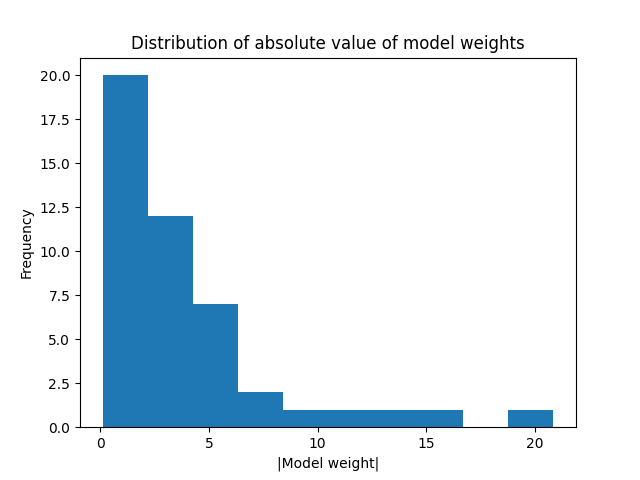}
        \caption[]%
        {{\small Sextic fields, square-indexed weights}}    
        \label{fig:mean and std of net24}
    \end{subfigure}
    \vskip\baselineskip
    \begin{subfigure}[b]{0.475\textwidth}   
        \centering 
        \includegraphics[width=\textwidth]{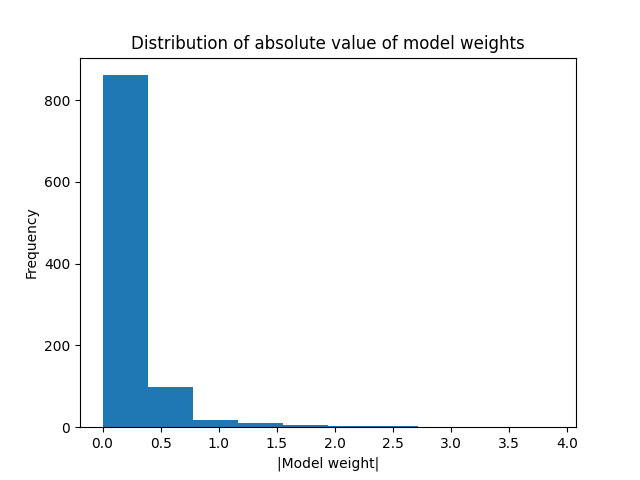}
        \caption[]%
        {{\small Decic fields, the whole set of weights}}    
        \label{fig:mean and std of net34}
    \end{subfigure}
    \hfill
    \begin{subfigure}[b]{0.475\textwidth}   
        \centering 
        \includegraphics[width=\textwidth]{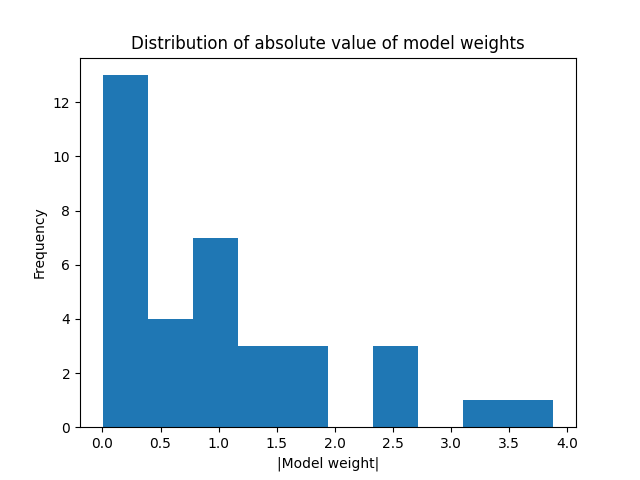}
        \caption[]%
        {{\small Decic fields, cube-indexed weights}}    
        \label{fig:mean and std of net44}
    \end{subfigure}
    \caption[]
    {\small Distribution of the absolute values of the logistic regression model weights  trained on normalized zeta coefficients of sextic and decic Galois extensions.} 
    \label{fig:galois_6_10_lr_weights}
\end{figure}

\subsection{Octic fields}

\begin{figure}[h]
    \centering
    \begin{subfigure}[t]{0.23\textwidth}
        \centering
        \scalebox{0.25}{
        \ConfusionMatrix{3}{$C_8$,$C_4 \times C_2$, $C_2^3$}{
          1/1/1188.80, 2/1/52.40,   3/1/0.00,
          1/2/48.60,   2/2/3860.60, 3/2/0.80,
          1/3/0.00,    2/3/0.00,    3/3/2153.40
        }
        }
        \caption{DT, ZC}
        \label{fig:octic_ab_dt_zc}
    \end{subfigure}
    \hfill
    \begin{subfigure}[t]{0.23\textwidth}
        \centering
        \scalebox{0.25}{
        \ConfusionMatrix{3}{$C_8$,$C_4 \times C_2$, $C_2^3$}{
          1/1/1028.40, 2/1/205.00,  3/1/7.80,
          1/2/205.00,  2/2/3095.00, 3/2/610.00,
          1/3/5.20,    2/3/631.60,  3/3/1516.60
        }
        }
        \caption{DT, PC}
        \label{fig:octic_nab_dt_pc}
    \end{subfigure}
    \hfill
    \begin{subfigure}[t]{0.23\textwidth}
        \centering
        \scalebox{0.25}{
        \ConfusionMatrix{3}{$C_8$,$C_4 \times C_2$, $C_2^3$}{
          1/1/980.40,  2/1/258.20,  3/1/2.60,
          1/2/168.00,  2/2/3690.40, 3/2/51.60,
          1/3/1.00,    2/3/38.60,   3/3/2113.80
        }
        }
        \caption{LR, normalized ZC}
        \label{fig:octic_nab_lr_zc_normalized}
    \end{subfigure}
    \hfill
    \begin{subfigure}[t]{0.23\textwidth}
        \centering
        \scalebox{0.25}{
        \ConfusionMatrix{3}{$C_8$,$C_4 \times C_2$, $C_2^3$}{
          1/1/431.20,  2/1/791.60,  3/1/18.40,
          1/2/43.40,   2/2/3451.00, 3/2/415.60,
          1/3/0.00,    2/3/1296.00, 3/3/857.40
        }
        }
        \caption{LR, normalized PC}
        \label{fig:octic_nab_lr_pc_normalized}
    \end{subfigure}
    \caption{Confusion matrices for the nonabelian octic Galois field experiments with (a) decision tree, using all ($n \le 1000$) zeta coefficients, (b) decision tree, using polynomial coefficients, (c) logistic regression, using all ($n \le 1000$) normalized zeta coefficients, (d) logistic regression, using normalized polynomial coefficients.}
    \label{fig:galois_8_ab_cm}
\end{figure}

\begin{figure}[h]
    \centering
    \begin{subfigure}[t]{0.23\textwidth}
        \centering
        \scalebox{0.25}{
        \BinaryConfusionMatrix
          {$D_4$}{$Q_8$}
          {5447.80}{111.40}
          {96.80}{9954.60}
        }
        \caption{DT, ZC}
        \label{fig:octic_nab_dt_zc}
    \end{subfigure}
    \hfill
    \begin{subfigure}[t]{0.23\textwidth}
        \centering
        \scalebox{0.25}{
        \BinaryConfusionMatrix
          {$D_4$}{$Q_8$}
          {5460.60}{98.60}
          {95.20}{9956.20}
        }
        \caption{DT, PC}
        \label{fig:octic_nab_dt_pc}
    \end{subfigure}
    \hfill
    \begin{subfigure}[t]{0.23\textwidth}
        \centering
        \scalebox{0.25}{
        \BinaryConfusionMatrix
          {$D_4$}{$Q_8$}
          {5327.80}{231.40}
          {153.00}{9898.40}
        }
        \caption{LR, normalized ZC}
        \label{fig:octic_nab_lr_zc_normalized}
    \end{subfigure}
    \hfill
    \begin{subfigure}[t]{0.23\textwidth}
        \centering
        \scalebox{0.25}{
        \BinaryConfusionMatrix
          {$D_4$}{$Q_8$}
          {5545.80}{13.40}
          {2432.40}{7619.00}
        }
        \caption{LR, normalized PC}
        \label{fig:octic_nab_lr_pc_normalized}
    \end{subfigure}
    \caption{Confusion matrices for the nonabelian octic Galois field experiments with (a) decision tree, using all ($n \le 1000$) zeta coefficients, (b) decision tree, using polynomial coefficients, (c) logistic regression, using all ($n \le 1000$) normalized zeta coefficients, (d) logistic regression, using normalized polynomial coefficients.}
    \label{fig:galois_8_nab_cm}
\end{figure}

\begin{figure}[H]
    \centering
    \begin{subfigure}[t]{0.23\textwidth}
        \centering
        \scalebox{0.25}{
        \BinaryConfusionMatrix
          {ab}{nab}
          {6934.00}{370.60}
          {413.80}{15196.80}
        }
        \caption{DT, ZC}
        \label{fig:octic_ab_nab_dt_zc}
    \end{subfigure}
    \hfill
    \begin{subfigure}[t]{0.23\textwidth}
        \centering
        \scalebox{0.25}{
        \BinaryConfusionMatrix
          {ab}{nab}
          {5546.80}{1757.80}
          {1889.80}{13720.80}
        }
        \caption{DT, PC}
        \label{fig:octic_ab_nab_dt_pc}
    \end{subfigure}
    \hfill
    \begin{subfigure}[t]{0.23\textwidth}
        \centering
        \scalebox{0.25}{
        \BinaryConfusionMatrix
          {ab}{nab}
          {3642.80}{3661.80}
          {1876.20}{13734.40}
        }
        \caption{LR, normalized ZC}
        \label{fig:octic_ab_nab_lr_zc_normalized}
    \end{subfigure}
    \hfill
    \begin{subfigure}[t]{0.23\textwidth}
        \centering
        \scalebox{0.25}{
        \BinaryConfusionMatrix
          {ab}{nab}
          {22.40}{7282.20}
          {45.20}{15565.40}
        }
        \caption{LR, normalized PC}
        \label{fig:octic_ab_nab_lr_pc_normalized}
    \end{subfigure}
    \caption{Confusion matrices for the abelian and nonabelian octic Galois field experiments with (a) decision tree, using all ($n \le 1000$) zeta coefficients, (b) decision tree, using polynomial coefficients, (c) logistic regression, using all ($n \le 1000$) normalized zeta coefficients, (d) logistic regression, using normalized polynomial coefficients.}
    \label{fig:galois_8_ab_nab_cm}
\end{figure}


\newpage
\bibliographystyle{acm} 
\bibliography{refs} 

@article{he2022machine,
  title={Machine-learning number fields},
  author={He, Yang-Hui and Lee, Kyu-Hwan and Oliver, Thomas},
  journal={Mathematics, Computation and Geometry of Data},
  volume={2},
  number={1},
  pages={49--66},
  year={2022},
  publisher={International Press of Boston}
}

@misc{lmfdb,
  shorthand    = {LMFDB},
  author       = {The {LMFDB Collaboration}},
  title        = {The {L}-functions and modular forms database},
  howpublished = {\url{https://www.lmfdb.org}},
  year         = {2025},
  note         = {[Online; accessed 19 July 2025]},
}

@article{amir2023machine,
  title={Machine learning class numbers of real quadratic fields},
  author={Amir, Malik and He, Yang-Hui and Lee, Kyu-Hwan and Oliver, Thomas and Sultanow, Eldar},
  journal={International Journal of Data Science in the Mathematical Sciences},
  volume={1},
  number={02},
  pages={107--134},
  year={2023},
  publisher={World Scientific}
}

@book{dummit2004abstract,
  title={Abstract algebra},
  author={Dummit, David Steven and Foote, Richard M and others},
  volume={3},
  year={2004},
  publisher={Wiley Hoboken}
}

@book{cohen2007number,
  title={{Number theory: Volume II: Analytic and modern tools}},
  author={Cohen, Henri},
  volume={240},
  year={2007},
  publisher={Springer Science \& Business Media}
}

@article{chen2023galois,
  title={Galois groups of certain even octic polynomials},
  author={Chen, Malcolm Hoong Wai and Chin, Angelina Yan Mui and Tan, Ta Sheng},
  journal={Journal of Algebra and its Applications},
  volume={22},
  number={12},
  pages={2350263},
  year={2023},
  publisher={World Scientific}
}

@article{awtrey2024galois,
  title={{On the Galois group of a reciprocal even octic polynomial}},
  author={Awtrey, Chad and Patane, Frank},
  journal={Communications in Algebra},
  volume={52},
  number={7},
  pages={3018--3026},
  year={2024},
  publisher={Taylor \& Francis}
}

@inproceedings{newton2012explicit,
  title={Explicit local reciprocity for tame extensions},
  author={Newton, Rachel},
  booktitle={Mathematical Proceedings of the Cambridge Philosophical Society},
  volume={152},
  number={3},
  pages={425--454},
  year={2012},
  organization={Cambridge University Press}
}

@article{cavallo2019elementary,
  title={{An elementary computation of the Galois groups of symmetric sextic trinomials}},
  author={Cavallo, Alberto},
  journal={arXiv preprint arXiv:1902.00965},
  year={2019}
}

@article{awtrey2015irreducible,
  title={{Irreducible sextic polynomials and their absolute resolvents}},
  author={Awtrey, Chad and French, Robin and Jakes, Peter and Russell, Alan},
  journal={Minnesota Journal of Undergraduate Mathematics},
  volume={1},
  number={1},
  year={2015}
}

@article{wang1950grunwald,
  title={{On Grunwald's theorem}},
  author={Wang, Shianghaw},
  journal={Ann. Math.},
  volume={51},
  number={2},
  pages={471--484},
  year={1950},
  publisher={JSTOR}
}

@article{wang1948counter,
  title={{A counter-example to Grunwald's theorem}},
  author={Wang, Shianghaw},
  journal={Ann. Math.},
  volume={49},
  number={4},
  pages={1008--1009},
  year={1948},
  publisher={JSTOR}
}

@article{roberts2006database,
  title={A database of local fields},
  author={Jones, John W and Roberts, David P},
  journal={Journal of symbolic computation},
  volume={41},
  number={1},
  pages={80--97},
  year={2006},
  publisher={Academic Press USA}
}

@article{carrillo2024finding,
  title={{Finding Galois splitting models to compute local invariants}},
  author={Carrillo, Benjamin},
  journal={Journal of Number Theory},
  volume={261},
  pages={241--251},
  year={2024},
  publisher={Elsevier}
}

@article{jones2008octic,
  title={Octic 2-adic fields},
  author={Jones, John W and Roberts, David P},
  journal={Journal of Number Theory},
  volume={128},
  number={6},
  pages={1410--1429},
  year={2008},
  publisher={Elsevier}
}

@inproceedings{jones2004nonic,
  title={Nonic 3-adic fields},
  author={Jones, John W and Roberts, David P},
  booktitle={International Algorithmic Number Theory Symposium},
  pages={293--308},
  year={2004},
  organization={Springer}
}

@article{jones2014database,
  title={A database of number fields},
  author={Jones, John W and Roberts, David P},
  journal={LMS Journal of Computation and Mathematics},
  volume={17},
  number={1},
  pages={595--618},
  year={2014},
  publisher={London Mathematical Society}
}

@article{Krefl2017,
  title={Machine learning of {C}alabi-{Y}au volumes},
  author={Daniel Krefl and Rak-Kyeong Seong},
  journal={Physical Review D},
  volume={96},
  number={6},
  pages={066014},
  year={2017},
  publisher={American Physical Society},
  doi={10.1103/PhysRevD.96.066014},
  url={https://link.aps.org/doi/10.1103/PhysRevD.96.066014}
}

@article{Ruehle2017,
  title={Evolving neural networks with genetic algorithms to study the String Landscape},
  author={Fabian Ruehle},
  journal={Journal of High Energy Physics},
  volume={2017},
  number={8},
  pages={1--20},
  year={2017},
  publisher={Springer},
  doi={10.1007/JHEP08(2017)038},
  url={https://link.springer.com/article/10.1007/JHEP08(2017)038}
}

@article{Carifio2017,
  title={Machine learning in the string landscape},
  author={Carifio, Jonathan and Halverson, James and Krioukov, Dmitri and Nelson, Brent D.},
  journal={Journal of High Energy Physics},
  volume={2017},
  number={9},
  pages={157},
  year={2017},
  publisher={Springer},
  doi={10.1007/JHEP09(2017)157},
  url={https://link.springer.com/article/10.1007/JHEP09(2017)157}
}

@article{HE2017564,
title = {Machine-learning the string landscape},
journal = {Physics Letters B},
volume = {774},
pages = {564-568},
year = {2017},
issn = {0370-2693},
doi = {https://doi.org/10.1016/j.physletb.2017.10.024},
url = {https://www.sciencedirect.com/science/article/pii/S0370269317308365},
author = {Yang-Hui He}
}

@article{davies2021advancing,
  title={Advancing mathematics by guiding human intuition with {AI}},
  author={Davies, Alex and Veli{\v{c}}kovi{\'c}, Petar and Buesing, Lars and Blackwell, Sam and Zheng, Daniel and Toma{\v{s}}ev, Nenad and Tanburn, Richard and Battaglia, Peter and Blundell, Charles and Juh{\'a}sz, Andr{\'a}s and others},
  journal={Nature},
  volume={600},
  number={7887},
  pages={70--74},
  year={2021},
  publisher={Nature Publishing Group}
}

@article {HLOa,
    AUTHOR = {He, Yang-Hui and Lee, Kyu-Hwan and Oliver, Thomas},
     TITLE = {Machine-learning the {S}ato-{T}ate conjecture},
   JOURNAL = {J. Symbolic Comput.},
  FJOURNAL = {Journal of Symbolic Computation},
    VOLUME = {111},
      YEAR = {2022},
     PAGES = {61--72},
      ISSN = {0747-7171,1095-855X},
   MRCLASS = {11G30 (14Q05 68T05)},
  MRNUMBER = {4352610},
       DOI = {10.1016/j.jsc.2021.11.002},
       URL = {https://doi.org/10.1016/j.jsc.2021.11.002},
}

@article {HLOc,
    AUTHOR = {He, Yang-Hui and Lee, Kyu-Hwan and Oliver, Thomas},
     TITLE = {Machine learning invariants of arithmetic curves},
   JOURNAL = {J. Symbolic Comput.},
  FJOURNAL = {Journal of Symbolic Computation},
    VOLUME = {115},
      YEAR = {2023},
     PAGES = {478--491},
      ISSN = {0747-7171,1095-855X},
   MRCLASS = {14Q05 (11G35 68T05)},
  MRNUMBER = {4477963},
MRREVIEWER = {Elena\ Angelini},
       DOI = {10.1016/j.jsc.2022.08.017},
       URL = {https://doi.org/10.1016/j.jsc.2022.08.017},
}

@article{HLOP,
     AUTHOR = {Yang-Hui He and Kyu-Hwan Lee and Thomas Oliver and Alexey Pozdnyakov},
      TITLE = {Murmurations of Elliptic Curves},
    JOURNAL = {Experimental Mathematics},
      PAGES = {1--13},
       YEAR = {2024},
  PUBLISHER = {Taylor \& Francis},
        DOI = {10.1080/10586458.2024.2382361},
}

@article{Kr,
  title={Data-scientific study of {K}ronecker coefficients},
  author={Kyu-Hwan Lee},
    JOURNAL = {Experimental Mathematics},
  year={2025},
        DOI = {10.1080/10586458.2025.2490576},
}

@article{quanta,
      title={Elliptic curve `murmurations' found with {AI} take flight}, 
      author={Lyndie Chiou},
      year={2024},
        journal="Quanta Magazine",
        url={https://www.quantamagazine.org/elliptic-curve-murmurations-found-with-ai-take-flight-20240305/}
}

@article{douglaslee_mds,
  author    = {Michael R. Douglas and Kyu-Hwan Lee},
  title     = {Mathematical Data Science},
  journal   = {Advances in Theoretical and Mathematical Physics},
  note      = {To appear},
  year      = {2025},
  eprint    = {2502.08620},
  archivePrefix = {arXiv},
  primaryClass = {math.NT},
  url       = {https://arxiv.org/abs/2502.08620}
}

@article{iwasawa1955galois,
  title={{On Galois groups of local fields}},
  author={Iwasawa, Kenkichi},
  journal={Transactions of the American Mathematical Society},
  volume={80},
  number={2},
  pages={448--469},
  year={1955},
  publisher={JSTOR}
}

@article{kluners2001database,
  title={A database for field extensions of the rationals},
  author={Kl{\"u}ners, J{\"u}rgen and Malle, Gunter},
  journal={LMS Journal of Computation and Mathematics},
  volume={4},
  pages={182--196},
  year={2001},
  publisher={Cambridge University Press}
}

@inproceedings{voight2008enumeration,
  title={Enumeration of totally real number fields of bounded root discriminant},
  author={Voight, John},
  booktitle={Algorithmic Number Theory: 8th International Symposium, ANTS-VIII Banff, Canada, May 17-22, 2008 Proceedings 8},
  pages={268--281},
  year={2008},
  organization={Springer}
}

@article{buchmann1993enumeration,
  title={Enumeration of quartic fields of small discriminant},
  author={Buchmann, Johannes and Ford, David and Pohst, Michael},
  journal={mathematics of computation},
  volume={61},
  number={204},
  pages={873--879},
  year={1993}
}

@article {oliver1992computation,
    AUTHOR = {Olivier, M.},
     TITLE = {The computation of sextic fields with a cubic subfield and no
              quadratic subfield},
   JOURNAL = {Math. Comp.},
  FJOURNAL = {Mathematics of Computation},
    VOLUME = {58},
      YEAR = {1992},
    NUMBER = {197},
     PAGES = {419--432},
      ISSN = {0025-5718,1088-6842},
   MRCLASS = {11R21 (11R32 11Y40)},
  MRNUMBER = {1106977},
MRREVIEWER = {Enric\ Nart},
       DOI = {10.2307/2153044},
       URL = {https://doi.org/10.2307/2153044},
}

@article {battistoni2020small,
    AUTHOR = {Battistoni, Francesco},
     TITLE = {On small discriminants of number fields of degree 8 and 9},
   JOURNAL = {J. Th\'eor. Nombres Bordeaux},
  FJOURNAL = {Journal de Th\'eorie des Nombres de Bordeaux},
    VOLUME = {32},
      YEAR = {2020},
    NUMBER = {2},
     PAGES = {489--501},
      ISSN = {1246-7405,2118-8572},
   MRCLASS = {11R21 (11R29 11Y40)},
  MRNUMBER = {4174280},
MRREVIEWER = {Ken\ Yamamura},
       DOI = {10.5802/jtnb.1131},
       URL = {https://doi.org/10.5802/jtnb.1131},
}

@article{pauli2001computation,
  title={On the computation of all extensions of a $p$-adic field of a given degree},
  author={Pauli, Sebastian and Roblot, Xavier-Fran{\c{c}}ois},
  journal={Mathematics of Computation},
  volume={70},
  number={236},
  pages={1641--1659},
  year={2001}
}

@article{fieker2014computation,
  title={{Computation of Galois groups of rational polynomials}},
  author={Fieker, Claus and Kl{\"u}ners, J{\"u}rgen},
  journal={LMS Journal of Computation and Mathematics},
  volume={17},
  number={1},
  pages={141--158},
  year={2014},
  publisher={London Mathematical Society}
}

@article{greenberg1974elementary,
  title={An elementary proof of the Kronecker-Weber theorem},
  author={Greenberg, MJ},
  journal={The American Mathematical Monthly},
  volume={81},
  number={6},
  pages={601--607},
  year={1974},
  publisher={Taylor \& Francis}
}

@article {magma,
    AUTHOR = {Bosma, Wieb and Cannon, John and Playoust, Catherine},
     TITLE = {The {M}agma algebra system. {I}. {T}he user language},
      NOTE = {Computational algebra and number theory (London, 1993)},
   JOURNAL = {J. Symbolic Comput.},
  FJOURNAL = {Journal of Symbolic Computation},
    VOLUME = {24},
      YEAR = {1997},
    NUMBER = {3-4},
     PAGES = {235--265},
      ISSN = {0747-7171},
   MRCLASS = {68Q40},
  MRNUMBER = {MR1484478},
       DOI = {10.1006/jsco.1996.0125},
       URL = {http://dx.doi.org/10.1006/jsco.1996.0125},
}

@article{pedregosa2011scikit,
  title={{Scikit-learn: Machine learning in Python}},
  author={Pedregosa, Fabian and Varoquaux, Ga{\"e}l and Gramfort, Alexandre and Michel, Vincent and Thirion, Bertrand and Grisel, Olivier and Blondel, Mathieu and Prettenhofer, Peter and Weiss, Ron and Dubourg, Vincent and others},
  journal={the Journal of machine Learning research},
  volume={12},
  pages={2825--2830},
  year={2011},
  publisher={JMLR. org}
}

@misc{sklearn-dt,
  title = {{Scikit-Learn, DecisionTreeClassifier}},
  howpublished = {\url{https://scikit-learn.org/stable/modules/generated/sklearn.tree.DecisionTreeClassifier.html}},
  note = {Accessed: 2026-03-10}
}

@misc{sklearn-dt2,
  title = {{Scikit-Learn, Decision Trees}},
  howpublished = {\url{https://scikit-learn.org/stable/modules/tree.html}},
  note = {Accessed: 2026-03-10}
}

@misc{sklearn-lr,
  title = {{Scikit-Learn, LogisticRegression}},
  howpublished = {\url{https://scikit-learn.org/stable/modules/generated/sklearn.linear\_model.LogisticRegression.html}},
  note = {Accessed: 2026-03-10}
}

@misc{lmfdb-completeness,
  title = {{LMFDB, Completeness of $p$-adic field database}},
  howpublished = {\url{https://www.lmfdb.org/padicField/Completeness}},
  note = {Accessed: 2026-03-10}
}

@misc{lmfdb-nf-source,
  title = {{LMFDB, Source and Acknowledgements for number field database}},
  howpublished = {\url{https://www.lmfdb.org/NumberField/Source}},
  note = {Accessed: 2026-03-10}
}

@misc{lmfdb-nf-poly,
  title = {{LMFDB, Canonical defining polynomial for number fields}},
  howpublished = {\url{https://www.lmfdb.org/knowledge/show/nf.polredabs}},
  note = {Accessed: 2026-03-10}
}

@manual{sagemath,
    label        = {Sag95},
    author       = {{The Sage Developers}},
    title        = {{S}age{M}ath, the {S}age {M}athematics {S}oftware {S}ystem},
    url          = {https://www.sagemath.org},
    version      = {9.5},
    year         = {2022},
    note         = {DOI 10.5281/zenodo.6259615},
}


\medskip

{Department of Mathematics, University of Connecticut, Storrs, CT 06269, USA  \hfill \break \indent Korea Institute for Advanced Study, Seoul 02455, Republic of Korea}

\href{mailto:khlee@math.uconn.edu}{khlee@math.uconn.edu}

\medskip

{Department of Mathematics, University of California---Berkeley, Berkeley, CA 94720, USA}

\href{mailto:seewoo5@berkeley.edu}{seewoo5@berkeley.edu}

\end{document}